\documentclass[11pt,a4paper]{amsart}
\usepackage[english]{babel}
\usepackage[T1]{fontenc}
\usepackage[latin1]{inputenc}
\usepackage{graphicx}
\usepackage{amsmath,amssymb,amsthm,mathrsfs,,txfonts}
\usepackage{soul}
\usepackage{array, multirow}
\usepackage{stmaryrd}
\usepackage{mathrsfs}
\usepackage{color}
\usepackage{hyperref}
\usepackage{enumerate}
\usepackage[all]{xy}

\theoremstyle{plain}
\newtheorem{theo}{Theorem}[section]
\newtheorem{lemme}[theo]{Lemma}
\newtheorem{prop}[theo]{Proposition}
\newtheorem{coro}[theo]{Corollary}

\theoremstyle{definition}
\newtheorem{defn}[theo]{Definition}

\newtheorem{nota}[theo]{Notation}

\newtheorem{conj}[theo]{Conjecture}
\newtheorem{exem}[theo]{Example}

\theoremstyle{remark}
\newtheorem{rema}[theo]{Remark}

\def\A{{\rm A}}
\def\B{{\rm B}}
\def\C{{\rm C}}
\def\D{{\rm D}}
\def\E{{\rm E}}
\def\F{{\rm F}}
\def\G{{\rm G}}
\def\H{{\rm H}}

\def\J{{\rm J}}

\def\L{{\rm L}}
\def\M{{\rm M}}

\def\O{{\rm O}}
\def\P{{\rm P}}

\def\R{{\rm R}}
\def\S{{\rm S}}
\def\T{{\rm T}}
\def\U{{\rm U}}
\def\V{{\rm V}}
\def\W{{\rm W}}
\def\X{{\rm X}}
\def\Y{{\rm Y}}
\def\Z{{\rm Z}}
\def\Id{{\rm Id}}
\def\Im{{\rm Im}}
\def\Re{{\rm Re}}
\def\st{{\rm st}}
\def\GL{{\rm GL}}
\def\Sp{{\rm Sp}}
\def\diag{{\rm diag}}
\def\Ad{{\rm Ad}}
\def\Mat{{\rm Mat}}
\def\End{{\rm End}}
\def\Hom{{\rm Hom}}
\def\Rad{{\rm Rad}}
\def\dim{{\rm dim}}
\def\Pin{{\rm Pin}}
\def\gr{{\rm gr}}
\def\Cliff{{\rm Cliff}}
\def\ad{{\rm ad}}
\def\exp{{\rm exp}}
\def\triv{{\rm triv}}
\def\conj{{\rm conj}}
\def\irr{{\rm irr}}

\def\R{\mathbb{R}}
\def\tilde{\widetilde}

\title{Classification and double commutant property for dual pairs in an orthosymplectic Lie supergroup}

\author[1]{Allan Merino}
\address[1]{ Department of Mathematics and Statistics \\ University of Ottawa \\ STEM Complex \\ 150 Louis-Pasteur PvT \\ Ottawa \\ ON \\ K1N6N5 \\ Canada}
\email[1]{amerino@uottawa.ca}
\author[2]{Hadi Salmasian}
\address[2]{ Department of Mathematics and Statistics \\ University of Ottawa \\ STEM Complex \\ 150 Louis-Pasteur PvT \\ Ottawa \\ ON \\ K1N6N5 \\ Canada}
\email[2]{hadi.salmasian@uottawa.ca}

\begin{document}

\begin{abstract}
    

We obtain a full classification of reductive dual pairs in the (real or complex) Lie superalgebra $\mathfrak{spo}(\E)$ and the Lie supergroup $\textbf{SpO}(\E)$. Using this classification we
prove that for a reductive dual pair $(\mathscr{G}\,, \mathscr{G}') = ((\G\,, \mathfrak{g})\,, (\G'\,, \mathfrak{g}'))$ in $\textbf{SpO}(\E)$, the superalgebra $\textbf{WC}(\E)^{\mathscr{G}}$  consisting of $\mathscr{G}$-invariant elements in the Weyl-Clifford algebra $\textbf{WC}(\E)$,  equipped 
with the natural action of the orthosymplectic Lie supergroup $\textbf{SpO}(\E)$,
is generated by the Lie superalgebra $\mathfrak{g}'$.
 As an application, we prove that Howe duality holds for the dual pairs $(\normalfont{\textbf{SpO}}(2n|1)\,, \normalfont{\textbf{OSp}}(2k|2l)) \subseteq \normalfont{\textbf{SpO}}(\mathbb{C}^{2k|2l} \otimes \mathbb{C}^{2n|1})$.   
\end{abstract}

\maketitle

\tableofcontents

\section{Introduction}

A dual pair in a group $\S$ is a pair of subgroups $(\G\,, \G')$ with the property that $\G$ and $\G'$ are mutual centralizers in $\S$. Dual pairs
in the symplectic group $\Sp(\W)$ over a local  field
play an important role in the study of representations of classical groups. Roger E. Howe pioneered the classification of (reductive) dual pairs in the symplectic group (or its double cover) and developed the machinery of Theta Correspondence (see \cite{HOWET}) which is one of the central themes in unitary representation theory and automorphic forms. 

Let $\X$ be an $n$-dimensional complex vector space. We equip $\W = \X \oplus \X^{*}$ with the  symplectic form $\langle \cdot\,, \cdot\rangle$ given by
\begin{equation*}
\langle x_{1} + x^{*}_{1}\,, x_{2} + x^{*}_{2}\rangle = x^{*}_{1}(x_{2}) - x^{*}_{2}(x_{1})\,, \qquad (x_{1}\,, x_{2} \in \X\,, x^{*}_{1}\,, x^{*}_{2} \in \X^{*})\,.
\end{equation*}
For every $x \in \X$ (resp. $x^{*} \in \X^{*}$), we denote by $\M_{x}$ (resp. $\D_{x^{*}}$) the associated multiplication operator (resp. derivation operator) on the symmetric algebra $\S(\X)$. The \emph{Weyl algebra} on $\W$ is the subalgebra $\mathscr{W}(\W)$ of $\End(\S(\X))$ generated by the operators $\M_{x}\,, \D_{x^{*}}\,, x \in \X\,, x^{*} \in \X^{*}$. 
Set 
$\mathbb{Z}_{+}:=\{x\in\mathbb{Z}\,:\,x\geq 0\}$ and let  $\left\{\mathscr{W}(\W)_{k}\right\}_{k\in\mathbb Z_+}$
denote the natural filtration of 
 $\mathscr{W}(\W)$ induced by setting $\deg x=\deg x^*=1$. Then 
\begin{equation*}
\mathscr{W}(\W)_{k} \cdot \mathscr{W}(\W)_{l} \subseteq \mathscr{W}(\W)_{k+l}\,, \qquad \left[\mathscr{W}(\W)_{k}\,, \mathscr{W}(\W)_{l}\right] \subseteq \mathscr{W}(\W)_{k+l-2}\,, \qquad (k\,, l \in \mathbb{Z}_{+})\,,
\end{equation*}
where $\left[\cdot\,, \cdot\right]$ is the standard commutator on $\mathscr{W}(\W)$. 
Let $\iota: \W \to \mathscr{W}(\W)$ be the natural inclusion.
Then the subspace $\Omega$ of $\mathscr{W}(\W)_{2}$ defined by \[
\Omega: = \mathrm{Span}_\mathbb C\left\{\iota(w_{1})\iota(w_{2}) + \iota(w_{2})\iota(w_{1})\,, w_{1}\,, w_{2} \in \W\right\},
\]
is a Lie algebra isomorphic to $\mathfrak{sp}(\W) := \mathfrak{sp}(\W\,, \langle\cdot\,, \cdot\rangle)$.  As $\Omega \subseteq \mathscr{W}(\W) \subseteq \End(\S(\X)) \curvearrowright \S(\X)$, we get an action of $\mathfrak{sp}(\W)$ on $\S(\X)$: this is the Fock model of the Weil representation. Moreover, the action of $\mathfrak{sp}(\W)$ on $\mathscr{W}(\W)$ (coming from $\left[\Omega\,, \mathscr{W}(\W)\right] \subseteq \mathscr{W}(\W)$) can be exponentiated to an action of $\Sp(\W)$ on $\mathscr{W}(\W)$ by automorphisms.
In his representation theoretic approach to classical invariant theory, Howe obtained the following result.

\begin{theo}{(\cite[Theorem~7]{HOWE89})}

Let $(\G\,, \G')$ be an irreducible reductive dual pair in $\Sp(\W)$. Then the associative algebra $\mathscr{W}(\W)^{\G}$ is generated  by $\mathfrak{g}'$.

\label{DoubleCommutantSymplecticComplex}

\end{theo}

The interest in Theorem~\ref{DoubleCommutantSymplecticComplex} lies in its connection with the Theta Correspondence for complex reductive groups (See Remark~\ref{rema:theta}).
Theorem \ref{DoubleCommutantSymplecticComplex} has an analogue in the real case as well. Let $\left\{x_{1}\,, \ldots\,, x_{n}\right\}$ be a basis of $\X$ and let $\{x^{*}_{1}\,, \ldots\,, x^{*}_{n}\}$ be the dual basis of $\X^*$. The map $\tau: \mathscr{W}(\W) \to \mathscr{W}(\W)$ given by
\begin{equation*}
\tau(x_{i}) = x^{*}_{i}\,, \quad \tau(x^{*}_{i}) = x_{i}\,, \quad \tau(\lambda x) = \bar{\lambda}\tau(x)\,, \quad \tau(xy) = \tau(y)\tau(x)\,, \quad (1 \leq i \leq n\,, x\,, y \in \mathscr{W}(\W)\,, \lambda \in \mathbb{C})
\end{equation*}
is a conjugate-linear involution on $\mathscr{W}(\W)$. The $\pm 1$-eigenspaces of $\tau$ (over $\mathbb R$) result in a direct sum decomposition  
\[\mathscr{W}(\W) = \mathscr{W}(\W)^{+} \oplus \mathscr{W}(\W)^{-}.\]
Let $\Omega_{\mathbb{R}} := \Omega \cap \mathscr{W}(\W)^{-}$ and $\W_{\mathbb{R}} := \W \cap \mathscr{W}(\W)^{-}.$ By \cite[Section~3]{HOWEPHYSICS}, we have $\ad(\Omega_{\mathbb{R}})(\W_{\mathbb{R}}) \subseteq \W_{\mathbb{R}}$ and the elements $\exp(\ad(u))$ for $u \in \Omega_{\mathbb{R}}$, generate the real symplectic group $\Sp(\W_{\mathbb{R}}) \subseteq \Sp(\W)$.

\begin{theo}{(\cite[Theorem~4.1]{HOWEPHYSICS})}

Let $(\G\,, \G')$ be an irreducible reductive dual pair in $\Sp(\W_{\mathbb{R}})$. Then the associative algebra $\mathscr{W}(\W)^{\G}$ is generated  by $\mathfrak{g}'$.

\label{DoubleCommutantSymplecticReal}

\end{theo}

The two theorems above also extend to the orthogonal group. Let $\E := \Y \oplus \Y^{*}$ be a complex vector space endowed with a bilinear form $\B$ given by
\begin{equation*}
\B(y_{1} + y^{*}_{1}\,, y_{2}+y^{*}_{2}) = y^{*}_{1}(y_{2}) + y^{*}_{2}(y_{1})\,, \qquad (y_{1}\,, y_{2} \in \Y\,, y^{*}_{1}\,, y^{*}_{2} \in \Y^{*})\,.
\end{equation*}
We denote by $\Cliff(\E\,, \B)$ the corresponding Clifford algebra. 

\begin{theo}{(\cite[Section~4]{ALLANGANGCLEMENT})}

Let $(\G\,, \G')$ be an irreducible reductive dual pair in $\O(\E\,, \B)$. Then the associative algebra $\Cliff(\E\,, \B)^{\G}$ is generated by $\mathfrak{g}'$.

\label{DoubleCommutantOrthogonalComplex}

\end{theo}

An analogous statement holds in the real case. As explained in \cite{ALLANGANGCLEMENT}, Theorem \ref{DoubleCommutantOrthogonalComplex} implies directly the Howe duality for the dual pairs in the Pin group $\Pin(\E\,, \B)$ associated to the orthogonal group $\O(\E\,, \B)$.

\begin{rema}
Consider  an irreducible reductive dual pair 
$(\G\,, \G')$ 
in $\Sp(\W)$. By looking at the classification of irreducible dual pairs in $\Sp(\W)$ we can assume, without loss of generality, that $\G \subseteq \GL(\X)$. Moreover, the action of $\G$ on $\S(\X)$ is semisimple. Let $\G_{\irr}$ be the set of irreducible finite dimensional representations of $\G$ (up to equivalence). Then
\begin{equation*}
\S(\X) = \bigoplus\limits_{(\lambda\,, \V_{\lambda}) \in \G^{\omega}_{\irr}} \V(\lambda) 
\end{equation*}
where $\G^{\omega}_{\irr} := \left\{(\lambda\,, \V_{\lambda}) \in \G_{\irr}\,, \Hom_{\G}(\V_{\lambda}\,, \S(\X)) \neq 0\right\}$ and $\V(\lambda)$ is the $\G$-isotypic component corresponding to $(\lambda\,, \V_{\lambda}) \in \G_{\irr}$. 
As explained in~\cite[Theorem 8]{HOWE89}
the algebra $\mathscr{W}(\W)^{\G}$ naturally acts on $\V(\lambda)$ for every $\lambda$, and  the joint action of $\G$ and $\mathscr{W}(\W)^{\G}$ on every $\V(\lambda)$ is irreducible, i.e.,
\begin{equation*}
\V(\lambda) = \lambda \otimes \theta(\lambda)\,,
\end{equation*}
where $\theta(\lambda)$ is an irreducible $\mathscr{W}(\W)^{\G}$-module. Moreover, if $\lambda \nsim \lambda'$, the corresponding $\mathscr{W}(\W)^{\G}$-modules $\theta(\lambda)$ and $\theta(\lambda')$ are not equivalent. From Theorem \ref{DoubleCommutantSymplecticComplex} it follows that the restriction of $\theta(\lambda)$ to $\mathfrak{g}'$ is irreducible and in particular, the correspondence
\begin{equation}
\G^{\omega}_{\irr} \ni \lambda \to \theta(\lambda) \in \mathfrak{g}'^{\omega}_{\irr}
\label{HoweDuality}
\end{equation}
is one-to-one, where $\mathfrak{g}'^{\omega}_{\irr}$ stands for the set (up to equivalence) of irreducible $\mathfrak{g}'$-modules $(\V_{\lambda'}\,, \lambda')$ such that $\Hom_{\mathfrak{g}'}(\V_{\lambda'}\,, \S(\X)) \neq 0$. The bijection given in Equation \eqref{HoweDuality} is an instance of the well known Theta Correspondence (also called Howe Duality).

\label{rema:theta}
\end{rema}

\section{Main results}

\label{sec:mainres}
Theorems \ref{DoubleCommutantSymplecticComplex}, \ref{DoubleCommutantSymplecticReal} and \ref{DoubleCommutantOrthogonalComplex} are about irreducible reductive dual pairs in $\Sp(\W)$ and $\O(\E\,, \B)$. The complete classification of such pairs is due to Howe 
\cite{HOWEALG} (for further references, see  \cite{HOWEPHYSICS}, 
\cite{HOWE89}, 
\cite{MVW}, \cite{SCHMIDT} and \cite{RUBENTHALER}).
In this paper, we classify reductive dual pairs in an orthosymplectic Lie superalgebra and  extend  Theorems \ref{DoubleCommutantSymplecticComplex}, \ref{DoubleCommutantSymplecticReal} and \ref{DoubleCommutantOrthogonalComplex} to the superized setting.  
We remark that indeed in  \cite{HOWE89}, Howe suggested a way to approach both symplectic and orthogonal cases simultaneously by using $\mathbb{Z}_{2}$-graded structures: Lie superalgebras and/or Lie supergroups. 
However, 
in the dual pairs considered in \cite{HOWE89}, one factor is always purely even.  
Nishiyama also introduced the notion of dual pairs in the Lie superalgebra $\mathfrak{spo}(\E\,, \B)$ in \cite{NISHIYAMA}, but he did not address the classification problem.

From now on we assume that  $\mathbb{Z}_{2}: = \left\{\bar{0}, \bar{1}\right\}$ denotes the group with two elements and $\mathbb{K} := \mathbb{R}$ or $\mathbb{C}$.
Let
$\E = \E_{\bar{0}} \oplus \E_{\bar{1}}$  be a $\mathbb{Z}_{2}$-graded vector space over $\mathbb{K}$.
We assume that $\E$ is equipped with a  non-degenerate, even, $(-1)$-supersymmetric bilinear form 
$\B:\E\times \E\to\mathbb K$.  Here 
being
even means
$\B = \B_{\bar 0} \oplus^{\perp} \B_{\bar1}$ where $\B_{i} = \B_{|_{\E_{{i}} \times \E_{{i}}}}$ for $i = \bar 0\,, \bar 1$ (see also 
Definition~\ref{DefSuperhermitian}), and being
\emph{$\epsilon$-supersymmetric} for $\epsilon\in\{\pm 1\}$ means
\[
\B(u,v)=\epsilon(-1)^{|u|\cdot|v|}\B(v,u),\qquad(u,v\in \E)
\]
where $|u|,|v|\in\mathbb Z_2$ denote the parities of $u$ and $v$. 
Let $\mathfrak{spo}(\E\,, \B)$ denote the  subalgebra of $\mathfrak{gl}(\E)$ defined by
\begin{equation*}
\mathfrak{spo}(\E\,, \B) := \left\{\X \in \mathfrak{gl}(\E)\,, \B(\X(u)\,, v) + (-1)^{\left|\X\right|\cdot\left|u\right|}\B(u\,, \X(v)) = 0\,, u\,, v \in \E\right\}\,.
\end{equation*}

Following the ideas of Howe in the symplectic case,
in Section \ref{SectionPreliminaries} we introduce the notion of reductive and irreducible dual pairs in $\mathfrak{spo}(\E\,, \B)$ and prove  that every reductive dual pair in $\mathfrak{spo}(\E\,, \B)$ is a direct sum of irreducible dual pairs (see Proposition \ref{Proposition1}). We also show  that irreducible reductive dual pairs $(\mathfrak{g}\,, \mathfrak{g}')$ in $\mathfrak{spo}(\E\,, \B)$ can be divided into two different families: Type I (i.e., the joint action of $\mathfrak{g}$ and $\mathfrak{g}'$ on $\E$ is irreducible) and Type II (i.e., the space $\E$ can be decomposed as $\E = \E_{1} \oplus \E_{2}$, where both $\E_{1}$ and $\E_{2}$ are $\mathbb{Z}_{2}$-graded, $\B$-isotropic and irreducible under the joint action of  $\mathfrak{g}$ and $\mathfrak{g}'$).

By definition, a division superalgebra $\mathbb{D}$ over a field $\mathbb{K} $ is  an associative $\mathbb{K}$-superalgebra such that every homogeneous element of $\mathbb{D}$ is invertible. The even part of the center of $\mathbb{D}$ is always a field $\mathbb{L}$ containing $\mathbb{K}$. For all the division superalgebras that appear in this paper we have $\mathbb{K},\mathbb{L}\in\{\mathbb{R},\mathbb{C}\}$. Up to isomorphism, there exist 2 division superalgebras for which $\mathbb{L}=\mathbb{C}$ and 8 division superalgebras for which $\mathbb{L}=\mathbb{R}$  (for more details see Section \ref{ClassificationDivisionSuperalgebra}).

For a left $\mathbb{D}$-module $\U$, we denote by $\mathfrak{gl}_{\mathbb{D}}(\U)$ the $\mathbb{K}$-Lie superalgebra  of $\mathbb{D}$-linear maps  $\X:\U\to\U$, i.e. maps satisfying $\X(\D \cdot u) = (-1)^{\left|\D\right| \cdot \left|\X\right|} \D \cdot \X(u)$ for $\D \in \mathbb{D}\,, u \in \U$. Similarly, for a right $\mathbb{D}$-module $\W$, we denote by $\mathfrak{gl}_{\mathbb{D}}(\W)$ the $\mathbb K$-Lie superalgebra of (right) $\mathbb{D}$-linear maps  $\X:\W\to\W$, i.e. maps satisfying $\X(w \cdot \D) = \X(w) \cdot \D$ for $\D \in \mathbb{D}\,, w \in \W$.
\begin{defn}
By a superinvolution of a $\mathbb K$-algebra 
 $\mathscr A$  we mean an even  $\mathbb{K}$-linear map $\iota:\mathscr A\to\mathscr A$ 
satisfying
\begin{equation}
\iota^{2}(\X) = \X\,, \qquad \iota(\X\Y) = (-1)^{\left|\X\right|\cdot\left|\Y\right|} \iota(\Y)\iota(\X)\,, \qquad (\X\,, \Y \in \mathscr A)\,.
\label{IotaProperties}
\end{equation}
\end{defn}
\begin{defn}
Let $\mathbb{D}$ be a division superalgebra over $\mathbb K$ that is equipped with an involution $\iota$.
Let $\W$ be a right $\mathbb{D}$-module and let $\gamma: \W \times \W \to \mathbb{D}$ be a $\mathbb{K}$-bilinear map.
\begin{enumerate}
\item We call $\gamma$  even (resp., odd) if $\gamma(\W_{\alpha}\,, \W_{\beta}) \subseteq \mathbb{D}_{\alpha+\beta}$ (resp. $\gamma(\W_{\alpha}\,, \W_{\beta}) \subseteq \mathbb{D}_{\alpha+\beta+\bar{1}})$ for $\alpha,\beta\in\{\bar 0,\bar 1\}$. We denote the parity of 
$\gamma$ by $|\gamma|$.
\item Assume that $\gamma$ is homogeneous (i.e., either even or odd). Then $\gamma$ is called $(\iota\,, \epsilon)$-superhermitian for $\epsilon=\pm 1$ if it satisfies the following constraints:
\begin{itemize}
\item $\gamma(w\,,v)=\epsilon(-1)^{\left|v\right|\cdot \left|w\right|}\iota(\gamma(v\,, w))$for homogeneous $v,w\in \W$,
\item $\gamma(v\D_{1}\,, w\D_{2})=(-1)^{\left|\D_{1}\right|\cdot \left|v\right| + \left|\D_{1}\right|\cdot \left|\gamma\right|}\iota(\D_{1})\gamma(v\,, w)\D_{2}$ for homogeneous $v,w\in \W$ and $\D_1,\D_2\in\mathbb{D}$.
\end{itemize}
\end{enumerate}
For a left $\mathbb{D}$-module $\U$, we define an $(\iota, \epsilon)$-superhermiatian form $\gamma:\U\times \U\to\mathbb D$ by similar conditions, except that the second  condition is modified as follows:
\begin{equation*}
\gamma(\D_{1}v\,, \D_{2}w)=(-1)^{\left|\D_{2}\right|\cdot \left|w\right| + \left|\D_{1}\right| \cdot \left|\gamma\right|}\D_{1}\gamma(v\,, w)\iota(\D_{2})\qquad \text{for homogeneous }v\,, w \in \U, \D_{1}, \D_{2} \in \mathbb{D}.
\end{equation*}

\label{DefSuperhermitian}

\end{defn}

Henceforth for a division superalgebra $\mathbb D:=\mathbb D_{\bar 0}\oplus\mathbb D_{\bar 1}$ we define
$
\delta:\mathbb D\to\mathbb D$
by $\delta(\D):=(-1)^{|\D|}\D$.
The complete classification of irreducible dual pairs in $\mathfrak{spo}(\E\,,\B)$  is obtained in Section \ref{SectionClassification}, where we establish the following theorems.

\begin{theo}{(Classification of Irreducible  Dual Pairs - Type I)}
Let $(\mathfrak{g}\,, \mathfrak{g}')$ be an irreducible dual pair in $\mathfrak{spo}(\E\,, \B)$ of Type I. Then there exist
\begin{itemize}
\item $\mathbb{D}$: a division superalgebra over $\mathbb{K}$\,,
\item $\iota$: a superinvolution on $\mathbb{D}$\,,
\item $\U$: left $\mathbb{D}$-module endowed with a homogeneous $(\iota\,, \epsilon)$-superhermitian form $\gamma$ (with $\epsilon \in \left\{\pm 1\right\}$)\,,
\item $\W$: right $\mathbb{D}$-module endowed with a $(\iota \circ \delta\,, -\epsilon)$-superhermitian form $\gamma'$ with $\left|\gamma\right| = \left|\gamma'\right|$
\end{itemize}
such that the following hold:
\begin{itemize}

\item 
$
\E \cong \W \otimes_{\mathbb{D}} \U
$ and 
\[
\B(w_{1} \otimes u_{1}\,, w_{2} \otimes u_{2}) = (-1)^{\left|u_{1}\right|\cdot\left|w_{2}\right| +\left|w_{1}\right|\cdot\left|w_{2}\right|}\Re\big(\gamma'(w_{2}\,, w_{1})\gamma(u_1\,,u_2)\big)
\,,
\qquad\text{($u_{1}\,, u_{2} \in \U\,, w_{1}\,, w_{2} \in \W$).}
\]
\item $(\mathfrak{g}\,, \mathfrak{g}') = (\mathfrak{g}(\U\,, \gamma)\,, \mathfrak{g}(\W\,, \gamma'))$, where 
\begin{equation*}
\mathfrak{g}(\U\,, \gamma): = \left\{\X \in \mathfrak{gl}_{\mathbb{D}}(\U)\,, \gamma(\X(u)\,, v) + (-1)^{\left|\X\right|\cdot\left|u\right|}\gamma(u\,, \X(v)) = 0\,, u\,, v \in \U\right\}\,,
\end{equation*}
and $\mathfrak{g}(\W,\gamma')$ is defined analogously.
\end{itemize}
The explicit classification of the Type I dual pairs in $\mathfrak{spo}(\E,\B)$ is given in 
Table \ref{Table:I} and Table \ref{Table:IC}.
\label{TheoremeIntroduction1}
\end{theo}

\begin{theo}{(Classification of Irreducible Dual Pairs - Type II)}

For an irreducible dual pair $(\mathfrak{g}\,, \mathfrak{g}')$ in $\mathfrak{spo}(\E\,, \B)$ of Type II, there exist 
\begin{itemize}
\item $\mathbb{D}$: a division superalgebra over $\mathbb{K}$\,,
\item $\U$: left $\mathbb{D}$-module\,,
\item $\W$: right $\mathbb{D}$-module
\end{itemize}
such that
$
\E \cong (\W \otimes_{\mathbb{D}} \U) \oplus (\W \otimes_{\mathbb{D}} \U)^{*}
$ where the right hand side is equipped with its canonical $(-1)$-supersymmetric form
and $(\mathfrak{g}\,, \mathfrak{g}') = (\mathfrak{gl}_{\mathbb{D}}(\U)\,, \mathfrak{gl}_{\mathbb{D}}(\W))$.
The explicit classification of Type II dual pairs in $\mathfrak{spo}(\E,\B)$ is given in Table \ref{Table:II} and Table \ref{Table:IIC}.

\label{TheoremeIntroduction2}
\end{theo}

\begin{table}[ht]
\begin{tabular}{|c|c|c|c|c|c|c|}
  \hline
  $\mathbb{D}$ & $\mathbb{\iota}$ & $\left|\gamma\right|\,,\,|\gamma'|$ & $\U$ & $\W$ & $\mathfrak{g}$ & $\mathfrak{g}'$  \\
  \hline
  $\mathbb{R}$ & $\triv$ & $\bar{0}$ & $\mathbb{R}^{2a|b+c}$ & $\mathbb{R}^{p+q|2r}$ & $\mathfrak{spo}(2a|b, c)$ & $\mathfrak{osp}(p, q| 2r)$ \\
  \hline
   $\mathbb{R}$ & $\triv$ & $\bar{1}$ & $\mathbb{R}^{a|a}$ & $\mathbb{R}^{p|p}$ & ${\mathfrak{p}}(a, \mathbb{R})$ & $\mathfrak{p}(p, \mathbb{R})$\\
   \hline
  $\mathbb{C}$ & $\triv$ & $\bar{0}$ & $\mathbb{C}^{2a|b}$ & $\mathbb{C}^{p|2q}$ & $\mathfrak{spo}(2a|b)$ & $\mathfrak{osp}(p| 2q)$ \\
  \hline
  $\mathbb{C}$ & $\triv$ & $\bar{1}$ & $\mathbb{C}^{a|a}$ & $\mathbb{C}^{p|p}$ & ${\mathfrak{p}}(a, \mathbb{C})$ & $\mathfrak{p}(p, \mathbb{C})$ \\
  \hline
  $\mathbb{C}$ & $\conj$ & $\bar{0}$ & $\mathbb{C}^{a+ b|c+ d}$ & $\mathbb{C}^{p+ q|r+ s}$ & $\mathfrak{u}(a, b| c, d)$ & $\mathfrak{u}(p, q|r, s)$ \\
  \hline
   $\mathbb{C}$ & $\conj$ & $\bar{1}$ & $\mathbb{C}^{a|a}$ & $\mathbb{C}^{p|p}$ & $\bar{\mathfrak{p}}(a)$ & $\bar{\mathfrak{p}}(p)$ \\
  \hline
  $\mathbb{H}$ & $\conj$ & $\bar{0}$ & $\mathbb{H}^{a|b+c}$ & $\mathbb{H}^{p+ q| r}$ & $\mathfrak{osp}^{*}(a| b, c)$ & $\mathfrak{spo}^{*}(p, q| r)$ \\
   \hline
   $\mathbb{H}$ & $\conj$ & $\bar{1}$ & $\mathbb{H}^{a|a}$ & $\mathbb{H}^{p|p}$ & ${\mathfrak{p}}^{*}(a)$ & $\mathfrak{p}^{*}(p)$ \\
   \hline
   $\mathrm{Cl}_1(\mathbb C)$ & $\iota_1$ & $\bar{0}$ & $\mathrm{Cl}_1(\mathbb C)^{a+ b}$ & $\mathrm{Cl}_1(\mathbb C)^{p+ q}$ & ${\mathfrak{q}}(a, b)$ & $\mathfrak{q}(p, q)$ \\
   \hline
   \end{tabular}
  
   \vspace{3mm}
   
   \caption{ \label{Table:I} Real  dual pairs of Type I}
\end{table}

\begin{table}[ht]
 
\begin{tabular}{|c|c|c|c|c|c|c|}
  \hline
  $\mathbb{D}$ & $\iota$ &$|\gamma|\,,\,|\gamma'|$ &$\U$ & $\W$ & $\mathfrak{g}$ & $\mathfrak{g}'$ \\
  \hline
   $\mathbb{C}$ & $\triv$ & $\bar 0$& $\mathbb{C}^{2a|b}$ & $\mathbb{C}^{p|2q}$ & $\mathfrak{spo}(2a|b, \mathbb{C})$ & $\mathfrak{osp}(p| 2q, \mathbb{C})$ \\
    \hline
    $\mathbb{C}$ & $\triv$ & $\bar 1$ & $\mathbb{C}^{a|a}$ & $\mathbb{C}^{p|p}$ & ${\mathfrak{p}}(a, \mathbb{C})$ & $\mathfrak{p}(p, \mathbb{C})$ \\
    \hline
\end{tabular}

\vspace{3mm}

\caption{\label{Table:IC} Complex dual pairs of Type I}
\medskip

\end{table}    

\begin{table}[ht]

\begin{tabular}{|c|c|c|c|c|}
\hline
  $\mathbb{D}$ & $\U$ & $\W$ & $\mathfrak{g}$ & $\mathfrak{g}'$  \\
\hline
  $\mathbb{R}$ & $\mathbb{R}^{a|b}$ & $\mathbb{R}^{p|q}$ & $\mathfrak{gl}(a|b, \mathbb{R})$ & $\mathfrak{gl}(p|q, \mathbb{R})$ \\
  \hline
   $\mathbb{C}$ & $\mathbb{C}^{a|b}$ & $\mathbb{C}^{p|q}$ & $\mathfrak{gl}(a|b, \mathbb{C})$ & $\mathfrak{gl}(p|q, \mathbb{C})$ \\
   \hline
    $\mathbb{H}$ & $\mathbb{H}^{a|b}$ & $\mathbb{H}^{p|q}$ & $\mathfrak{gl}(a|b, \mathbb{H})$ & $\mathfrak{gl}(p|q, \mathbb{H})$ \\
    \hline
    $\mathrm{Cl}_1(\mathbb R)$ & $\mathrm{Cl}_1(\mathbb R)^{a}$ & $\mathrm{Cl}_1(\mathbb R)^{p}$ &${\mathfrak{q}}(a, \mathbb{R})$ & $\mathfrak{q}(p, \mathbb{R})$ \\
    \hline
    $\mathrm{Cl}_1(\mathbb C)$ & $\mathrm{Cl}_1(\mathbb C)^{a}$ & $\mathrm{Cl}_1(\mathbb C)^{p}$ & ${\mathfrak{q}}(a, \mathbb{C})$ & $\mathfrak{q}(p, \mathbb{C})$  \\
    \hline
    $\mathrm{Cl}_3(\mathbb R)$ & $\mathrm{Cl}_3(\mathbb R)^{a}$ & $\mathrm{Cl}_3(\mathbb R)^{p}$ & ${\mathfrak{q}}(a, \mathbb{H})$ & $\mathfrak{q}(p, \mathbb{H})$  \\
    \hline
    $\mathrm{Cl}_6(\mathbb R)$ & $\mathrm{Cl}_6(\mathbb R)^{a}_{}$ & $\mathrm{Cl}_6(\mathbb R)^{p}$ & $\bar{\mathfrak{q}}(a)$ & $\bar{\mathfrak{q}}(p)$  \\
    \hline
\end{tabular}

\vspace{3mm}

\caption{\label{Table:II} Real dual pairs of Type II}

\end{table}

\begin{table}[ht]
\begin{tabular}{|c|c|c|c|c|}
\hline
  $\mathbb{D}$ & $\U$ & $\W$ & $\mathfrak{g}$ & $\mathfrak{g}'$ \\
  \hline
    $\mathbb{C}$ & $\mathbb{C}^{a|b}$ & $\mathbb{C}^{p|q}$ & $\mathfrak{gl}(a|b, \mathbb{C})$ & $\mathfrak{gl}(p|q, \mathbb{C})$ \\
    \hline
    $\mathrm{Cl}_1(\mathbb C)$ & $\mathrm{Cl}_1(\mathbb C)^{a}$ & $\mathrm{Cl}_1(\mathbb C)^{p}$ & ${\mathfrak{q}}(a, \mathbb{C})$ & $\mathfrak{q}(p, \mathbb{C})$ \\
   \hline
\end{tabular}

\vspace{3mm}
\caption{\label{Table:IIC}Complex dual pairs of Type II}

\end{table}

Let us explain the notation and conventions used in Tables \ref{Table:I}--\ref{Table:IIC}. 
The 
division superalgebras  
$\mathrm{Cl}_k(\mathbb F)$ for $\mathbb F\in\{\mathbb R,\mathbb C\}$ are defined in Proposition \ref{ClassificationDivisionSuperalgebra}. We use
$\mathrm{Cl}_k(\mathbb F)^a$ to denote a free module of rank $a$ over $\mathrm{Cl}_k(\mathbb F)$. Explicit realizations of the Lie superalgebras that occur as 
$\mathfrak g$ and $\mathfrak{g}'$ 
are given in Appendix~\ref{AppendixExplicitRealization} (see also Table~\ref{Table:LieSup}).  Note that the content of Appendix \ref{AppendixExplicitRealization} directly pertains to 
$\mathfrak{gl}_\mathbb D(\W)$ and its subalgebra $\mathfrak{g}(\W,\gamma')$ for a \emph{right} $\mathbb D$-module $\W$. 
However, realizations of  $\mathfrak{gl}_\mathbb D(\U)$ and $\mathfrak{g}(\U,\gamma)$ for a left $\mathbb D$-module $\U$
can be obtained from 
the fact that $\U$ is canonically a right $\mathbb{D}^\mathrm{sop}$-module, where 
$\mathbb{D}^\mathrm{sop}$ is 
the super-opposite of $\mathbb D$, i.e., the division superalgebra with product \begin{equation}
\label{eq:superopp}
\D\bullet \D':=(-1)^{|\D|\cdot|\D'|}\D'\D.\end{equation} 
For Type II dual pairs  we have $\mathfrak g=\mathfrak{gl}_\mathbb D(\U)$ and $\mathfrak g'=\mathfrak{gl}_\mathbb D(\W)$.
In Tables \ref{Table:I} and \ref{Table:IC}, 
the $\mathbb K$-bilinear forms $\gamma:\U\times \U\to\mathbb D$ are always  $(\iota,-1)$-superhermitian, whereas the forms  
$\gamma':\W\times \W\to\mathbb D$ are always $(\iota\circ \delta,1)$-superhermitian (where $\delta$ is the identity map on $\mathbb D$ if $\mathbb D_{\bar 1}=0$).  
The superinvolution denoted by $\triv$  (resp. $\conj$) is the identity map (resp. complex or quaternionic conjugation) on $\mathbb{D} = \mathbb{D}_{\bar{0}}$, 
while the superinvolution  $\iota_1$ of  $\mathrm{Cl}_1(\mathbb C)$ is defined in Proposition \ref{PropositionDivisionIota}.

\begin{rema}
An important point is that  the real division superalgebras $\mathrm{Cl}_k(\mathbb R)$ for $k\in\{2,5,7\}$ do not appear 
in Table
\ref{Table:II}. The reason is that the latter division superalgebras do not result in new examples of dual pairs. 
This is because we have $\mathrm{Cl}_k(\mathbb R)\cong \mathrm{Cl}_{8-k}(\mathbb R)^\mathrm{sop}$ for $1\leq k\leq 7$,
and in general if we equip $\U$ and $\W$ with their canonical right and left $\mathbb{D}^\mathrm{sop}$-module structures, then the $\mathbb R$-linear  isomorphism  $\W\otimes_{\mathbb D}\U\cong \U\otimes_{\mathbb D^\mathrm{sop}}\W$ identifies the dual pairs $(\mathfrak{gl}_\mathbb D(\W),\mathfrak{gl}_\mathbb D(\U))$ and
$(\mathfrak{gl}_{\mathbb D^\mathrm{sop}}(\U),\mathfrak{gl}_{\mathbb D^\mathrm{sop}}(\W))$ in $\mathfrak{spo}(\E\,,\B)$. Thus, the dual pairs that occur for $k\in\{2,5,7\}$ are already in Table \ref{Table:II}. 

For a similar reason, in Table \ref{Table:I} we do not include a row corresponding to $\mathrm{Cl}_1(\mathbb C)$ such that either $\iota=\iota_2$ or $|\gamma|=\bar 1$. Suppose that  a Type I real dual pair corresponds to a factorization $\W\otimes_\mathbb D \U$ such that $\gamma$ is $(\iota_2,\epsilon)$-superhermitian and consequently $\gamma'$ is $(\iota_1,-\epsilon)$-superhermitian. We can equip $\U$ with a right $\mathbb D$-module structure by setting $u\cdot d:=(-1)^{|d|\cdot |u|}\iota_2(d)u$ for $u\in \U$ and $d\in \mathbb D$. Then on this right $\mathbb D$-module $\gamma$ will be a $(\iota_2,\epsilon)$-superhermitian form. Similarly, using $\iota_1=\iota_2\circ\delta$ we can equip $\W$ with a left $\mathbb D$-module structure and then $\gamma'$ will be $(\iota_1,-\epsilon)$-superhermitian. It follows that the map 
\[
\W\otimes_\mathbb D \U\to \U\otimes_{\mathbb D} \W
\ ,\ 
w\otimes u\mapsto (-1)^{|\gamma|+|u|\cdot |w|}
u\otimes w
\]
 is an isomorphism of real $(-1)$-supersymmetric spaces
that identifies a dual pair $(\mathfrak g(\W,\gamma'),\mathfrak g(\U,\gamma))$ with 
the dual pair
$(\mathfrak g(\U,\gamma),\mathfrak g(\W,\gamma'))$. The latter dual pair corresponds to $\iota_1$. Finally, if $\gamma$ and $\gamma'$ are odd,  then we can replace them by even superhermitian forms $\varepsilon\gamma$ and $\varepsilon\gamma'$ (see Remark~\ref{SuperHermitianForms}).

\end{rema}

Although the statements of Theorems
 \ref{TheoremeIntroduction2} 
and
\ref{TheoremeIntroduction1}
are similar to the well-known classification results in the non-superized setting, and their proofs use many ideas from the seminal works \cite{HOWEALG}, \cite{HOWE89}, and \cite{MVW}, there are a number of substantial differences between the superized and non-superized settings that serve as our motivation for writing this paper. First, we notice that only one division superalgebra $\mathbb D$ satisfying $\mathbb D_{\bar 1}\neq 0$ results in dual pairs of Type I. Second, the correspondence between superhermitian forms on a $\mathbb D$-module $\U$ and superinvolutions on $\End_\mathbb D(\U)$ works slightly differently in the two cases
$\mathbb D_{\bar 1}=0$
and $\mathbb D_{\bar 1}\neq 0$: in the former case the parity of the superhermitian form is determined by the superinvolution, whereas in the latter case one can associate two forms of opposite parity to the same superinvolution (see Propositions \ref{PropositionC4} and \ref{PropositionC5}).  
Third, the proofs of classification in the non-super context (see for example \cite{HOWEALG} or \cite{MVW}) usually use facts about reductive algebraic groups that must be replaced by different arguments that are applicable to Lie superalgebras.

The strategy to prove Theorems  
\ref{TheoremeIntroduction1} and
\ref{TheoremeIntroduction2} 
is as follows. Let $\U$ be an irreducible $\mathfrak{g}$-module such that we have $\Hom_{\mathfrak{g}}(\U\,, \E) \neq \{0\}$. Let $\mathbb{D}: = \End_{\mathfrak{g}}(\U)$ and $\W: = \Hom_{\mathfrak{g}}(\U\,, \E)$. One can see that the spaces $\U$ and $\W$ are $\mathbb{D}$-modules, $\mathfrak{g}'$ acts on $\W$ as $\mathbb{D}$-linear transformations and we have a canonical map $\W \otimes_{\mathbb{D}} \U\to \E$ (see Section \ref{SectionTechnicalLemma}). 
Using this, in
 Lemmas~\ref{Lemma01} and~\ref{Lemma02}
we  prove that either  $\E=\W\otimes_\mathbb D\U $ (Type I) or $\E=\W\otimes_\mathbb D\U \oplus (\W\otimes_\mathbb D\U )^*$ (Type II). 
For Type II dual pairs, $\mathfrak{gl}_{\mathbb{D}}(\U)$ and $\mathfrak{gl}_{\mathbb{D}}(\W)$ are subalgebras of $\mathfrak{spo}(\E\,, \B)$. As explained in Proposition \ref{PropositionTypeII}, using that $\mathfrak{g}$ (resp. $\mathfrak{g}'$) supercommutes with $\mathfrak{gl}_{\mathbb{D}}(\W)$ (resp. $\mathfrak{gl}_{\mathbb{D}}(\U)$), we get $(\mathfrak{g}\,, \mathfrak{g}') = (\mathfrak{gl}_{\mathbb{D}}(\U)\,, \mathfrak{gl}_{\mathbb{D}}(\W))$.
Assume now that $(\mathfrak{g}\,, \mathfrak{g}')$ is a Type I dual pair. The form $\B$ induces a superinvolution $\eta$ on $\End(\E)$ uniquely defined by
\begin{equation*}
\B(\T(u)\,, v) = (-1)^{\left|\T\right|\cdot\left|u\right|} \B(u\,, \eta(\T)(v)) = 0\,, \qquad (u\,, v \in \E)\,.
\end{equation*}
This superinvolution satisfies $\eta(\X) = -\X$ for every $\X \in \mathfrak{spo}(\E\,, \B)$. In particular, it  preserves the subalgebra of $\End(\E)$ that is generated by the image of $\mathfrak{g}$. By the superized form of the Jacobson Density Theorem~(see \cite{RACINE}), it follows that  $\eta$ defines a superinvolution on $\End_{\mathbb{D}}(\U)$. In Proposition \ref{PropositionDivisionIota}, we prove that if $\mathbb{D}$ is such that $\End_{\mathbb{D}}(\Y)$ is endowed with a superinvolution, where
\begin{equation*}
\Y = \begin{cases} \mathbb{D}^{a|b} & \text{ if } \mathbb{D}_{\bar{1}} = \{0\} \\ \mathbb{D}^{a} & \text{ if } \mathbb{D}_{\bar{1}} \neq \{0\}, \end{cases}
\end{equation*}
then  either $\mathbb{D}_{\bar{1}} = \{0\}$, or  $\mathbb{K} = \mathbb{R}$ and $\mathbb{D} = \mathbb{C} \oplus \mathbb{C} \cdot \varepsilon$, with $\varepsilon^{2} = 1$ and $i \varepsilon = \varepsilon i$. 
The superinvolution $\eta$ on $\End_{\mathbb{D}}(\U)$ preserves $\mathscr{Z}(\End_{\mathbb{D}}(\U)) \cong \mathscr {Z}(\mathbb{D})$. As explained in Section \ref{SectionClassification}, there exists a superinvolution $\iota$ on $\mathbb{D}$ such that $\iota_{|_{\mathscr{Z}(\mathbb{D})}} = \eta$. We construct a $(\iota\,, \epsilon)$-superhermitian form $\gamma$ on $\U$ and a $(\iota \circ \delta\,, -\epsilon)$-superhermitian form $\gamma'$ on $\W$ such that  
$\B=\gamma' \otimes \gamma$ when $\mathbb K=\mathbb  C$ and  
$\B = \Re(\gamma' \otimes \gamma)$
when $\mathbb K=\mathbb R$. We then show that  $\mathfrak{g} = \mathfrak{g}(\U\,, \gamma)$ and $\mathfrak{g}' = \mathfrak{g}(\W\,, \gamma')$.


\medskip

As noted 
by Howe in \cite{HOWE89}, the action of $\mathfrak{sp}(\W)$ on $\S(\X)$ that we described in the beginning of this introduction can be superized 
to a representation of the complex orthosymplectic algebra $\mathfrak{spo}(\E\,, \B)$ on the supersymmetric algebra $\S(\E)$, where $\E_{\bar{0}}$ and $\E_{\bar{1}}$ are both even dimensional complex vector spaces.
 Indeed, let $\E = \E_{\bar{0}} \oplus \E_{\bar{1}}$, with $\dim_{\mathbb{C}}(\E_{\bar{0}})$ and  $\dim_{\mathbb{C}}(\E_{\bar{1}})$ even, endowed with an even, non-degenerate, $(-1)$-supersymmetric form $\B$. 
In Section \ref{SpinorOscillatorRepresentation}
we consider the subalgebra $\textbf{WC}(\E\,, \B)$ of $\End(\S(\E))$ generated by multiplication and derivation operators on $\S(\E)$. Again, if we denote by $\textbf{WC}(\E\,, \B)_{k}$ the elements of $\textbf{WC}(\E\,, \B)$ of degree not greater than $k$, with $k \in \mathbb{Z}_{+}$, we get the following relations
\begin{equation*}
\textbf{WC}(\E\,, \B)_{k} \cdot \textbf{WC}(\E\,, \B)_{l} \subseteq \textbf{WC}(\E\,, \B)_{k+l}\,, \quad \left[\textbf{WC}(\E\,, \B)_{k}\,, \textbf{WC}(\E\,, \B)_{l}\right] \subseteq \textbf{WC}(\E\,, \B)_{k+l-2}\,, \quad (k\,, l \in \mathbb{Z}_{2})\,,
\end{equation*}
where $\left[\cdot\,, \cdot\right]$ is the supercommutator on $\textbf{WC}(\E\,, \B)$. Similar to what happens with $\mathscr W(\W)$, the elements of degree $2$ of $\textbf{WC}(\E\,, \B)$ span a Lie superalgebra isomorphic to $\mathfrak{spo}(\E\,, \B)$ (see Lemma \ref{SectionLieSupergroups}). In particular, we get an action of $\mathfrak{spo}(\E\,, \B)$ on $\S(\E)$ coming from the natural inclusion $\textbf{WC}(\E\,, \B) \subseteq \End(\S(\E))$. Moreover, the action of $\mathfrak{spo}(\E\,, \B)_{\bar{0}} \cong \mathfrak{sp}(\E_{\bar{0}}\,, \B_{\bar{0}}) \oplus \mathfrak{so}(\E_{\bar{1}}\,, \B_{\bar 1})$ on $\textbf{WC}(\E\,, \B)$ exponentiates to a group action of $\Sp(\E_{\bar{0}}\,, \B_{0}) \times \O(\E_{\bar{1}}\,, \B_{1})$ on $\textbf{WC}(\E\,, \B)$ by group automorphisms (see Remark \ref{RemarkActionSpO}). 
In conclusion, we get an action 
on $\textbf{WC}(\E\,,\B)$
of the pair \[
\textbf{SpO}(\E\,, \B) := (\Sp(\E_{\bar{0}}\,, \B_{0}) \times \O(\E_{\bar{1}}\,, \B_{1})\,, \mathfrak{spo}(\E\,, \B)),
\] which is 
the Harish-Chandra pair of the orthosymplectic Lie supergroup (see Section \ref{SectionLieSupergroups}). 
We
can also define irreducible reductive  dual pairs in $\textbf{SpO}(\E\,,\B)$. The classification of such dual pairs follows from
Theorems \ref{TheoremeIntroduction1} and \ref{TheoremeIntroduction2}
(see Proposition~\ref{DualPairsInOrthosymplecticSupergroup}). Using this classification, in Section \ref{SectionDoubleCommutant}
we obtain the following simultaneous extension of Theorems \ref{DoubleCommutantSymplecticComplex} and \ref{DoubleCommutantOrthogonalComplex}:
\begin{theo}{(Double Commutant Theorem)}

Let $\big(\mathscr{G}\,, \mathscr{G}'\big) = \big((\G\,, \mathfrak{g})\,, (\G'\,, \mathfrak{g}')\big)$ be an irreducible reductive dual pair in $\normalfont{\textbf{SpO}}(\E\,, \B)$. Then the superalgebra $\normalfont{\textbf{WC}}(\E\,, \B)^{\mathscr{G}}$ is generated by $\mathfrak{g}'$.
\label{thm-thm1.8-hadi}
\end{theo}
We also prove an analogue of Theorem~\ref{DoubleCommutantSymplecticReal}. In this case $\E$ will be a \emph{real} $(-1)$-supersymmetric  space but we associate to it the Weyl-Clifford algebra 
$\normalfont{\textbf{WC}}(\E^{\mathbb{C}}\,, \B^{\mathbb{C}})$. We obtain the following result:
\begin{theo}
Let $(\mathscr{G}\,, \mathscr{G}') := \left((\G\,, \mathfrak{g})\,, (\G'\,, \mathfrak{g}')\right)$ be an irreducible reductive dual pair in $\normalfont{\textbf{SpO}}(\E\,, \B)$. Then the superalgebra $\normalfont{\textbf{WC}}(\E^{\mathbb{C}}\,, \B^{\mathbb{C}})^{\mathscr{G}}$ is generated by $\mathfrak{g}'$.

\label{DoubleCommutantTheoremR}

\end{theo}

Theorems~\ref{thm-thm1.8-hadi} and 
\ref{DoubleCommutantTheoremR}
exhibit a "compatibility" between  dual pairs in $\textbf{SpO}(\E\,, \B)$ and the Spinor-Oscillator representation of $\mathfrak{spo}(\E\,, \B)$ (or $\textbf{SpO}(\E\,, \B)$). Even though this circle of ideas pertains to Howe duality in the semisimple case, we do not consider  Howe duality in full generality in this paper. However,  in this context we refer the reader to other works (see for example  \cite[Theorem~8]{HOWE89}, \cite[Theorem~3.2]{CHENGWANG1}, \cite[Theorem~5.28]{CHENGWANG}, \cite{COULEMBIER}, \cite{HOWELU}).

\medskip

This paper is organized as follows. In Section \ref{SectionPreliminaries}, we introduce the notion of reductive and irreducible dual pairs and prove that every reductive dual pair is a direct sum
of irreducible pairs. In Section \ref{SectionTechnicalLemma}, we prove a technical lemma that will play a key part in our classification of dual pairs. In Section \ref{SectionDivisionSuperalgebra}, we study the division superalgebras with superinvolutions. In Section \ref{SectionClassification}, we classify irreducible reductive dual pairs in the real and complex orthosymplectic Lie superalgebra. In Section \ref{SectionLieSupergroups}, we introduce the concept of dual pairs in an orthosymplectic Lie supergroup. A classification of such pairs is directly obtained from 
Section \ref{SectionClassification}. In Section \ref{SpinorOscillatorRepresentation}, we construct the Spinor-Oscillator representation of the orthosymplectic Lie superalgebra $\mathfrak{spo}$ (over $\mathbb{K} = \mathbb{R}$ or $\mathbb{C}$). In Section \ref{SectionDoubleCommutant} we prove Theorems~\ref{thm-thm1.8-hadi} 
and \ref{DoubleCommutantTheoremR}.
Finally, as an application of Theorem~\ref{thm-thm1.8-hadi}, in Section~\ref{Section9} 
we prove that Howe duality holds for the dual pair
$(\normalfont{\textbf{SpO}}(2n|1)\,, \normalfont{\textbf{OSp}}(2k|2l)) \subseteq \normalfont{\textbf{SpO}}(\mathbb{C}^{2k|2l} \otimes \mathbb{C}^{2n|1})$.
\medskip

\noindent \textbf{Acknowledgments: }
The authors thank Tomasz Przebinda,  Michel Racine, 
Alistair Savage
and   Vera Serganova for fruitful discussions. This work was done while the first author was partially supported by an NFRF grant as a postdoctoral fellow at the University of Ottawa. The first author thanks the
Institute for Mathematical Sciences at the National University of Singapore for invitation to present this work in the conference on
Representations and Characters: Revisiting the Works of Harish-Chandra and Andr\'{e} Weil. The productive  ambience of the conference played a catalytic role in the completion of this project. 
The second author is partially supported by an NSERC Discovery Grant (RGPIN-2018-04044).

\section{Generalities and dual pairs in orthosymplectic Lie superalgebras}

\label{SectionPreliminaries}

As usual, a superalgebra simply means a $\mathbb Z_2$-graded (possibly nonassociative) algebra. A Lie superalgebra over $\mathbb{K}$ is a superalgebra $\mathfrak{g} = \mathfrak{g}_{\bar{0}} \oplus \mathfrak{g}_{\bar{1}}$ over $\mathbb{K}$  whose product $\left[\cdot\,, \cdot\right]$ satisfies the following relations  for every homogeneous $\X, \Y, \Z \in \mathfrak{g}$:
\begin{enumerate}
\item $\left[\X\,, \Y\right] = -(-1)^{\left|\X\right|\cdot\left|\Y\right|}\left[\Y\,, \X\right]$\,,
\item $(-1)^{\left|\X\right|\cdot\left|\Z\right|} \left[\X\,, \left[\Y\,, \Z\right]\right] + (-1)^{\left|\Y\right|\cdot\left|\X\right|} \left[\Y\,, \left[\Z\,, \X\right]\right] + (-1)^{\left|\Z\right|\cdot\left|\Y\right|} \left[\Z\,, \left[\X\,, \Y\right]\right] = 0$\,.
\end{enumerate}

Let $\V = \V_{\bar{0}} \oplus \V_{\bar{1}}$ be a finite dimensional $\mathbb Z_2$-graded vector space over $\mathbb{K}$. We denote by $\mathfrak{gl}(\V)$ the general linear Lie superalgebra associated to $\V$. The $\mathbb{Z}_{2}$-grading of $\mathfrak{gl}(\V)$ is given by 
\begin{equation*}
\mathfrak{gl}(\V)_{\alpha} = \left\{\X \in \mathfrak{gl}(\V)\,, \X(\V_{\beta}) \subseteq \V_{\alpha + \beta}\,, \beta \in \mathbb{Z}_{2}\right\}\,, \qquad (\alpha \in \mathbb{Z}_{2})\,,
\end{equation*}
and its superbracket $\left[\cdot\,, \cdot\right]$ is given  by
$\left[\X\,, \Y\right] = \X\Y - (-1)^{\left|\X\right|\cdot\left|\Y\right|}\Y\X$
for homogeneous $\X, \Y \in \mathfrak{gl}(\V)$.

Let $\mathfrak{g}$ be a 
Lie superalgebra. 
For any two $\mathfrak{g}$-modules 
$(\lambda,\W)$ and $(\lambda',\W')$ we define
the $\mathbb{Z}_{2}$-graded subspace 
$\Hom_{\mathfrak g}(\W\,, \W')$
of $\Hom(\W\,, \W')$  by
\begin{equation}
\Hom_{\mathfrak{g}}(\W\,, \W')_{\alpha} = \left\{\T \in \Hom(\W\,, \W')_{\alpha}\,, \T \circ \lambda(\X) = (-1)^{\alpha\left|\X\right|} \lambda'(\X) \circ \T\right\}\,,\qquad(\alpha\in\mathbb Z_2)\,.
\label{Equation1}
\end{equation}
For two irreducible $\mathfrak{g}$-modules $\W$ and $\W'$, we write $\W \cong \W'$ if there exists a homogeneous invertible $\T \in \Hom_{\mathfrak{g}}(\W\,, \W')$.

\begin{rema}

Intertwining operators are usually defined as linear maps $\T:\W\to\W'$ satisfying the relation $\T \circ \lambda(\X) = \lambda'(\X) \circ \T$. The relation between intertwining operators and elements of $\Hom_{\mathfrak g}(\W,\W')$ is as follows. There exists a linear map $\T\to \T^\diamond$ on $\Hom(\W,\W')$ that is uniquely determined by setting 
\begin{equation*}
\T^{\diamond}(w) := (-1)^{\left|\T\right|\cdot\left|w\right|}\T(w)\,, \qquad (w \in \W)\,,
\end{equation*}
for homogeneous $\T$. Then  $\T$ is an intertwining operator if and only if $\T^\diamond\in\Hom_{\mathfrak g}(\W,\W')$. Indeed if $\T$ is an interwining operator then 
\begin{equation*}
\T^{\diamond} \circ \lambda(\X)(w) = (-1)^{\left|\T\right|\cdot\left|\lambda(\X)(w)\right|} \T(\lambda(\X)(w)) = (-1)^{\left|\T\right|\cdot\left|\X\right|}(-1)^{\left|\T\right|\cdot\left|w\right|} \lambda'(\X)(\T(w)) = (-1)^{\left|\T\right|\cdot\left|\X\right|} \lambda'(\X) \circ \T^{\diamond}(w)\,,
\end{equation*}
i.e. $\T^{\diamond} \in \Hom_{\mathfrak{g}}(\W\,, \W')$. The converse is proved similarly. For more details, see \cite[Section~1.1.1]{CHENGWANG}.

\end{rema}

\begin{prop}

Let $\mathfrak{g}$ and $\mathfrak{g}'$ be two Lie superalgebras and let $\V$ be a $\mathfrak{g} \oplus \mathfrak{g}'$-module. Assume that $\V$ is semisimple both as a $\mathfrak{g}$-module and as a $\mathfrak{g}'$-module. Then $\V$ is also semisimple as a $\mathfrak{g} \oplus \mathfrak{g}'$-module.

\label{PropositionJointReducibility}

\end{prop}

\begin{proof}

Let $\iota: \mathscr{U}(\mathfrak{g}) \to \End(\V)$ and $\iota': \mathscr{U}(\mathfrak{g}') \to \End(\V)$ be the superalgebra homomorphisms obtained from the actions of $\mathfrak{g}$ and $\mathfrak{g}'$ on $\V$. Furthermore, set $\mathscr{A}:=\iota(\mathscr{U}(\mathfrak{g}))$ and $\mathscr{B}:=\iota'(\mathscr{U}(\mathfrak{g}'))$.  From Lemma \ref{LemmaAppendixB1} it follows that both $\mathscr{A}$ and $\mathscr{B}$ are semisimple superalgebras. Now consider the map
\begin{equation*}
 j: \mathscr{A} \otimes  \mathscr{B} \ni a\otimes b \mapsto \T_{a, b} \in \End(\V)\,, \qquad \T_{a,b}(v):=abv\,, \qquad (v \in \V)\,.
\end{equation*}
Since $\mathscr{A}$ and $\mathscr{B}$ supercommute, the map $j$ is a homomorphism of associative superalgebras. Furthermore, its image is the same as the image of $\mathscr{U}(\mathfrak{g} \oplus \mathfrak{g}') \cong \mathscr{U}(\mathfrak{g})\otimes \mathscr{U}(\mathfrak{g}')$. Thus by Theorem \ref{SuperWedderburnTheorem} it suffices to verify that the image of $j$ is a semisimple superalgebra. From Proposition \ref{PropositionSemisimpleSuperalgebras} it follows that $\mathscr{A} \otimes \mathscr{B}$ is semisimple, so $\mathscr{A} \otimes \mathscr{B} / \ker(j) \cong \Im(j)$ is semisimple.
\end{proof}


In the rest of this section, $\B$ will be an even, non-degenerate, $(-1)$-supersymmetric form on $\V = \V_{\bar{0}} \oplus \V_{\bar{1}}$. We use the notation  $\mathfrak{spo}(\V\,, \B)$ for the corresponding orthosymplectic Lie superalgebra, i.e.
\begin{equation*}
\mathfrak{spo}(\V\,, \B)_{\alpha} = \left\{\A \in \mathfrak{gl}(\V)_{\alpha}\,, \B(\A(x)\,, y) + (-1)^{\alpha \xi}\B(x, \A(y)) = 0\,, x \in \V_{\xi}\,, y \in \V\right\}\,, \qquad (\alpha \in \mathbb{Z}_{2})\,.
\end{equation*}

We denote by $(\pi\,, \V)$ the natural action of $\mathfrak{spo}(\V\,, \B)$ on $\V$. \\

Let $\mathfrak{g}$ be a ($\mathbb Z_2$-graded) subalgebra of $\mathfrak{spo}(\V,\,\B)$. Let $(\lambda\,, \W)$, where $\W = \W_{\bar{0}} \oplus \W_{\bar{1}}$, be an irreducible $\mathfrak{g}$-module such that $\Hom_{\mathfrak{g}}(\W\,, \V) \neq \{0\}$. We denote by $\W_{\lambda}$ the $\mathfrak{g}$-isotypic component of $(\lambda\,, \W)$ in $\V$, i.e.
\begin{equation*}
\W_{\lambda} = \left\{\T(\W)\,, \T \in \Hom_{\mathfrak{g}}(\W\,, \V)\right\}\,.
\end{equation*}

\begin{nota}

For every sub-superalgebra $\mathfrak{l}$ of $\mathfrak{spo}(\V\,, \B)$, we denote by $\mathcal{C}_{\mathfrak{spo}(\V\,, \B)}(\mathfrak{l})$ the supercommutant of $\mathfrak{l}$ in $\mathfrak{spo}(\V\,, \B)$, i.e. 
\begin{equation*}
\mathcal{C}_{\mathfrak{spo}(\V\,, \B)}(\mathfrak{l}) = \left\{\X \in \mathfrak{spo}(\V\,, \B)\,, \left[\X\,, \Y\right] = 0\,, \Y \in \mathfrak{l}\right\}\,.
\end{equation*}

\end{nota}

\begin{defn}

A dual pair in $\mathfrak{spo}(\V\,, \B)$ is a pair of subalgebras $(\mathfrak{g}\,, \mathfrak{g}')$ 
of the Lie superalgebra $\mathfrak{spo}(\V\,, \B)$
such that $\mathcal{C}_{\mathfrak{spo}(\V\,, \B)}(\mathfrak{g}) = \mathfrak{g}'$ and $\mathcal{C}_{\mathfrak{spo}(\V\,, \B)}(\mathfrak{g}') = \mathfrak{g}$.
We say that the dual pair $(\mathfrak{g}\,, \mathfrak{g}')$ is reductive if $\mathfrak{g}$ and $\mathfrak{g}'$ both act on $\V$ semisimply.

\end{defn}

\begin{rema}

Assume that there exists a $\B$-orthogonal decomposition of $\V$ of the form $\V = \V^{1} \oplus^{\perp} \V^{2}$, where both $\V^{1} = \V^{1}_{\bar{0}} \oplus \V^{1}_{\bar{1}}$ and $\V^{2} = \V^{2}_{\bar{0}} \oplus \V^{2}_{\bar{1}}$ are $\mathfrak{g}$ and $\mathfrak{g}'$-invariant. Then, $(\mathfrak{g}_{1}\,, \mathfrak{g}'_{1})$ (resp. $(\mathfrak{g}_{2}, \mathfrak{g}'_{2}))$ is a dual pair in $\mathfrak{spo}(\V^{1}, \B^{1})$ (resp. $\mathfrak{spo}(\V^{2}, \B^{2})$), where $\mathfrak{g}_{i} = \mathfrak{g}_{|_{\V^{i}}}$, $\mathfrak{g}'_{i} = \mathfrak{g}'_{|_{\V^{i}}}$ and $\B^{i} = \B_{|_{\V^{i} \times \V^{i}}}$.
In this case, we say that $(\mathfrak{g}\,, \mathfrak{g}')$ is the direct sum of $(\mathfrak{g}_{1}\,, \mathfrak{g}'_{1})$ and $(\mathfrak{g}_{2}\,, \mathfrak{g}'_{2})$. In particular, we have $\mathfrak{g} = \mathfrak{g}_{1} \oplus \mathfrak{g}_{2}$ and 
$\mathfrak g'=\mathfrak g'_1\oplus\mathfrak g'_2$. Indeed, because both $\V_{1}$ and $\V_{2}$ are $\mathfrak{g}$-invariant, it follows that $\mathfrak{g} \subseteq \mathfrak{g}_{1} \oplus \mathfrak{g}_{2} \subseteq \mathfrak{spo}(\V\,, \B)$. Moreover, one can easily see that $\mathfrak{g}_{1} \oplus \mathfrak{g}_{2}$ commutes with $\mathfrak{g}'$, and then $\mathfrak{g}_{1} \oplus \mathfrak{g}_{2} \subseteq \mathcal{C}_{\mathfrak{spo}(\V\,, \B)} (\mathfrak{g}') = \mathfrak{g}$, so $\mathfrak{g} = \mathfrak{g}_{1} \oplus \mathfrak{g}_{2}$.
The argument for $\mathfrak g'$ is similar. 
\label{Remark1}

\end{rema}

In view of Remark \ref{Remark1}, the following definition is natural. 
\begin{defn}

The dual pair $(\mathfrak g,\mathfrak g')$ is called irreducible if there is no proper $\B$-orthogonal decomposition of $\V$ of the form $\V = \V^{1} \oplus^{\perp} \V^{2}$, where both $\V^{1} $ and $\V^{2}$ are $\mathfrak{g}\oplus\mathfrak{g}'$-invariant.
\end{defn}

From now on, we assume that $(\mathfrak{g}, \mathfrak{g}')$ is a reductive dual pair in $\mathfrak{spo}(\V\,, \B)$.
Let $(\lambda_{1}\,, \W_{1})\,, \ldots\,, (\lambda_{n}\,, \W_{n})$ be a complete list of representatives for the irreducible summands of the $\mathfrak{g}$-module $\V$ and let $\W_{\lambda_{i}}\,, 1 \leq i \leq n$, be the $\mathfrak{g}$-isotypic component corresponding to $\lambda_{i}$, so that  $\V = \bigoplus_{i=1}^{n} \W_{\lambda_{i}}$.
\begin{lemme}

For every $1 \leq i \leq n$, $\W_{\lambda_{i}}$ is a $\mathfrak{g}'$-module.

\label{LemmaJuly11}

\end{lemme}

\begin{proof}
It suffices  to verify that $\pi(\X') \circ \T \in \Hom_{\mathfrak{g}}(\W_{i}\,, \V)$ for every $\X' \in \mathfrak{g}'$ and $\T \in \Hom_{\mathfrak{g}}(\W_{i}\,, \V)$. For homogeneous $\X \in \mathfrak{g}, \X' \in \mathfrak{g}'$ and $\T \in \Hom_{\mathfrak{g}}(\W_{i}\,, \V)$, we have
\begin{eqnarray*}
(\pi(\X') \circ \T) \circ \lambda_{i}(\X) & = & \pi(\X') \circ \T \circ \lambda_{i}(\X) = (-1)^{\left|\T\right|\cdot\left|\X\right|}\pi(\X') \circ \pi(\X) \circ \T \\
& = & (-1)^{\left|\X\right|\cdot\left|\X'\right|}(-1)^{\left|\T\right|\cdot\left|\X\right|}\pi(\X) \circ (\pi(\X') \circ \T) = (-1)^{\left|\X\right|\cdot(\left|\T\right|+\left|\X'\right|)} \pi(\X) \circ (\pi(\X') \circ \T)\,,
\end{eqnarray*}
and the lemma follows from the relation $\left|\pi(\X') \circ \T\right| = \left|\T\right|+\left|\X'\right|$.
\end{proof}

We now refine the decomposition of  $\V$ into $\mathfrak g$-isotypic components  as \[
\V = \bigoplus\limits_{i=1}^{n} \W_{\lambda_{i}} = \bigoplus\limits_{i=1}^{n} \left(\bigoplus\limits_{j=1}^{d_{i}} \T_{j}(\W_{i})\right),
\] where $d_{i} = \dim_{\mathbb{K}} \Hom_{\mathfrak{g}}(\W_{i}\,, \V)$ and $\T_{1}, \ldots, \T_{d_{i}}$ is a homogeneous $\mathbb K$-basis of $\Hom_{\mathfrak{g}}(\W_{i}\,, \V)$.

\begin{nota}

For every $1 \leq i \leq n$ and $1 \leq j \leq d_{i}$, we denote 
the subspace 
$\T_{j}(\W_{i})$ of $\W_{\lambda_{i}}$
by $\V^{j}_{i}$.
\end{nota}

\begin{lemme}

The restriction of $\B$ to $\V^{j}_{i}$ is either zero or non-degenerate.

\label{Lemma0}

\end{lemme}

\begin{proof}

The subspace of $\V^{j}_{i}$ defined by
$
\left\{v \in \V^{j}_{i}\,, \B(v\,, w) = 0, w \in \V^{j}_{i}\right\}
$
is $\mathbb{Z}_{2}$-graded and $\mathfrak{g}$-invariant. Consequently,  this subspace is equal to either $\{0\}$ (and in this case  $\B: \V^{j}_{i} \times \V^{j}_{i} \to \mathbb{K}$ is non-degenerate) or $\V^{j}_{i}$ (and in this case $\B: \V^{j}_{i} \times \V^{j}_{i} \to \mathbb{K}$ is zero).
\end{proof}

\begin{rema}

\begin{enumerate}
\item Assume that the restriction of $\B$ to $\V^{j}_{i}$ is non-degenerate. It follows that the $\mathfrak{g}$-module $\V^{j}_{i}$ is self-dual (the $\mathfrak{g}$-equivariant map $\S: \V^{j}_{i} \to {\V^{j}_{i}}^{*}$ is given by $\S(v)(w) = \B(v\,, w)\,, v, w \in \V^{j}_{i}$).
\item Assume that the restriction of $\B$ to $\V^{j}_{i}$ is zero. Then there exist $1 \leq k \leq n$ and $1 \leq l \leq d_{k}$ such that the form $\B$ on $\V^{j}_{i} \times \V^{l}_{k}$ is non-zero. Similar to the proof of Lemma \ref{Lemma0}, it follows that $\V^{j}_{i} \cong {\V^{l}_{k}}^{*}$.
\end{enumerate}

\end{rema}

\begin{lemme}

Let $i \in [|1, n|]$ be such that the restriction of $\B$ to $\W_{\lambda_{i}} \times \W_{\lambda_{i}}$ is non-zero. Then $\B$ is non-degenerate on $\W_{\lambda_{i}}$ and the space $\W_{\lambda_{i}}$ is $\B$-orthogonal to every $\W_{\lambda_{j}}, j \neq i$.

\label{Lemma01}

\end{lemme}

\begin{proof}

Fix such $i$. Because the form $\B$ on $\W_{\lambda_{i}} \times \W_{\lambda_{i}}$ is non-zero, there exists $1 \leq k < t \leq d_{i}$ such that $\B(\V^{k}_{i}\,, \V^{t}_{i}) \neq \{0\}$. It follows from Lemma \ref{Lemma0} that $\V^{k}_{i} \cong {\V^{t}_{i}}^{*}$. By using that $\V^{k}_{i} \cong {\V^{t}_{i}}$, it follows that $\V^{k}_{i} \cong {\V^{k}_{i}}^{*}$ as a $\mathfrak{g}$-module. 

If $\B(\V^{a}_{i}\,, \V^{b}_{j}) \neq 0$ for some $b, j$ such that $i \neq j$, then Lemma \ref{Lemma0} yields that $\V^{a}_{i} \cong {\V^{b}_{j}}^{*}$. By using that $\V^{a}_{i} \cong {\V^{a}_{i}}^{*}$, it follows that $\V^{a}_{i} \cong \V^{b}_{j}$, which is impossible. In particular, $\W_{\lambda_{i}} \perp \W_{\lambda_{j}}$ and it follows that $\Rad(\B_{|_{\W_{\lambda_{i}} \times \W_{\lambda_{i}}}}) \subseteq \Rad(\B) = \{0\}$.
\end{proof}

\begin{lemme}

Let $i \in [|1, n|]$ such that  $\B_{|_{\W_{\lambda_{i}} \times \W_{\lambda_{i}}}} = 0$. Then there exists a unique $j \neq i$ such that $B: \W_{\lambda_{i}} \oplus \W_{\lambda_{j}} \times \W_{\lambda_{i}} \oplus \W_{\lambda_{j}} \to \mathbb{K}$ is non-degenerate with $\B_{|_{\W_{\lambda_{j}} \times \W_{\lambda_{j}}}} = 0$. In particular,  for  $1\leq a \leq  d_{i}$ and $1\leq b \leq d_{j}$ we have, $\V^{a}_{i} \cong {\V^{b}_{j}}^{*}$. Moreover, $\W_{\lambda_{i}} \oplus \W_{\lambda_{j}} \perp \W_{\lambda_{k}}$ for all$ k \neq i, j$.

\label{Lemma02}

\end{lemme}

\begin{proof}

The existence of such $j$ is straightforward (otherwise, $\W_{\lambda_{i}} \subseteq \Rad(\B) = \{0\}$). Let us first prove its uniqueness. If $k$ is such that $\B: \W_{\lambda_{i}} \times \W_{\lambda_{k}} \to \mathbb{K}$ is non-zero, then we get that $\V^{1}_{i} \cong {\V^{1}_{k}}^{*}$, but we also know that $\V^{1}_{i} \cong {\V^{1}_{j}}^{*}$, hence $\V^{1}_{j} \cong \V^{1}_{k}$. This implies that $k = j$. Moreover, it follows from Lemma \ref{Lemma01} that $\B: \W_{\lambda_{j}} \times \W_{\lambda_{j}} \to \mathbb{K}$ is zero, otherwise $\W_{\lambda_{j}}$ would be orthogonal to $\W_{\lambda_{i}}$. 

One can easily see that $\W_{\lambda_{j}}$ is orthogonal to every $\W_{\lambda_{k}}, k \neq i$ (by applying the above argument for uniqueness of $j$ to $\W_{\lambda_{j}}$), and it follows that the restriction of $\B$ to $\W_{\lambda_{i}} \oplus \W_{\lambda_{j}}$ is non-degenerate.
\end{proof}

\begin{nota}

Without loss of generality, we assume that $\B_{|_{\W_{\lambda_{i}}}}\,, 1 \leq i \leq m$, is non-degenerate and $\B_{|_{\W_{\lambda_{i}}}}\,, m+1 \leq i \leq n$, is zero. For every $1 \leq k \leq m + \frac{n-m}{2}$, we denote by $\widetilde{\W}^{k}$ the subspace of $\V$ given by
\begin{equation*}
\widetilde{\W}^{k} = \begin{cases} \W_{\lambda_{k}} & \text{ if } 1 \leq k \leq m \\ \W_{\lambda_{m + 2(k-m)-1}} \oplus \W_{\lambda_{m + 2(k-m)}} & \text{ otherwise } \end{cases}\,.
\end{equation*}

\end{nota}

From Lemmas \ref{Lemma01} and \ref{Lemma02} it follows that the restriction of $\B$ to $\widetilde{\W}^{k}, 1 \leq k \leq m$, is non-degenerate,  and the subspace $\widetilde{\W}^{k}, m+1 \leq k \leq m + \frac{n-m}{2},$ is a direct sum of two ($\B$-isotropic) $\mathfrak{g}$-isotypic components such that $\B_{|_{\widetilde{\W}^{k} \times \widetilde{\W}^{k}}}$ is non-degenerate. Futhermore, $\widetilde{\W}^{i} \perp \widetilde{\W}^{j}$ for $i \neq j$.

The following proposition is a superized analogue  of one of the steps in the classification of dual pairs in the purely even case (see \cite{HOWEALG} or \cite{MVW}). However, we are not able to find a superized extension of the  proofs in \cite{HOWEALG} or \cite{MVW}. Thus, the proof that we give below is essentially new.

\begin{prop}

The joint action of $\mathfrak{g}$ and $\mathfrak{g}'$ on $\W_{\lambda_{i}}, 1 \leq i \leq n$, is irreducible. In particular, every $\mathfrak{g}$-isotypic component $\W_{\lambda_{i}}$ is a full $\mathfrak{g}'$-isotypic component.

\label{Proposition12}

\end{prop}

\begin{proof}

Let $\U \subseteq \V$ be an irreducible $\mathfrak{g} \oplus \mathfrak{g}'$-submodule.
From Lemma \ref{LemmaJuly11} it follows that 
 $\U$ is a direct sum of several copies of a single irreducible $\mathfrak{g}$-module (i.e., $\U\subseteq \W_{\lambda_i}$ for some $1\leq i\leq n$), and also several copies of a single irreducible $\mathfrak{g}'$-module. By invariance of $\B(\cdot\,,\cdot)$, we know that  $\B\big|_{\U \times \U}$ is either non-degenerate or zero.

\textbf{Case I:}  $\B\big|_{\U\times \U}$ is non-degenerate. Then we can decompose $\V$ as 
\begin{equation*}
\V = \U \oplus \U^{\perp}\,.
\end{equation*}
Note that $\B_{|\U^\perp\times \U^\perp}$ is also nondegenerate. Furthermore, both $\U$ and $\U^{\perp}$ are $\mathfrak{g} \oplus \mathfrak{g}'$-invariant. Consequently, elements of both $\mathfrak{g}$ and $\mathfrak{g}'$ can be represented by matrices of the form $\diag(\X_{\U}\,, \X_{\U^{\perp}})$. We show that for any pair of irreducible $\mathfrak{g}$-modules $\E$ and $\E'$ such that $\E \subseteq \U$ and $\E' \subseteq \U^{\perp}$, we have $\E \not\cong \E'$. This in particular implies that $\U = \W_{\lambda_i}$ for some $1\leq i\leq n$.

Suppose, on the contrary, that $\E,\E'$ exist such that $\E\cong\E'$.  For any  $\X \in \mathfrak{g}$, the matrix $\Y: = \diag(\X\big|_{\U}\,, 0)$ is in $\mathfrak{spo}(\V\,, \B)$, and it also supercommutes  with $\mathfrak{g}'$, hence we have $\Y \in \mathfrak{g}$. Clearly $\Y \cdot \E' = 0$, hence by the assumption that  $\E \cong \E'$ we should have $\Y \cdot \E = 0$. But the actions of $\X$ and $\Y$ on $\E$ are the same, hence $\X \cdot \E=0$. As $\U$ is a direct sum of several copies of $\E$ or $\Pi(\E)$, it follows that $\X \cdot \U = 0$. 
In conclusion, we have proved that for every $\X \in \mathfrak{g}$ we have $\X\cdot \U = 0$. Thus, in particular $\E$ is the trivial $\mathfrak{g}$-module. 

Now let $\W'$ denote the $\mathfrak{g}$-isotypic component for the trivial $\mathfrak{g}$-module inside $\U^{\perp}$. Then $\W'$ is also $\mathfrak{g}'$-invariant (see Lemma \ref{LemmaJuly11}). Take an irreducible $\mathfrak{g} \oplus \mathfrak{g}'$ submodule of $\U' \subseteq \W'$. There are two cases to consider:

\textbf{Case I.a: }$\B\big|_{\U'\times \U'}$ is non-degenerate. In this case we can write $\U^\perp$ as $\U^{\perp}=\U'\oplus {\U'}^{\perp}$, therefore we have a $\mathfrak{g}\oplus \mathfrak{g}'$-invariant decomposition of $\V$ into nondegenerate, $(-1)$-supersymmetric spaces, i.e.,
\begin{equation*}
\V = \U \oplus \U'\oplus {\U'}^{\perp}\,.
\end{equation*}
Since the action of $\mathfrak{g}$ on $\U$ and $\U'$ is trivial, it follows that the elements of $\mathfrak{g}$ are of the form
\begin{equation*}
\diag(0_{\U}\,, 0_{\U'}\,,*_{{\U'}^{\perp}})\,.
\end{equation*}
Clearly matrices of the form $(\mathfrak{spo}(\U\oplus \U')\,,0_{{\U'}^{\perp}})$ supercommute with the elements of $\mathfrak{g}$, hence they belong to $\mathfrak{g}'$. But the latter elements form a subalgebra of $\mathfrak{spo}(\V\,, \B)$ that does \emph{not} preserve the decomposition $\V = \U \oplus \U^{\perp}$. This is a contradiction.

\textbf{Case I.b:} $\B\big|_{\U'\times \U'}=0$. In this case, we can find an irreducible $\mathfrak{g}\oplus \mathfrak{g}'$-submodule $\U''\subseteq \U^{\perp}$ such that 
$\B(\U'\,, \U'')\neq 0$. From irreducibility of $\U'$ and $\U''$ it follows that indeed the bilinear map $\B\big|_{\U'\times \U''}$ is non-degenerate. We claim that $\B(\U''\,, \U'') = 0$. Indeed from non-degeneracy of $\B\big|_{\U'\times \U''}$ it follows that the irreducible $\mathfrak{g}$-submodules of $\U'$ and $\U''$ are dual to each other, hence $\U''$ is a direct sum of trivial $\mathfrak{g}$-modules. If $\B\big|_{\U''\times \U''}$ is nonzero, then it is nondegenerate because of irreducibility of $\U''$. Thus, we  are back to Case I.a above, with $\U'$ replaced by $\U''$, and we arrive at a contradiction.  

By the above argument, from now on we can assume that  $\B(\U''\,, \U'')=0$. Thus obtain a $\mathfrak{g}\oplus \mathfrak{g}'$-invariant decomposition into 
nondegenerate $(-1)$-supersymmetric spaces
\begin{equation}
\label{V=UUU1}
\V = \U\oplus (\U'\oplus \U'')\oplus \widetilde{\U}\,,
\end{equation}
where $\widetilde{\U}$ is the orthogonal complement of $\U'\oplus \U''$ in $\U^{\perp}$. Elements of $\mathfrak{g}$ and $\mathfrak{g}'$ are of the form 
\begin{equation*}
\diag(\A\,, \B\,, \C)\,,
\end{equation*}
compatible with the above direct sum decomposition. Furthermore, for elements of $\mathfrak{g}$ we have $\A = \B = 0$. But then it follows that the supercommutant of $\mathfrak{g}$ should contain matrices of the form 
\begin{equation*}
\diag(\mathfrak{spo}(\U \oplus (\U'\oplus \U''))\,, 0_{\tilde \U})\,,
\end{equation*}
and the subalgebra formed by the above matrices (which is included in $\mathfrak{g}'$) does not leave the decomposition \eqref{V=UUU1} invariant.  
This is a contradiction. 

\textbf{Case II:} 
$\B\big|_{\U\times \U}=0$. Then there is an irreducible $\mathfrak{g}\oplus \mathfrak{g}'$-module $\U'$ such that $\B|_{\U \times \U'}$ is nonzero, hence non-degenerate (because $\U$ and $\U'$ are irreducible). In particular, as $\mathfrak{g}$-modules $\U$ and $\U'$ are dual to each other. If $\B(\U'\,, \U')\neq 0$, then we are back to Case I above, with $\U$ replaced by $\U'$. If $\B(\U'\,, \U')=0$, then the restriction of $\B$ to $\U\oplus\U'$ is a nondegenerate, $(-1)$-supersymmetric bilinear form and therefore we obtain  a $\mathfrak{g} \oplus \mathfrak{g}'$-invariant decomposition of $\V$, i.e.,
\begin{equation*}
\V = \U \oplus \U'\oplus (\U \oplus \U')^{\perp}\,.
\end{equation*}
Note that the restriction of $\B$ to $(\U\oplus \U')^\perp$ is also nondegenerate. 
Now the matrix
\begin{equation*}
\Y:=\diag(\Id_{\U}\,, -\Id_{\U'}\,, 0_{(\U\oplus \U')^{\perp}})
\end{equation*}
is in $\mathfrak{spo}(\V\,, \B)$ and supercommutes with
elements of $\mathfrak{g}'$, hence  $\Y\in\mathfrak{g}$. 
Since   $\Y\big|_{(\U \oplus \U')^{\perp}}=0$ and $\Y\big|_{\U}=-1$, it follows that $\Hom_{\mathfrak{g}}(\U\,, (\U\oplus \U')^\perp)=\{0\}$. Consequently, $\U$ is a full $\mathfrak{g}$-isotypic component. 
\end{proof}
The following assertion is an immediate consequence of Proposition \ref{Proposition12}.
\begin{prop}

For every $1 \leq i \leq m + \frac{n-m}{2}$, $(\mathfrak{g}_{|_{\widetilde{\W}^{i}}}\,, \mathfrak{g}'_{|_{\widetilde{\W}^{i}}})$ is an irreducible reductive dual pair in $\mathfrak{spo}(\widetilde{\W}^{i}\,, \B_{|_{\widetilde{\W}^{i}}})$. In particular, every reductive dual pair is the direct sum of irreducible reductive dual pairs.

\label{Proposition1}

\end{prop}

\begin{defn}

Let $(\mathfrak{g}\,, \mathfrak{g}')$ be an irreducible reductive dual pair in $\mathfrak{spo}(\V\,, \B)$. If the joint action of $\mathfrak{g}$ and $\mathfrak{g}'$ on $\V$ is irreducible, we say that $(\mathfrak{g}\,, \mathfrak{g}')$ is of Type I. If $\V = \V^{1} \oplus \V^{2}$ is a direct sum of two isotropic subspaces, where $\V^{1}$ and $\V^{2}$ are invariant under both $\mathfrak{g}$ and $\mathfrak{g}'$, and $\V^{1} \cong {\V^{2}}^{*}$ (as $\mathfrak{g}\oplus \mathfrak{g}'$-modules), we say that $(\mathfrak{g}\,, \mathfrak{g}')$ is of Type II.

\end{defn}

\section{The tensor product decomposition of $\V$}

\label{SectionTechnicalLemma}

In this section, we prove a general lemma that is  well-known in the purely even setting.

\begin{lemme}

Let $\mathfrak{g}$ and $\mathfrak{g}'$ be two Lie superalgebras over $\mathbb K$ 
and
let $\V = \V_{\bar{0}} \oplus \V_{\bar{1}}$ be a finite dimensional  irreducible  
$\mathfrak{g}\oplus \mathfrak {g}'$-module. Then there exists a division superalgebra $\mathbb{D} = \mathbb{D}_{\bar{0}} \oplus \mathbb{D}_{\bar{1}}$ over $\mathbb{K}$, a left-$\mathbb{D}$-module $\U$, and a right-$\mathbb{D}$-module $\W$ such that 
\begin{equation*}
\V \cong \W \otimes_{\mathbb{D}} \U\,,
\end{equation*}
where $\mathfrak{g}$ (resp. $\mathfrak{g}'$) acts on $\U$ (resp. $\W$) by $\mathbb{D}$-linear transformations, and  the $\mathfrak{g}$-module $\U$ is irreducible.

\label{TechnicalLemmaTypeI}

\end{lemme}

\begin{proof}

Let $\U = \U_{\bar{0}} \oplus \U_{\bar{1}} \subseteq \V$ be an irreducible $\mathfrak{g}$-module and set $\mathbb{D} := \End_{\mathfrak{g}}(\U)$. It is easily seen that $\mathbb{D} = \mathbb{D}_{\bar{0}} \oplus \mathbb{D}_{\bar{1}}$ is a division superalgebra over $\mathbb{K}$. 
The $\mathbb{K}$-vector space $\U$ can be seen as a left-$\mathbb{D}$-module, where the action of $\mathbb{D}$ on $\U$ is given by
\begin{equation*}
\D \cdot u = \D(u)\,, \qquad (u \in \U\,, \D \in \mathbb{D})\,.
\end{equation*}
Similarly the $\mathbb{K}$-vector space $\Hom_{\mathfrak{g}}(\U\,, \V)$ has a natural structure of right-$\mathbb{D}$-module, where
\begin{equation*}
\T \cdot \D = \T \circ \D\,, \qquad \left(\T \in \Hom_{\mathfrak{g}}(\U\,, \V)\,, \D \in \mathbb{D}\right)\,.
\end{equation*}
We define a map $\S_{\mathbb{K}}: \Hom_{\mathfrak{g}}(\U\,, \V) \otimes_{\mathbb{K}} \U \to \V$ by
\begin{equation*}
\S_{\mathbb{K}}(\T \otimes u) = \T(u)\,, \qquad (u \in \U, \T \in \Hom_{\mathfrak{g}}(\U\,, \V)\,.
\end{equation*}
The actions of $\mathfrak{g}$ and $\mathfrak{g}'$ on $\Hom_{\mathfrak{g}}(\U\,, \V) \otimes_{\mathbb{K}} \U$ are given by 
\begin{equation}
\X \cdot \T \otimes u = (-1)^{\left|\T\right|\cdot\left|\X\right|} \T \otimes \pi(\X)u\,, \quad \X' \cdot \T \otimes u = \pi(\X') \circ \T \otimes u\,, 
\label{ActionLieG}
\end{equation}
for $u \in \U\,, \T \in \Hom_{\mathfrak{g}}(\U\,, \V)\,, \X \in \mathfrak{g}\,, \X' \in \mathfrak{g}'$.
The action of $X'$ is well-defined since the actions of $\mathfrak{g}$ and $\mathfrak{g}'$ supercommute and thus  $\pi(\X') \circ \T \in \Hom_{\mathfrak{g}}(\U\,, \V)$ for every $\X' \in \mathfrak{g}'$ and $\T \in \Hom_{\mathfrak{g}}(\U\,, \V)$.
The map $\S_\mathbb{K}$ is a 
$\mathfrak{g}\oplus\mathfrak{g}'$-module homomorphism, hence by 
irreducibility of $\V$ this map is also surjective. 
For every $\D \in \mathbb{D}, u \in \U$ and $\T \in \Hom_{\mathfrak{g}}(\U\,, \V)$, we have
\begin{equation*}
\S_{\mathbb{K}}(\T \otimes \D\cdot u) = \T(\D(u)) = \T \circ \D(u) = \S_{\mathbb{K}}(\T \cdot \D \otimes u)\,,
\end{equation*}
and in particular, we obtain a map
\begin{equation*}
\S_{\mathbb{D}}: \Hom_{\mathfrak{g}}(\U\,, \V) \otimes_{\mathbb{D}} \U \to \V\,.
\end{equation*}
The map $\S_{\mathbb{D}}$ is also $\mathfrak{g}\oplus\mathfrak{g}'$-equivariant. It follows that $\V$ is a direct sum of copies of $\U$ and $\Pi \U$, where $\Pi$ is the parity change functor. If $\V=\U^{\oplus a}\oplus\Pi\U^{\oplus b}$ then $\Hom_{\mathfrak g}(\U,\V)$ is isomorphic to $\mathbb D^{\oplus a}\oplus \Pi \mathbb D^{\oplus b}$ as a $\mathbb D$-module. It follows immediately that $\S_\mathbb D$ is injective.
\end{proof}

\begin{rema}

If $\End_{\mathfrak{g}}(\U)_{\bar{1}} \neq \{0\}$, then $\dim_{\mathbb{K}} \Hom_{\mathfrak{g}}(\U\,, \V)_{\bar{0}} = \dim_{\mathbb{K}} \Hom_{\mathfrak{g}}(\U\,, \V)_{\bar{1}}$.

\end{rema}

\section{Division superalgebras over $\mathbb{R}$ and $\mathbb{C}$}

\label{SectionDivisionSuperalgebra}

In \cite{WALL}, Wall gave an explicit description of division superalgebras over $\mathbb{R}$ and $\mathbb C$. We summarise Wall's result in the following proposition.

\begin{prop}

Every division superalgebra $\mathbb{D}$ over $\mathbb{R}$ or $\mathbb{C}$
is isomorphic to one of the following:
\begin{enumerate}
\item[\rm $\mathrm{Cl}_0(\mathbb R)$:] $ \mathbb{R}$ (where $\mathbb K=\mathbb R$),

\item[\rm $\mathrm{Cl}_1(\mathbb R)$:] $\mathbb{R} \oplus \mathbb{R}\cdot \varepsilon$, with $\varepsilon$ odd and $\varepsilon^{2} = -1$(where $\mathbb K=\mathbb R$),

\item[\rm $\mathrm{Cl}_2(\mathbb R)$:] $\mathbb{C} \oplus \mathbb{C}\cdot \varepsilon$, with $\varepsilon$ odd, $\varepsilon^{2} = -1$ and $i\cdot\varepsilon = -\varepsilon\cdot i$ (where $\mathbb K=\mathbb R$),

\item[\rm $\mathrm{Cl}_3(\mathbb R)$:]  $\mathbb{H} \oplus \mathbb{H}\cdot \varepsilon$, with $\varepsilon$ odd, $\varepsilon^{2} = 1$ and $\varepsilon$ commutes with elements of $\mathbb{H}$
(where $\mathbb K=\mathbb R$),

\item[\rm $\mathrm{Cl}_4(\mathbb R)$:] $ \mathbb{H}$
(where $\mathbb K=\mathbb R$),

\item[\rm $\mathrm{Cl}_5(\mathbb R)$:] $ \mathbb{H} \oplus \mathbb{H}\cdot \varepsilon$, with $\varepsilon$ odd, $\varepsilon^{2} = -1$ and $\varepsilon$ commutes with elements of $\mathbb{H}$
(where $\mathbb K=\mathbb R$).

\item[\rm $\mathrm{Cl}_6(\mathbb R)$:] $ \mathbb{C} \oplus \mathbb{C}\cdot \varepsilon$, with $\varepsilon$ odd, $\varepsilon^{2} = 1$ and $i\cdot\varepsilon = -\varepsilon\cdot i$
(where $\mathbb K=\mathbb R$),

\item[\rm $\mathrm{Cl}_7(\mathbb R)$:] $ \mathbb{R} \oplus \mathbb{R}\cdot \varepsilon$, with $\varepsilon$ odd and $\varepsilon^{2} = 1$
(where $\mathbb K=\mathbb R$),
\item[\rm $\mathrm{Cl}_0(\mathbb C)$:] $\mathbb{C}$
(where $\mathbb K=\mathbb R$ or $\mathbb C$),
\item[\rm $\mathrm{Cl}_1(\mathbb C)$:] $\mathbb{C} \oplus \mathbb{C}\cdot \varepsilon$, with $\varepsilon$ odd, $\varepsilon^{2} = 1$ and $i\cdot\varepsilon = \varepsilon\cdot i$
(where $\mathbb K=\mathbb R$ or $\mathbb C$),

\end{enumerate}

\label{ClassificationDivisionSuperalgebra}

\end{prop}

\begin{rema}
The choice of notation $\mathrm{Cl}_k(\mathbb F)$ is due to the fact that the division superalgebra corresponding to  $\mathrm{Cl}_k(\mathbb F)$ is Morita equivalent (as an $\mathbb F$-algebra) to the Clifford algebra with $k$ generators $g_1,\ldots,g_k$ satisfying $g_ag_b=-g_bg_a$ for $1\leq a\neq b\leq k$ and $g_a^2=-1$ for $1\leq a\leq k$.  
\end{rema}



Henceforth, 
for any superalgebra $\mathscr A$  we use $\mathscr Z(\mathscr A)$ to denote the centre
of $\mathscr A$ in the \emph{ungraded} sense (see Definition~\ref{dfn-frstApx}).
\begin{prop}

Let $\mathbb{D}$ be a division superalgebra over $\mathbb{K} = \mathbb{R}$ or $\mathbb{C}$. Suppose that there exists a   superinvolution $\iota$ on the $\mathbb K$-algebra $\End_{\mathbb{D}}(\Y)$, with
\begin{equation*}
\Y = \begin{cases} \mathbb{D}^{a|b} & \text{ if } \mathbb{D}_{\bar{1}} = \{0\} \\ \mathbb{D}^{a} & \text{ if } \mathbb{D}_{\bar{1}} \neq \{0\} \end{cases}
\end{equation*}
Then either  $\mathbb D= \mathbb{D}_{\bar{0}}$,
or $\mathbb{D}\cong \mathrm{Cl}_1(\mathbb C)$ and  
$\mathbb{K} = \mathbb{R}$. Moreover, in the latter case the restriction of the superinvolution  $\iota$ to $\mathscr{Z}(\End_{\mathbb{D}}(\U)) \cong \mathbb{D}$ is equal to either $\iota_1$ or $\iota_2$, where 
\begin{equation}
\iota_{1}(\alpha + \beta\varepsilon) = \bar{\alpha} + i\bar{\beta}\varepsilon\,, \qquad \iota_{2}(\alpha + \beta\varepsilon) = \bar{\alpha} - i\bar{\beta}\varepsilon\,, \qquad (\alpha, \beta \in \mathbb{C})\,.
\label{eq:superinv}
\end{equation}
\label{PropositionDivisionIota}

\end{prop}
\begin{rema}
The explicit description of superinvolutions on 
$\End_{\mathbb D}(\Y)$  is given in
Appendix~\ref{AppendixSuperInvolution}. 
\end{rema}
In order to prove 
Proposition~\ref{PropositionDivisionIota} we begin with the following lemma.  

\begin{lemme}

Let $\mathbb{D}$ be a division superalgebra over $\mathbb{K}$. Then  $\mathbb D$ 
considered as a $\mathbb K$-algebra 
has a  superinvolution if and only if  either $\mathbb{D} = \mathbb{D}_{\bar{0}}$, or  $\mathbb{D} \cong \mathrm{Cl}_1(\mathbb C)$ and $\mathbb{K} = \mathbb{R}$. 
\label{LemmaTechnicalDecember}

\end{lemme}

\begin{proof}
A superinvolution of $\mathbb D$ is an isomorphism between $\mathbb D$ and $\mathbb D^\mathrm{sop}$, where the latter is defined in~\eqref{eq:superopp}. However, 
from the classification of division superalgebras in Proposition~\ref{ClassificationDivisionSuperalgebra} it follows that 
if $1\leq k\leq 7$ and  $k\neq 4$ then \[
\mathrm{Cl}_k(\mathbb R)^\mathrm{sop}
\cong \mathrm{Cl}_{8-k}(\mathbb R)\not\cong \mathrm{Cl}_k(\mathbb R).
\] 
Note that $\mathrm{Cl}_4(\mathbb R)\cong \mathbb H$. Thus, the only division superalgebra $\mathbb D$ with $\mathbb D_{\bar 1}\neq \{0\}$
that can have a superinvolution is $\mathrm{Cl}_1(\mathbb C)$. Note that $\mathrm{Cl}_1(\mathbb C)\cong \mathrm{Cl}_1(\mathbb C)^\mathrm{sop}$ as $\mathbb C$-algebras but the latter isomorphism (say $\theta$) does not yield a superinvolution of $\mathrm{Cl}_1(\mathbb C)$, because it  satisfies $\theta(\varepsilon)=\pm i\varepsilon$, hence we have $\theta^2(\varepsilon)=(\pm i)^2\varepsilon=-\varepsilon\neq \varepsilon$. 
But $\mathrm{Cl}_1(\mathbb C)$ does admit $\mathbb R$-linear involutions: if $\iota$ is an $\mathbb R$-linear involution of 
$\mathrm{Cl}_1(\mathbb C)$ that is nontrivial on $\mathbb C$ then we must have $\iota(\alpha)=\bar \alpha$ for $\alpha\in\mathbb C$, hence  
$
\iota(\alpha \varepsilon) = \iota(\varepsilon)\iota(\alpha) = \iota(\varepsilon)\bar{\alpha}
$
for every $\alpha \in \mathbb{C}$. Furthermore from $(\iota(\varepsilon))^2=-\iota(\varepsilon^2)=-1$ we obtain $\iota(\varepsilon) \in \left\{\pm i\varepsilon\right\}$. It follows immediately that there are exactly two involutions $\iota_{1}$ and $\iota_{2}$ on $\mathbb{D}$ given by
\begin{equation*}
\iota_{1}(\alpha + \beta\varepsilon) = \bar{\alpha} + i\bar{\beta}\varepsilon\,, \qquad \iota_{2}(\alpha + \beta\varepsilon) = \bar{\alpha} - i\bar{\beta}\varepsilon\,, \qquad (\alpha, \beta \in \mathbb{C})\,.
\qedhere\end{equation*}
\end{proof}

\begin{rema}

Let $\mathbb{D} := \mathrm{Cl}_1(\mathbb C)$ and let $\iota_{1}$ and $\iota_{2}$ be the two super-involutions on $\mathbb{D}$ given in Proposition \ref{PropositionDivisionIota}. One can see that $\iota_{1} = \iota_{2} \circ \delta$, where $\delta: \mathbb{D} \to \mathbb{D}$ is the map given by
\begin{equation*}
\delta(\X) = (-1)^{\left|\X\right|}\X\,, \qquad (\X \in \mathbb{D})\,.
\end{equation*}

\label{RemarkDeltaInvolution}

\end{rema}

\begin{proof}[Proof of the Proposition \ref{PropositionDivisionIota}]

If $\mathbb{D} = \mathbb{D}_{\bar{0}}$ then there is nothing to prove. Assume that $\mathbb{D}_{\bar{1}} \neq \{0\}$. The assciative $\mathbb K$-superalgebra  $\End_\mathbb D(\Y)$ is simple. 
Now according to~\cite[Corollary~23]{ELDUQUEVILLA}, 
every simple $\mathbb K$-superalgebra with a superinvolution is isomorphic to $\End_{{\mathbb D'}}(\V')$, where ${\mathbb D'}$ is a 
division superalgebra over $\mathbb K$ that has a superinvolution (say $\iota'$), $\V'$ is a free $\mathbb D'$-module that is equipped with a 
$(\iota',\epsilon)$-superhermitian form (for some $\epsilon\in\{\pm 1\}$), and the involution of $\End_{\mathbb D'}(\V')$ is induced from this superhermitian form as in Propositions~\ref{PropositionC4}
 and~\ref{PropositionC5}. 
Thus 
$\End_{\mathbb D}(\Y)\cong \End_{\mathbb D'}(\V')$ for such $\mathbb D'$ and $\V'$,
 and  in particular  we must have $\mathbb D\cong \mathbb D'$ (the latter assertion follows  from the Isomorphism Theorem proved in \cite[p. 594]{RACINE}).
 By Lemma \ref{LemmaTechnicalDecember},  we also obtain that $\mathbb{D} \cong \mathrm{Cl}_1(\mathbb C)$
and $\mathbb K=\mathbb R$. 
\end{proof}

\section{Classification of irreducible reductive dual pairs in $\mathfrak{spo}$}
\label{SectionClassification}

In this section we prove 
Theorems 
\ref{TheoremeIntroduction1} and 
\ref{TheoremeIntroduction2}. 
We begin by defining families of Lie superalgebras that occur in the classification of dual pairs. 

\begin{nota}

Let $\mathbb{D}$ be a division superalgebra over $\mathbb{K}$ endowed with a superinvolution $\iota$ and let $\U$ be the right-$\mathbb{D}$-module 
\begin{equation*} 
\U := \begin{cases} \mathbb{D}^{n|m} & \text{ if } \mathbb{D}_{\bar{1}} = \{0\} \\ \mathbb{D}^{k} & \text{ if } \mathbb{D}_{\bar{1}} \neq \{0\} \end{cases} \,.
\end{equation*}
Recall that we denote by $\mathfrak{gl}_{\mathbb{D}}(\U)$ the $\mathbb K$-Lie superalgebra of $\mathbb{D}$-linear maps $\X:\U\to \U$, i.e. satisfying $\X(u \cdot \D) = \X(u) \cdot \D$ for every $\D \in \mathbb{D}$ and $u \in \U$. 
If $\mathbb{D}_{\bar{1}} = \{0\}$, we fix a $\mathbb{D}_{\bar{0}}$-basis $\{e_{1}, \ldots, e_{m}\}$ of $\U_{\bar{0}}$ and a $\mathbb{D}_{\bar{0}}$-basis $\{f_{1}, \ldots, f_{m}\}$ of $\U_{\bar{1}}$ and let $\mathscr{B} = \left\{e_{1}, \ldots, e_{n}, f_{1}, \ldots, f_{m}\right\}$.
If $\mathbb{D}_{\bar{1}} \neq \{0\}$, then as a $\mathbb{D}_{\bar{0}}$-module $\U = \mathbb{D}^{k} = \mathbb{D}^{k}_{\bar{0}} \oplus \mathbb{D}^{k}_{\bar{0}} \cdot \varepsilon$, where $(v_{1}, \ldots, v_{k})\cdot\varepsilon = (v_{1}\varepsilon, \ldots, v_{k}\varepsilon)$, with $v_{1}, \ldots, v_{k} \in \mathbb{D}_{\bar{0}}$. We fix a $\mathbb{D}_{\bar{0}}$-basis $\left\{h_{1}, \ldots, h_{k}\right\}$ of $\mathbb{D}^{k}_{\bar{0}}$. In particular, $\left\{h_{1}\cdot\varepsilon, \ldots, h_{k}\cdot\varepsilon\right\}$ is a basis of $\mathbb{D}_{\bar{1}} = \mathbb{D}^{k}_{\bar{0}} \cdot \varepsilon$ and $\mathscr{B}_{1} = \left\{h_{1}, \ldots, h_{k}\right\}$ is a $\mathbb{D}$-basis of $\U$.

\label{NotationsBasisU}

\end{nota}

\begin{defn}
Let $\U$ be as in Notation~\ref{NotationsBasisU}.
Let $\gamma$ be a non-degenerate, homogeneous, $(\iota, \varepsilon)$-superhermitian form on $\U$, with $\varepsilon \in \{\pm 1\}$. We denote by $\mathfrak{g}(\U\,, \gamma)$ the subset of $\mathfrak{gl}_{\mathbb{D}}(\U)$ given by 
\begin{equation*}
\mathfrak{g}(\U\,, \gamma) := \left\{\X \in \mathfrak{gl}_{\mathbb{D}}(\U)\,, \gamma(\X(u)\,, v) + (-1)^{\left|\X\right|\cdot\left|u\right|} \gamma(u\,, \X(v)) = 0\,, u\,, v \in \U\right\}\,.
\end{equation*}

\label{DefinitionOfGUGamma}

\end{defn}

 In Appendix \ref{AppendixExplicitRealization}, we give explicit realizations of $\mathfrak{g}(\U\,, \gamma)$ for every possible $\mathbb{D}\,, \iota$ and $\gamma$.\\

Henceforth we assume that $\E=\E_{\bar 0}\oplus \E_{\bar 1}$ is a finite dimensional $\mathbb Z_2$-graded $\mathbb K$-vector space that is equipped with an even, non-degenerate, $(-1)$-supersymmetric bilinear  form \[
\B:\E\times \E\to\mathbb K.
\]
Furthermore, we assume that $(\mathfrak g,\mathfrak g')$ is an irreducible dual pair in $\mathfrak{spo}(\E,\B)$. 
To proceed with the classification, we consider the Type I and Type II dual pairs separately. \\

Let us first consider the case of  dual pairs of Type II, which are easier to analyze. Thus, let $(\mathfrak{g}\,, \mathfrak{g}')$ be an irreducible reductive dual pair in $\mathfrak{spo}(\E\,, \B)$ of type II. Then it follows from Lemma \ref{TechnicalLemmaTypeI} that there exists a division superalgebra $\mathbb{D}$ over $\mathbb{K}$, a $\mathfrak{g}$-irreducible left-$\mathbb{D}$-module $\U$, a $\mathfrak{g}'$-invariant right-$\mathbb{D}$-module $\W$ such that $\E = (\W \otimes_{\mathbb{D}} \U) \oplus (\W \otimes_{\mathbb{D}} \U)^{*}$ with $\mathfrak{g} \subseteq \mathfrak{gl}_{\mathbb{D}}(\U)$ and $\mathfrak{g}' \subseteq \mathfrak{gl}_{\mathbb{D}}(\W)$.

\begin{prop}

We have $(\mathfrak{g}\,, \mathfrak{g}') = (\mathfrak{gl}_{\mathbb{D}}(\U)\,, \mathfrak{gl}_{\mathbb{D}}(\W))$.

\label{PropositionTypeII}

\end{prop}

\begin{proof}

One can easily see that $\mathfrak{gl}_{\mathbb{D}}(\U)$ is a subalgebra of $\mathfrak{spo}(\E\,, \B)$ and commutes with $\mathfrak{g'}$. In particular, $\mathfrak{gl}_{\mathbb{D}}(\U) \subseteq \mathfrak{g}$. By using that $\mathfrak{g} \subseteq \mathfrak{gl}_{\mathbb{D}}(\U)$, it follows that $\mathfrak{g} = \mathfrak{gl}_{\mathbb{D}}(\U)$. Similarly, we get that $\mathfrak{g}' \subseteq \mathfrak{gl}_{\mathbb{D}}(\W)$ and the proposition follows.
\end{proof}

Next let us consider dual pairs of Type I. We will give the argument in the case   $\mathbb{K} = \mathbb{R}$. 
The case $\mathbb{K}=\mathbb C$ can be treated analogously and is indeed slightly easier because   there are no $\mathbb C$-linear superinvolutions on $\mathbb C\oplus\mathbb C\varepsilon$,
hence we can assume that $\mathbb D=\mathbb D_{\bar 0}$.

In the rest of this section  $(\mathfrak{g}\,, \mathfrak{g}')$ will be an irreducible reductive Type I dual pair in $\mathfrak{spo}(\E\,, \B)$. By Lemma \ref{TechnicalLemmaTypeI}, there exists a factorization $\E \cong \W \otimes_{\mathbb{D}} \U$ with $\mathfrak{g} \subseteq \mathfrak{gl}_{\mathbb{D}}(\U)$ and $\mathfrak{g}' \subseteq \mathfrak{gl}_{\mathbb{D}}(\W)$, where $\mathbb{D} = \End_{\mathfrak{g}}(\U)$. 
By using the form $\B$, we construct a superinvolution $\eta$ on $\End_{\mathbb{R}}(\E)$ (as in Appendix \ref{AppendixSuperInvolution}) by 
\begin{equation*}
\B(\T(u)\,, v) = (-1)^{\left|\T\right| \cdot \left|u\right|} \B(u\,, \eta(\T)(v))\,, \qquad (u\,, v \in \E)\,.
\end{equation*}
In particular, 
\begin{equation}
    \label{eq:eta(x)}
    \eta(\X) = -\X,\qquad    (\X \in \mathfrak{spo}(\E\,, \B)).
\end{equation}
We denote by $\mathscr{A}(\mathfrak{g})$ the subalgebra of $\End_{\mathbb{R}}(\E)$ generated by $\mathfrak{g}$. By using the superized version of the Jacobson Density Theorem (see \cite{RACINE}), we get that $\mathscr{A}(\mathfrak{g}) = \End_{\mathbb{D}}(\U)$. Similarly, we denote by $\mathscr{A}(\mathfrak{g}')$ the subalgebra of $\End_{\mathbb{R}}(\E)$ generated by $\mathfrak{g}'$. Of course $\mathscr{A}(\mathfrak{g}') \subseteq \End_{\mathbb{D}}(\W)$), but a priori we do not know that  equality holds (because we have not proved that $\W$ is  irreducible).

In the next proposition, a pair of ($\mathbb Z_2$-graded) subalgebras $(\mathscr A,\mathscr B)$ of $\End_{\mathbb R}(\E)$ is called a \emph{dual pair} if \[
\mathscr A=\{\T\in\End_\mathbb R(\E)\,:\,[\mathscr B,\T]=0\}\quad \text{and}\quad
\mathscr B=\{\T\in\End_\mathbb R(\E)\,:\,[\mathscr A,\T]=0\},
\]
where $[\cdot,\cdot]$ denotes the standard superbracket of $\End_\mathbb R(\E)$. 
\begin{prop}

The pair of subalgebras $(\End_{\mathbb{D}}(\U), \End_{\mathbb{D}}(\W))$ is a dual pair in   $\End_{\mathbb{R}}(\E)$. 
\label{PropositionEndUEndW}

\end{prop}

\begin{proof}

Let $\{x_{1}, \ldots, x_{d}\}$ be an homogeneous $\mathbb{R}$-basis of $\mathbb{D}$ and $\left\{e_{1}, \ldots, e_{u}\right\}$ (resp. $\left\{f_{1}, \ldots, f_{w}\right\}$) be a $\mathbb{D}$-basis of $\U$ (resp. $\W$) as in Notation \ref{NotationsBasisU}. In particular, 
\begin{equation*}
\left\{f_{i} \otimes x_{k}e_{j}\,, 1 \leq i \leq u\,, 1 \leq k \leq d\,, 1 \leq j \leq w\right\}
\end{equation*} 
is an homogeneous $\mathbb{R}$-basis of $\V$. Note that $f_{i} \otimes x_{k}e_{j} = f_{i}x_{k} \otimes e_{j}$.

For every $1 \leq p\,, q \leq w$, we denote by $\E_{p\,, q}$ the endomorphism of $\End_{\mathbb{D}}(\W)$ such that 
\begin{equation*}
\E_{p, q}(f_{p}) = f_{q}\,, \qquad \E_{p, q}(f_{i}) = 0\,, \qquad (1 \leq i \leq w\,, i \neq p)\,,
\end{equation*}
and let $\T^{p}_{q}: \mathbb{D} \to \End_{\mathbb{D}}(\W)$ be the map given by $\T^{p}_{q}(a) = \E_{p, q}a, a \in \mathbb{D}$ ($\T^{p}_{q}(a)$ can be seen as an element of $\End_{\mathbb{R}}(\V)$ by identifying $\T^{p}_{q}(a)$ with $\T^{p}_{q}(a) \otimes \Id_{\U}$).

Let $\T$ be an element in the supercommutant of $\End_{\mathbb{D}}(\W)$ of parity $\xi$. In particular, $\left[\T\,, \T^{p}_{q}(a)\right] = 0$ for every $1 \leq p\,, q \leq w$ and $a \in \mathbb{D}$.

For every $x \in \mathbb{D}$, $1 \leq i \leq w$ and $1 \leq j \leq u$,
\begin{equation*}
\T(f_{i} \otimes xe_{j}) = \sum\limits_{b = 1}^{w} \sum\limits_{c = 1}^{u} \sum\limits_{k=1}^{d} \alpha_{b\,, c\,, k}(i\,, j\,, x) f_{b} \otimes x_{k}e_{c}\,,
\end{equation*}
with $\alpha_{b\,, c\,, k}(i\,, j\,, x) \in \mathbb{R}$.  In particular, $\T(f_{i} \otimes xe_{j}) = \sum\limits_{b = 1}^{w} \sum\limits_{c = 1}^{u} f_{b} \otimes \beta_{b\,, c}(i\,, j\,, x)e_{c}$
where we define $\beta_{b\,, c}(i\,, j\,, x) = \sum\limits_{k = 1}^{d} \alpha_{b\,, c\,, k}(i\,, j\,, x) x_{k} \in \mathbb{D}$.

On one hand, for every $a \in \mathbb{D}$ and $1 \leq p\,, q \leq w$,
\begin{eqnarray*}
(\T^{p}_{q}(a) \circ \T)(f_{i} \otimes xe_{j}) & = & \T^{p}_{q}(a)\left(\sum\limits_{b = 1}^{w} \sum\limits_{c = 1}^{u} \beta_{b\,, c}(i\,, j\,, x) f_{b} \otimes e_{c}\right) = \sum\limits_{b = 1}^{w} \sum\limits_{c = 1}^{u} \E_{p\,, q}a(f_{b}) \otimes \beta_{b, c}(i, j, x)e_{c} \\
& = & \sum\limits_{b = 1}^{w} \sum\limits_{c = 1}^{u} \delta_{b\,, p} f_{q}a \otimes \beta_{b\,, c}(i\,, j\,, x)e_{c} = \sum\limits_{c = 1}^{u} f_{q} \otimes a\beta_{p\,, c}(i\,, j\,, x)e_{c}\,.
\end{eqnarray*}
On the other hand,
\begin{equation*}
(\T \circ \T^{p}_{q}(a))(f_{i} \otimes xe_{j}) = \delta_{i\,, p} \T(f_{q}a \otimes xe_{j}) = \delta_{i\,, p} \T(f_{q} \otimes axe_{j})  = \delta_{i, p} \sum\limits_{b = 1}^{w} \sum\limits_{c = 1}^{u} f_{b} \otimes \beta_{b\,, c}(q\,, j\,, ax)e_{c}\,.
\end{equation*}
In particular, it implies that if $i \neq p$, $\beta_{p\,, c}(i\,, j\,, x) = 0$ and then 
\begin{equation*}
\T(f_{i} \otimes xe_{j}) = f_{i} \otimes \left(\sum\limits_{c = 1}^{u} \beta_{i\,, c}(i\,, j\,, x)e_{c}\right)\,.
\end{equation*}
Assume that $\mathbb{D} \neq \mathbb{D}_{\bar{0}}$. For every $1 \leq i' \leq w$, we get:
\begin{eqnarray*}
\T(f_{i'} \otimes xe_{j}) & = & \T \circ (\E_{i\,, i'} \otimes \Id_{\U})(f_{i} \otimes xe_{j}) = \E_{i\,, i'} \otimes \Id_{\U} \circ \T (f_{i} \otimes xe_{j}) \\
& = & \E_{i\,, i'} \otimes \Id_{\U}\left(f_{i} \otimes\left(\sum\limits_{c = 1}^{u} \beta_{i\,, c}(i\,, j\,, x)e_{c}\right)\right) = f_{i'} \otimes\left(\sum\limits_{c = 1}^{u} \beta_{i\,, c}(i\,, j\,, x)e_{c}\right)\,,
\end{eqnarray*}  
i.e. $\beta_{i\,, c}(i\,, j\,, x) = \beta_{i'\,, c}(i'\,, j\,, x)$ and then $\T(f_{i} \otimes xe_{j}) = f_{i} \otimes \left(\sum\limits_{c = 1}^{u} \beta_{c}(j\,, x)e_{c}\right)$, where $\beta_{c}(j\,, x) = \beta_{1\,, c}(1\,, j\,, x)$. In particular, there exists a map $\S: \U \to \U$ such that
\begin{equation*}
\T(f_{i} \otimes xe_{j}) = f_{i} \otimes \S(xe_{j})\,, \qquad (1 \leq i \leq w\,, 1 \leq j \leq u\,, x \in \mathbb{D})\,.
\end{equation*}
Note that the map $\S$ is homogeneous and $\left|\S\right| = \left|\T\right| = \xi$. 

Assume now that $\mathbb{D} = \mathbb{D}_{\bar{0}}$. Then for every $1 \leq i' \leq w$, we get:
\begin{eqnarray*}
\T(f_{i'} \otimes xe_{j}) & = & \T \circ (\E_{i\,, i'} \otimes \Id_{\U})(f_{i} \otimes xe_{j}) = (-1)^{\xi \cdot \left|\E_{i\,, i'}\right|} \E_{i\,, i'} \otimes \Id_{\U} \circ \T (f_{i} \otimes xe_{j}) \\
                  & = & (-1)^{\xi(\left|i\right| + \left|i'\right|)}\E_{i\,, i'} \otimes \Id_{\U}\left(f_{i} \otimes\left(\sum\limits_{c = 1}^{u} \beta_{i\,, c}(i\,, j\,, x)e_{c}\right)\right) =  (-1)^{\xi(\left|i\right| + \left|i'\right|)}f_{i'} \otimes\left(\sum\limits_{c = 1}^{u} \beta_{i\,, c}(i\,, j\,, x)e_{c}\right)\,
\end{eqnarray*}   
i.e.
\begin{equation*}
\T(f_{i} \otimes xe_{j}) = \alpha(i\,, \T) f_{i} \otimes \left(\sum\limits_{c = 1}^{u} \beta_{c}(j\,, x)e_{c}\right)\,.
\end{equation*}
where $\alpha(i\,, \T) = \begin{cases} (-1)^{\left|i\right|\cdot \left|\T\right|} & \text{ if } \dim(\U_{\bar{0}}) \neq 0 \\ 1 & \text{ otherwise} \end{cases}$ and $\beta_{c}(j\,, x) = \beta_{1\,, c}(1\,, j\,, x)$. In particular,
\begin{equation*}
\T(f_{i} \otimes xe_{j}) = \alpha(i\,, \T) f_{i} \otimes \S(xe_{j})\,, \qquad (1 \leq i \leq w\,, 1 \leq j \leq u\,, x \in \mathbb{D})\,.
\end{equation*}

To prove that $[\mathcal{C}_{\End_{\mathbb{R}}(\V)}(\End_{\mathbb{D}}(\W))]_{\xi} = \End_{\mathbb{D}}(\U)_{\xi}$, we still need to prove that $\S$ is $\mathbb{D}$-linear, i.e. we have $\S(\D xe_{j}) = (-1)^{\left|\S\right| \cdot \left|\D\right|}\D\S(xe_{j})$ for every $\D, x \in \mathbb{D}$. Then for every $\D \in \mathbb{D}$, $1 \leq i \leq w$ and $1 \leq j \leq u$, we get
\begin{equation*}
\T \circ (\E_{i\,, i}\D \otimes \Id_{\U})(f_{i} \otimes xe_{j}) = \T(f_{i}\D \otimes xe_{j}) = \T(f_{i} \otimes \D xe_{j}) = \Omega(i\,, \T\,, \mathbb{D}) f_{i} \otimes \left(\sum\limits_{c = 1}^{u} \beta_{c}(j\,, \D x)e_{c}\right)
\end{equation*}
and 
\begin{eqnarray*}
(-1)^{\left|\T\right| \cdot \left|\D\right|}\left(\E_{i\,, i}\D \otimes \Id_{\U}\right) \circ \T(f_{i} \otimes xe_{j}) & = & (-1)^{\left|\T\right| \cdot \left|\D\right|} \Omega(i\,, \T\,, \mathbb{D}) \D\E_{i\,, i} \otimes \Id_{\U}\left(f_{i} \otimes \left(\sum\limits_{c = 1}^{u} \beta_{c}(j\,, x)e_{c}\right)\right) \\
    & = & (-1)^{\left|\S\right| \cdot \left|\D\right|} \Omega(i\,, \T, \mathbb{D}) f_{i} \otimes \left(\sum\limits_{c = 1}^{u} d\beta_{c}(j\,, x)e_{c}\right)\,,
\end{eqnarray*}
where $\Omega(i\,, \T\,, \mathbb{D}) = \begin{cases} 1 & \text{ if } \mathbb{D}_{\bar{1}} \neq \{0\} \\ \alpha(i\,, \T) & \text{ if } \mathbb{D}_{\bar{1}} = \{0\} \end{cases}$. In particular $\D\beta_{c}(j\,, x) = (-1)^{\left|\S\right| \cdot \left|\D\right|}\beta_{c}(j\,, \D x)$ and then we get  $[\mathcal{C}_{\End_{\mathbb{R}}(\V)}(\End_{\mathbb{D}}(\W))]_{\xi} = \End_{\mathbb{D}}(\U)_{\xi}$, $\xi \in \mathbb{Z}_{2}$. By symmetry, we get $\mathcal{C}_{\End_{\mathbb{R}}(\V)}(\End_{\mathbb{D}}(\U)) = \End_{\mathbb{D}}(\W)$ and the result follows.
\end{proof}

One can easily see that the superinvolution $\eta$ preserves $\mathscr{A}(\mathfrak{g}) = \End_{\mathbb{D}}(\U)$ and $\mathscr{A}(\mathfrak{g}')$. Moreover, $\eta$ preserves $\mathscr{C}_{\End_{\mathbb{R}}(\V)}(\End_{\mathbb{D}}(\U))$, hence by Proposition \ref{PropositionEndUEndW} it also preserves $\End_{\mathbb{D}}(\W)$.

\begin{rema}
Recall that $\mathscr Z(\mathbb D)$ denotes the centre of $\mathbb D$ in the \emph{ungraded} sense. 
Since $\mathscr{Z}(\mathbb D)\cong  \mathscr Z(\End_{\mathbb D}(\U))$, 
the restriction of $\eta$ to $\mathscr{Z}(\End_{\mathbb{D}}(\U))$ defines a superinvolution on $\mathscr{Z}(\mathbb{D})$. Using Propositions \ref{ClassificationDivisionSuperalgebra} and \ref{PropositionDivisionIota}, there exists a superinvolution $\iota$ on $\mathbb{D}$ such that $\eta_{\big|_{\mathscr{Z}(\mathbb{D})}} = \iota_{\big|_{\mathscr{Z}(\mathbb{D})}}$. Note that 
$\mathscr Z(\mathbb D)=\mathbb D$ when $\mathbb{D}_{\bar 1}\neq 0$.
\end{rema}

By Propositions \ref{PropositionC4} and \ref{PropositionC5},
the superinvolution of $\End_{\mathbb D}(\U)$
that is obtained by restriction of $\eta$ corresponds to a $(\iota,\epsilon_1)$-superhermitian form 
\[
\gamma:\U\times \U\to\mathbb D,
\] 
for some $\epsilon_1\in\{\pm 1\}$.
From \eqref{eq:eta(x)} for $\X\in\mathfrak g$ it follows that $\gamma$ is $\mathfrak g$-invariant, since 
\[
\gamma(\X u,v)=-\gamma(\eta(\X)u,v)=-(-1)^{|\X|\cdot|u|}
\gamma(u,\X v).
\]
Similarly, the restriction of $\eta$ to $\End_{\mathbb D}(\W)$ corresponds to a $(\iota\circ\delta,\epsilon_2)$-superhermitian form \[
\gamma':\W\times \W\to\mathbb D,
\]
for some $\epsilon_2\in\{\pm 1\}$.
Here $\delta$ is as in  Remark \ref{RemarkDeltaInvolution}. 
Note that when $\mathbb D=\mathbb D_{\bar 0}$ we have $\iota\circ\delta=\iota$. The latter form  is also $\mathfrak g'$-invariant.

\begin{prop}

The forms $\gamma$ and $\gamma'$ have the same parity.

\label{PropositionNovember16}

\end{prop}

\begin{proof}

First assume that $\mathbb{D} = \mathbb{D}_{\bar{0}}$. 
Using Propositions \ref{PropositionC4} and \ref{PropositionC5} and part (1) of Remark \ref{remark:B1}, 
we obtain $\mathbb D$-linear maps  $\Psi: \U \to \U^{*}$ and $\Phi: \W \to \W^{*}$ corresponding to the restrictions of $\eta$ to $\End_{\mathbb D}(\U)$ and $\End_{\mathbb D}(\W)$. 
According to Remark \ref{RemarkNovember14AH},
we have 
\[\eta(\Id_{\W} \otimes \T) = 
(-1)^{|\T|\cdot|\Psi|}
(\Id_{\W} \otimes \Psi)^{-1} \circ (\Id_{\W} \otimes \T)^{*} \circ \Id_{\W} \otimes \Psi\,,
\qquad(\T\in\End_{\mathbb D}(\U))\,,
\]
and 
\[\eta(\S \otimes \Id_{\U}) = 
(-1)^{|\S|\cdot |\Phi|}
(\Phi \otimes \Id_{\U})^{-1} \circ (\S \otimes \Id_{\U})^{*} \circ \Phi \otimes \Id_{\U}\,,
\qquad(\S\in\End_{\mathbb D}(\W))\,.
\]
It follows that 
\begin{equation*}\eta(\S \otimes \T) = 
(-1)^{(|\Phi|+\Psi|)(|\S|+|\T|)}
(\Phi \otimes \Psi)^{-1} \circ (\S \otimes \T)^{*} \circ \Phi \otimes \Psi\,, \qquad \left(\T \in \End_{\mathbb{D}}(\U), \S \in \End_{\mathbb{D}}(\W)\right)\,.
\end{equation*}
Similarly, by part (1) of Remark \ref{remark:B1}  the involution $\eta: \End_{\mathbb{R}}(\E) \to \End_{\mathbb{R}}(\E)$ corresponds to an even $\mathbb R$-linear  map $h: \E \to \E^{*}$ such that
\begin{equation*}
\eta(\F) = h^{-1} \circ \F^{*} \circ h\,, \qquad (\F \in \End_{\mathbb{R}}(\E))\,.
\end{equation*}
From the last two relations it follows that for every $\T \in \End_{\mathbb{D}}(\U)$ and $\S \in \End_{\mathbb{D}}(\W)$, we have 
\begin{equation*}
(\Phi \otimes \Psi) \circ h^{-1} \circ (\S \otimes \T)^{*} = 
(-1)^{(|\S|+|\T|)(|\Phi|+|\Psi|)}
(\S \otimes \T)^{*} \circ (\Phi \otimes \Psi) \circ h^{-1}\,.
\end{equation*}
It follows that $\Phi \otimes \Psi \circ h^{-1}$ supercommutes with $\End_{\mathbb{D}}(\U)$ and $\End_{\mathbb{D}}(\W)$. Thus by Proposition
\ref{PropositionEndUEndW} we have 
 $\Phi \otimes \Psi \circ h^{-1} \in \End_{\mathbb{D}}(\U) \cap \End_{\mathbb{D}}(\W) = \mathscr{Z}(\mathbb{D})$. In particular, $\Phi \otimes \Psi \circ h^{-1}$ is even. Because $h$ is even, it follows that $\Phi \otimes \Psi$ is even and using that $\bar{0} = \left|\Phi \otimes \Psi\right| = \left|\Phi\right| + \left|\Psi\right|$, we get that $\left| \Phi\right| = \left|\Psi\right|$ and the proposition follows.

Assume now that $\mathbb{D}_{\bar{1}} \neq \{0\}$. In particular, it follows from Proposition \ref{PropositionDivisionIota} that $\mathbb D \cong \mathbb{C} \oplus \mathbb{C} \cdot \varepsilon$, with $\varepsilon^{2} = 1$ and $i \varepsilon = \varepsilon i$. The superinvolutions induced on $\End_{\mathbb{D}}(\U)$ and on $\End_{\mathbb{D}}(\W)$ by $\eta$  preserve the spaces $\mathscr{Z}(\End_{\mathbb{D}}(\U))$ and $\mathscr{Z}(\End_{\mathbb{D}}(\W))$, where
\begin{equation*}
\mathscr{Z}(\End_{\mathbb{D}}(\U)) = \left\{\X \in \End_{\mathbb{D}}(\U), \X\Y = \Y\X\,, \Y \in \End_{\mathbb{D}}(\U)\right\}\,.
\end{equation*}
Note that $\mathscr{Z}(\End_{\mathbb{D}}(\U)) \cong \mathbb{D} \cong \mathscr{Z}(\End_{\mathbb{D}}(\W))$. According to 
Proposition \ref{PropositionC5}, it is enough to prove that $\eta_{|_{\mathscr{Z}(\End_{\mathbb{D}}(\U))}} = \eta_{|_{\mathscr{Z}(\End_{\mathbb{D}}(\W))}} \circ \delta$, where $\delta: \mathbb{D} \to \mathbb{D}$ is defined in Remark \ref{RemarkDeltaInvolution}.

For every $\D \in \mathbb{D}$, we denote by $\A_{\D}$ (resp. $\F_{\D}$) the endomorphism of $\End_{\mathbb{D}}(\U)$ (resp. $\End_{\mathbb{D}}(\W)$) defined by 
\begin{equation*}
\A_{\D}(u) = (-1)^{\left|\D\right| \cdot \left|u\right|} \D \cdot u\,, \qquad \F_{\D}(w) = w \cdot \D\,, \qquad (u \in \U, w \in \W)\,.
\end{equation*}
In particular, the map
\begin{equation*}
\mathbb{D} \ni \D \to \A_{\D} \in {\mathscr{Z}}(\End_{\mathbb{D}}(\U))
\end{equation*}
is an isomorphism. We denote by $\widetilde{\A}_{\D} = \Id_{\W} \otimes \A_{\D}$ and $\widetilde{\F}_{\D} = \F_{\D} \otimes \Id_{\U}$ the corresponding elements of $\End_{\mathbb{R}}(\W \otimes_{\mathbb{D}} \U)$. As explained in Equation \eqref{ActionLieG}, we have:
\begin{equation*}
\widetilde{\A}_{\D}(w \otimes u) = (-1)^{\left|\D\right| \cdot \left|w\right|} w \otimes \A_{\D}(u)\,, \qquad \widetilde{\F}_{\D}(w \otimes u) = \F_{\D}(w) \otimes u\,, \qquad (u \in \U, w \in \W)\,.
\end{equation*}
As explained in Proposition \ref{PropositionDivisionIota}, there exists $a, b \in \{\pm 1\}$ such that $\eta(\widetilde{\A}_{\varepsilon}) = \widetilde{\A}_{ia\varepsilon}$ and $\eta(\widetilde{\F}_{\varepsilon}) = \widetilde{\F}_{ib\varepsilon}$. 

For every $w_{1}, w_{2} \in \W$ and $u_{1}, u_{2} \in \U$, we get:
\begin{equation*}
\B(\widetilde{\A}_{\varepsilon}(w_{1} \otimes u_{1})\,, w_{2} \otimes u_{2}) = (-1)^{\left|w_{1} \otimes u_{1}\right|} \B(w_{1} \otimes u_{1}\,, \eta(\widetilde{\A}_{\varepsilon})(w_{2} \otimes u_{2}))\,,
\end{equation*}
i.e. 
\begin{equation*}
(-1)^{\left|w_{1}\right| + \left|u_{1}\right|} \B(w_{1} \otimes \varepsilon u_{1}\,, w_{2} \otimes u_{2}) = (-1)^{\left|u_{1}\right| + \left|w_{1}\right| + \left|u_{2}\right| + \left|w_{2}\right|} \B(w_{1} \otimes u_{1}\,, w_{2} \otimes ia\varepsilon u_{2})\,.
\end{equation*}
Similarly, $\B(\widetilde{\F}_{\varepsilon}(w_{1} \otimes u_{1})\,, w_{2} \otimes u_{2}) = (-1)^{\left|w_{1} \otimes u_{1}\right|} \B(w_{1} \otimes u_{1}\,, \eta(\widetilde{\F}_{\varepsilon})(w_{2} \otimes u_{2}))$, i.e.
\begin{equation*}
\B(w_{1}\cdot\varepsilon \otimes u_{1}\,, w_{2} \otimes u_{2}) = (-1)^{\left|u_{1}\right| + \left|w_{1}\right|} \B(w_{1}\otimes u_{1}\,, w_{2}ib\varepsilon \otimes u_{2})\,,
\end{equation*}
and using that $\B(w_{1} \otimes \varepsilon u_{1}\,, w_{2} \otimes u_{2}) = \B(w_{1}\varepsilon \otimes u_{1}\,, w_{2} \otimes u_{2})$ and \[
\B(w_{1} \otimes u_{1}\,, w_{2} \otimes ia\varepsilon u_{2}) = \B(w_{1} \otimes u_{1}\,, w_{2} ia\varepsilon \otimes u_{2}) = a\B(w_{1} \otimes u_{1}\,, w_{2} i\varepsilon \otimes u_{2}),
\]it follows that
\begin{equation}
a(-1)^{\left|u_{2}\right| + \left|w_{2}\right|} 
\B(w_{1} \otimes \varepsilon u_{1}\,, w_{2} \otimes u_{2})
= b(-1)^{\left|u_{1}\right| + \left|w_{1}\right|}
\B(w_{1} \otimes \varepsilon u_{1}\,, w_{2} \otimes u_{2})\,.
\label{EquationAandB}
\end{equation}
The form $\B$ is even, in particular, $\B(w_{1} \otimes \varepsilon u_{1}\,, w_{2} \otimes u_{2})  \neq 0$ if and only if $\left|u_{1}\right| + \left|u_{2}\right| + \left|w_{1}\right| + \left|w_{2}\right| = \bar {1}$. From Equation \eqref{EquationAandB}, we get that $a = -b$ and the proposition follows.
\end{proof}

Let $\tau: \W \otimes_{\mathbb{R}} \U \times \W \otimes_{\mathbb{R}} \U \to \mathbb{R}$ be the bilinear form given by
\begin{equation*}
\tau(w_{1} \otimes u_{1}\,, w_{2} \otimes u_{2}) = (-1)^{\left|u_{1}\right|\cdot\left|w_{2}\right| + \left|w_{1}\right|\cdot\left|w_{2}\right|} \Re(\gamma'(w_{2}\,, w_{1})\gamma(u_{1}\,, u_{2}))\,.
\end{equation*}
\begin{rema}
For $\mathbb D:=\mathbb C\oplus\mathbb C\varepsilon$ we define 
${\Re}(a+b\varepsilon):={\Re}(a)$.
\end{rema}\begin{lemme}

For every $\D \in \mathbb{D}$, we have
\begin{equation*}
\tau(w_{1} \otimes \D u_{1}\,, w_{2} \otimes u_{2}) = \tau(w_{1}\D \otimes u_{1}\,, w_{2} \otimes u_{2})\qquad\text{and} \qquad \tau(w_{1} \otimes u_{1}\,, w_{2} \otimes \D u_{2}) = \tau(w_{1} \otimes u_{1}\,, w_{2}\D \otimes u_{2})\,.
\end{equation*}
In particular, the form $\tau$ factors through $\E = \W \otimes_{\mathbb{D}} \U$.

\label{LemmaNovember3}

\end{lemme}

\begin{proof}
We have
\begin{eqnarray*}
\tau(w_{1} \otimes \D u_{1}\,, w_{2} \otimes u_{2}) & = & (-1)^{\left|\D u_{1}\right|\cdot \left|w_{2}\right| + \left|w_{1}\right|\cdot \left|w_{2}\right|} \Re(\gamma'(w_{2}\,, w_{1})\gamma(\D u_{1}\,, u_{2})) \\
                                             & = & (-1)^{\left|\D u_{1}\right|\cdot \left|w_{2}\right| + \left|w_{1}\right|\cdot \left|w_{2}\right|} \Re(\gamma'(w_{2}\,, w_{1})\D\gamma(u_{1}\,, u_{2})) \\
                                             & = & (-1)^{\left|\D u_{1}\right|\cdot \left|w_{2}\right| + \left|w_{1}\right|\cdot \left|w_{2}\right|} (-1)^{\left|u_{1}\right|\cdot \left|w_{2}\right| + \left|w_{1}\D\right|\cdot \left|w_{2}\right|} \tau(w_{1}\D \otimes u_{1}\,, w_{2} \otimes u_{2}) 
\end{eqnarray*}
Similarly, we have
\begin{align*}
\tau(w_{1} \otimes &u_{1} \,, w_{2} \otimes \D u_{2})
=  (-1)^{\left|u_{1}\right|\cdot \left|w_{2}\right| + \left|w_{1}\right|\cdot \left|w_{2}\right|} \Re(\gamma'(w_{2}\,, w_{1})\gamma(u_{1}\,, \D u_{2})) \\
                          = & (-1)^{\left|u_{1}\right|\cdot \left|w_{2}\right| + \left|w_{1}\right|\cdot \left|w_{2}\right|} (-1)^{\left|\D\right|\cdot \left|u_{2}\right|}\Re(\gamma'(w_{2}\,, w_{1})\gamma(u_{1}\,, u_{2})\iota(\D)) \\
                          = & (-1)^{\left|u_{1}\right|\cdot \left|w_{2}\right| + \left|w_{1}\right|\cdot \left|w_{2}\right|} (-1)^{\left|\D\right|\cdot \left|u_{2}\right|} (-1)^{\left|\D\right|} \Re(\gamma'(w_{2}\,, w_{1})\gamma(u_{1}\,, u_{2})\iota_{2}(\D)) \\
                          = & (-1)^{\left|u_{1}\right|\cdot \left|w_{2}\right| + \left|w_{1}\right|\cdot \left|w_{2}\right|} (-1)^{\left|\D\right|\cdot \left|u_{2}\right|} (-1)^{\left|\D\right|} (-1)^{\left|\D\right|\cdot \left|w_{2}\right|} \Re(\gamma'(w_{2}\D\,, w_{1})\gamma(u_{1}\,, u_{2})) \\
                          = & (-1)^{\left|u_{1}\right|\cdot \left|w_{2}\right| + \left|w_{1}\right|\cdot \left|w_{2}\right|} (-1)^{\left|\D\right|\cdot \left|u_{2}\right|} (-1)^{\left|\D\right|} (-1)^{\left|\D\right|\cdot \left|w_{2}\right|} (-1)^{\left|u_{1}\right|\cdot \left|w_{2}\D\right| + \left|w_{1}\right|\cdot \left|w_{2}\D\right|} \Re(\gamma'(w_{2}\D\,, w_{1})\gamma(u_{1}\,, u_{2})) \\ 
                          = & (-1)^{\left|\D\right|\cdot \left(\left|u_{1}\right|+\left|u_{2}\right|+\left|w_{1}\right|+\left|w_{2}\right|+\bar{1}\right)} \tau(w_{1} \otimes u_{1}\,, w_{2}\D \otimes u_{2}).
\end{align*}
The forms $\gamma$ and $\gamma'$ have the same parity. In particular, $\Re(\gamma'(w_{2}\,, w_{1})\gamma(u_{1}\,, \D u_{2}))$ is non-zero only if 
\begin{equation*}
\begin{cases} \left|w_{2}\right| & = \left|w_{1}\right| \\ \left|u_{1}\right| & = \left|\D u_{2}\right| \end{cases} \qquad \text{ or } \qquad \begin{cases} \left|w_{2}\right| & = \left|w_{1}\right| + \bar{1} \\ \left|u_{1}\right| & = \left|\D u_{2}\right| + \bar{1} \end{cases} 
\end{equation*}
i.e.
\begin{equation}
\begin{cases} \left|w_{2}\right| & = \left|w_{1}\right| \\ \left|u_{1}\right| & = \left|u_{2}\right| + \left|\D\right| \end{cases} \qquad \text{ or } \qquad \begin{cases} \left|w_{2}\right| & = \left|w_{1}\right| + \bar{1} \\ \left|u_{1}\right| & = \left|u_{2}\right| + \left|\D\right| + \bar{1} \end{cases} 
\label{Equation111111}
\end{equation}                      
If $\left|\D\right| = \bar{0}$ then $(-1)^{\left|\D\right|\cdot \left(\left|u_{1}\right|+\left|u_{2}\right|+\left|w_{1}\right|+\left|w_{2}\right|+\bar{1}\right)} = 1$. If $\left|\D\right| = \bar{1}$, then whenever we have  \[\Re(\gamma'(w_{2}\,, w_{1})\gamma(u_{1}\,, \D u_{2})) \neq 0,
\]
it follows from Equation \eqref{Equation111111} that $\left|u_{1}\right|+\left|u_{2}\right|+\left|w_{1}\right|+\left|w_{2}\right|+\bar{1} = \bar{0}$, i.e. $(-1)^{\left|\D\right|\cdot \left(\left|u_{1}\right|+\left|u_{2}\right|+\left|w_{1}\right|+\left|w_{2}\right|+\bar{1}\right)} = 1$, and the Lemma follows.
\end{proof}

Recall that $\gamma$ is $(\iota,\epsilon_1)$-superhermitian and
$\gamma'$ is $(\iota\circ\delta,\epsilon_2)$-superhermitian. Of course when $\mathbb D_{\bar 1}=0$ we have $\iota\circ\delta=\iota$.
\begin{lemme}

The form $\tau$ is even and $\epsilon_{1}\epsilon_{2}$-supersymmetric.

\label{LemmaNovember31}

\end{lemme}

\begin{proof}

Using Proposition \ref{PropositionNovember16}, it follows that the form $\tau$ is even. Moreover, for all homogeneous elements $u_{1}, u_{2} \in \U$ and $w_{1}, w_{2} \in \W$, we have
\begin{align*}
 \tau(w_{1} &\otimes u_{1}\,, w_{2} \otimes u_{2}) = (-1)^{\left|u_{1}\right|\cdot \left|w_{2}\right| + \left|w_{1}\right|\cdot \left|w_{2}\right|} \Re(\gamma'(w_{2}\,, w_{1})\gamma(u_{1}\,, u_{2})) \\
             = & (-1)^{\left|u_{1}\right|\cdot \left|w_{2}\right| + \left|w_{1}\right|\cdot \left|w_{2}\right|} \Re(\iota(\gamma'(w_{2}\,, w_{1})\gamma(u_{1}\,, u_{2}))) \\
             = & 
            (-1)^{\left|u_{1}\right|\cdot \left|w_{2}\right| + \left|w_{1}\right|\cdot \left|w_{2}\right|} (-1)^{(\left|u_{1}\right|+\left|u_{2}\right|+|\gamma|)\cdot (\left|w_{1}\right|+\left|w_{2}\right|+|\gamma'|)} \Re(\iota(\gamma(u_{1}\,, u_{2}))\iota(\gamma'(w_{2}\,, w_{1}))) \\ 
             = & (-1)^{\left|u_{1}\right|\cdot \left|w_{2}\right| + \left|w_{1}\right|\cdot \left|w_{2}\right|} (-1)^{(\left|u_{1}\right|+\left|u_{2}\right|+|\gamma|)\cdot (\left|w_{1}\right|+\left|w_{2}\right|+|\gamma'|)} (-1)^{\left|w_{1}\right|+\left|w_{2}\right|+|\gamma'|} \Re(\iota(\gamma(u_{1}\,, u_{2}))\iota\circ\delta(\gamma'(w_{2}\,, w_{1}))) \\ 
             = & \epsilon_{1}\epsilon_{2}(-1)^{\left|u_{1}\right|\cdot \left|w_{2}\right| + \left|w_{1}\right|\cdot \left|w_{2}\right|} (-1)^{(\left|u_{1}\right|+\left|u_{2}\right|+|\gamma|)\cdot (\left|w_{1}\right|+\left|w_{2}\right|+|\gamma'|)} (-1)^{\left|w_{1}\right|+\left|w_{2}\right|+|\gamma'|+\left|u_{1}\right|\cdot \left|u_{2}\right|+\left|w_{1}\right|\cdot \left|w_{2}\right|} \Re(\gamma(u_{2}\,, u_{1})\gamma'(w_{1}\,, w_{2})) \\
             = & \epsilon_{1}\epsilon_{2} (-1)^{(\left|u_{1}\right|+\left|w_{2}\right|)\cdot (\left|w_{1}\right|+\left|u_{2}\right|) + \left|w_{1}\right|+\left|w_{2}\right|} 
            (-1)^{|\gamma|(|w_1|+|w_2|+|u_1|+|u_2|)}\tau(w_{2} \otimes u_{2}\,, w_{1} \otimes u_{1})
\end{align*}            
We want to prove that we can replace \[(-1)^{(\left|u_{1}\right|+\left|w_{2}\right|)\cdot (\left|w_{1}\right|+\left|u_{2}\right|) + \left|w_{1}\right|+\left|w_{2}\right|}(-1)^{|\gamma|(|w_1|+|w_2|+|u_1|+|u_2|)}
\] by $(-1)^{(\left|u_{1}\right|+\left|w_{1}\right|)\cdot (\left|u_{2}\right|+\left|w_{2}\right|)}$. It is enough to prove that these two signs are equal whenever $\tau(w_{1} \otimes u_{1}\,, w_{2} \otimes u_{2}) \neq 0$. One can see that $\tau(w_{1} \otimes u_{1}\,, w_{2} \otimes u_{2}) \neq 0$ only if 
\begin{equation*}
\begin{cases} \left|w_{2}\right| & = \left|w_{1}\right| \\ \left|u_{1}\right| & = \left|u_{2}\right| \end{cases} \qquad \text{ or } \qquad \begin{cases} \left|w_{2}\right| & = \left|w_{1}\right| + \bar{1} \\ \left|u_{1}\right| & = \left|u_{2}\right| + \bar{1} \,.\end{cases} 
\end{equation*}
In the first case, we get:
\begin{equation*}
(-1)^{(\left|u_{1}\right|+\left|w_{2}\right|)\cdot (\left|w_{1}\right|+\left|u_{2}\right|) + \left|w_{1}\right|+\left|w_{2}\right|} = (-1)^{(\left|u_{1}\right|+\left|w_{1}\right|)\cdot (\left|u_{2}\right|+\left|w_{2}\right|)}
\end{equation*}
and in the second case
\begin{eqnarray*}
& & (-1)^{(\left|u_{1}\right|+\left|w_{2}\right|)(\left|w_{1}\right|+\left|u_{2}\right|) + \left|w_{1}\right|+\left|w_{2}\right|} = (-1)^{(\left|u_{1}\right|+\left|w_{1}\right|+\bar{1})(\left|w_{2}\right|+\left|u_{2}\right|+\bar{1}) + \left|w_{1}\right|+\left|w_{1}\right| + \bar{1}}. \\
\end{eqnarray*}
Furthermore in both cases we have $(-1)^{|\gamma|(|w_1|+|w_2|+|u_1|+|u_2|)}=1$
and the lemma follows.
\end{proof}

\begin{lemme}
There exists $c\in\mathscr Z(\mathbb D)_{\bar 0}$   such that 
$\iota(c)=\pm c$ and 
$\tau=c\B$.
\label{lemma:tauB}
\end{lemme}
\begin{proof}
Let $\eta_\tau$ denote the involution of $\End_\mathbb R(\E)$ 
that corresponds to $\tau$. 
Let $\X\in \End_\mathbb D(\W)$. We denote 
$\eta(\X\otimes 1)$ by $\overline \eta(X)$. Note that $\overline \eta$ is the involution of $\End_\mathbb D(\W)$ that corresponds to $\gamma'$. We have

\begin{eqnarray*}
\tau((\X\otimes 1)(w_1\otimes u_1),w_2\otimes u_2)&=&
\tau(\X w_1\otimes u_1,w_2\otimes u_2)\\
&=&
(-1)^{|u_1|\cdot |w_2|+(|\X|+|w_1|)\cdot |w_2|}
\Re(\gamma'(w_2,\X w_1)\gamma(u_1,u_2))\\
&=&
(-1)^{|u_1|\cdot |w_2|+(|\X|+|w_1|)\cdot |w_2|}
(-1)^{|\X|\cdot |w_2|}
\Re(\gamma'(\overline \eta(\X)w_2,w_1)\gamma(u_1,u_2))\\
&=&
(-1)^{|u_1|\cdot |w_2|+|w_1|\cdot |w_2|}
\Re(\gamma'(\overline \eta(\X)w_2,w_1)\gamma(u_1,u_2)).
\end{eqnarray*}

We also have
\begin{align*}
(-1)^{|\X|\cdot(|w_1|+|u_1|)}
&\tau(w_1\otimes u_1,\overline \eta(\X)w_2\otimes u_2)\\
&=
(-1)^{|\X|\cdot(|w_1|+|u_1|)}
(-1)^{|u_1|\cdot (\X|+|w_2|)+|w_1|\cdot (|\X|+|w_2|)}
\Re(\gamma'(\overline \eta(\X)w_2,w_1)\gamma(u_1,u_2))\\
&=
(-1)^{|u_1|\cdot |w_2|+|w_1|\cdot |w_2|}
\Re(\gamma'(\overline \eta(\X)w_2,w_1)\gamma(u_1,u_2)).
\end{align*}
From the above calculation it follows that 
$\eta_\tau(\X\otimes 1)=\overline \eta(\X)\otimes 1$, in particular $\eta$ and $\eta_\tau$ agree on $\End_\mathbb D(\W)$. By a similar calculation we can verify that $\eta$ and $\eta'$ also agree on $\End_\mathbb D(\U)$. 
It follows that $\theta:=\eta^{-1}\circ \eta_\tau$ is a $\mathscr Z(\mathbb D)_{\bar 0}$-linear  automorphism of $\End_\mathbb R(\E)$ that fixes $\End_\mathbb D(\W)$ and $\End_\mathbb D(\U)$ pointwise. By
Proposition \ref{prp:Skolem}
 there exists $\P\in\End_\mathbb R(\E)$ such that 
\[
\theta(\X)=(-1)^{|\P|\cdot |X|}\P\X\P^{-1}.
\]
Since both $\tau$ and $\B$ are even, it follows that $|\P|=\bar 0$. From Proposition \ref{PropositionEndUEndW} it also follows that $\P\in \End_\mathbb D(\U)\cap \End_\mathbb D(\W)$, hence in conclusion we obtain $\P=c\Id_\E$ for some $c\in\mathscr Z(\mathbb D)_{\bar 0}$. Finally, note that $\tau$ can be $\pm 1$-supersymmetric 
only if $\iota(c)=\pm c$.
\end{proof}

Let $\tilde{\gamma}: \U \times \U \to \mathbb{D}$ be the form on $\U$ given by
\begin{equation*}
\tilde{\gamma}(u_{1}\,, u_{2}) = \gamma(c^{-1}u_{1}\,, u_{2})\,, \qquad (u_{1}\,, u_{2} \in \U)\,,
\end{equation*}
with $c\in\mathscr Z(\mathbb D)_{\bar 0}$ as in Lemma \ref{lemma:tauB}. 
The following lemma is now immediate. 
\begin{lemme}

The form $\tilde{\gamma}$ is non-degenerate, $\mathfrak{g}$-invariant and $(\iota, \tilde{\epsilon_1})$-superhermitian where
$\tilde{\epsilon_1}:=\iota(c)c^{-1}\epsilon_1$.
Furthermore,
\begin{equation*}
\B(w_{1} \otimes u_{1}\,, w_{2} \otimes u_{2}) = (-1)^{\left|u_{1}\right|\cdot\left|w_{2}\right| + \left|w_{1}\right|\cdot\left|w_{2}\right|} \Re(\gamma'(w_{2}\,, w_{1})\tilde{\gamma}(u_{1}\,, u_{2}))\,.
\end{equation*}
Moreover, $\tilde{\epsilon_{1}}\epsilon_{2} = -1$.
\label{lem:Btildegamma}
\end{lemme}

\begin{coro}
In the notation of Definition \ref{DefinitionOfGUGamma}, we get $(\mathfrak{g}\,, \mathfrak{g}') = (\mathfrak{g}(\U\,, \tilde{\gamma})\,, \mathfrak{g}(\W\,, \gamma'))$.

\label{cor:gUgamma}
\end{coro}

\begin{proof}
By Lemma \ref{lem:Btildegamma} we have $\mathfrak g\subseteq \mathfrak g(\U,\tilde{\gamma})$ and $\mathfrak g'\subseteq \mathfrak g(\W,\gamma')$. An argument analogous to the Type II case proves that the latter pair of Lie superalgebras forms a dual pair. 
    \end{proof}

We can now complete the proofs of the classification theorems of dual pairs. 
\begin{proof}[Proof of Theorems \ref{TheoremeIntroduction2} and \ref{TheoremeIntroduction1}.]
Theorem \ref{TheoremeIntroduction2} follows from Proposition \ref{PropositionTypeII} and the computations in Appendix C.I. Theorem \ref{TheoremeIntroduction1} follows from Corollary \ref{cor:gUgamma} and the computations in Appendix C.II. More precisely, we choose $\tilde{\gamma}$ to be a
$(\iota,\tilde{\epsilon_1})$-superhermitian form. Then we should choose 
${\gamma'}$ to be of the same parity as $\tilde{\gamma}$ and 
 $(\iota\circ\delta,-\tilde{\epsilon_1})$-superhermitian (and of course $\iota\circ\delta=\iota$ when $\mathbb D_{\bar 1}=0$).
By switching the roles of $\U$ and $\W$ if necessary, we can assume that $\tilde{\epsilon_1}=-1$. Finally, from Proposition \ref{PropositionC5} it follows that when $\mathbb D_{\bar 1}\neq 0$ both $\tilde{\gamma}$ and $\gamma'$ are even. 
\end{proof}

\begin{rema}

One can see that for every irreducible reductive complex dual pair $(\mathfrak{g}\,, \mathfrak{g}')$ of type II in $\mathfrak{spo}(\E_{\mathbb{C}}, \B_{\mathbb{C}})$, there exists an irreducible reductive dual pair $(\mathfrak{h}\,, \mathfrak{h}')$ in $\mathfrak{spo}(\E\,, \B)$ such that $(\mathfrak{h}_{\mathbb{C}}\,, \mathfrak{h}'_{\mathbb{C}}) = (\mathfrak{g}\,, \mathfrak{g}')$.
Indeed, $\mathfrak{u}(p\,, q| r\,, s\,, \mathbb{C}) \otimes_{\mathbb{R}} \mathbb{C} \cong \mathfrak{gl}(p+q|r+s\,, \mathbb{C})$ and  $\mathfrak{q}(p\,, q) \otimes_{\mathbb{R}} \mathbb{C} \cong \mathfrak{q}(p+q\,, \mathbb{C})$.

\end{rema}

\section{Dual pairs in an orthosymplectic Lie supergroup}

\label{SectionLieSupergroups}

 As explained in \cite{FIORESI}, there is an equivalence of categories between the category of Lie supergroups and the category of Harish-Chandra pairs, whose definition we recall below.
\begin{defn}
Let $\mathbb{K} = \mathbb{R}$ or  $\mathbb{C}$.
A Harish-Chandra pair is a pair $\mathscr{G} := (\G\,, \mathfrak{g})$, where $\G$ is a $\mathbb{K}$-Lie group, $\mathfrak{g} = \mathfrak{g}_{\bar{0}} \oplus \mathfrak{g}_{\bar{1}}$ is a $\mathbb{K}$-Lie superalgebra, and the following hold:
\begin{enumerate}
\item $\mathfrak{g}_{\bar{0}}$ is the Lie algebra of $\G$\,,
\item $\G$ acts on $\mathfrak{g}$ smoothly by $\mathbb{K}$-linear transformations\,,
\item The differential of the action $\G \curvearrowright \mathfrak{g}$ is identical  to the adjoint action $\mathfrak{g}_{\bar{0}} \curvearrowright \mathfrak{g}$.
\end{enumerate}

\end{defn}
In this section a Lie supergroup will be synonymous to a Harish-Chandra pair. 
\begin{rema}
Let $\mathscr{G} = (\G\,, \mathfrak{g})$ be a Lie supergroup. To simplify, we will denote by $\Ad$ the action of $\G$ on $\mathfrak{g}$.
 A Lie sub-supergroup is a Harish-Chandra pair $\mathscr{H} := (\H\,, \mathfrak{h})$ where $\H$ is a Lie subgroup of $\G$\,, $\mathfrak{h}$ is a Lie sub-superalgebra of $\mathfrak{g}$ and the Lie algebra of  $\H$ is  $\mathfrak{h}_{\bar 0}$. The supercommutant of $\mathscr{H}$ in $\mathscr{G}$ is the Harish-Chandra pair $\mathscr{C} := (\C\,, \mathfrak{c})$ where 
\begin{equation*}
\C := \left\{c \in \mathcal{C}_{\G}(\H)\,, \Ad(c)_{|_{\mathfrak{h}_{\bar{1}}}} = \Id_{|_{\mathfrak{h}_{\bar{1}}}}\right\}\qquad\text{and}\qquad \mathfrak{c} = \mathcal{C}_{\mathfrak{g}}(\mathfrak{h})\,.
\end{equation*}
\end{rema}

\begin{exem}

Let $\V = \V_{\bar{0}} \oplus \V_{\bar{1}}$ be a finite dimensional
$\mathbb Z_2$-graded vector space over $\mathbb{C}$. 
\begin{itemize}
\item Let $\G = \GL(\V_{\bar{0}}) \times \GL(\V_{\bar{1}})$ and $\mathfrak{g} = \mathfrak{gl}(\V)$. Here we see $\G$ and $\mathfrak{g}$ as subspaces of $\End(\V)$. Let $\Ad: \G \to \GL(\mathfrak{g})$ be the action defined by
\begin{equation*}
\Ad(g)(\X) = g\X g^{-1}\,, \qquad (g \in \G\,, \X \in \mathfrak{g})\,.
\end{equation*}
Then $(\G\,, \mathfrak{g})$ is a Harish-Chandra pair. We denote by $\textbf{GL}(\V)$ the corresponding Lie supergroup.
\item Let $\B$ be an even, non-degenerate, $(-1)$-supersymmetric form on $\V$ and let $\B_{i} = \B_{|_{\V_{\bar{i}} \times \V_{\bar{i}}}}$. Let $\H = \Sp(\V_{\bar{0}}\,, \B_{0}) \times \O(\V_{\bar{1}}\,, \B_{1})$ and $\mathfrak{h} = \mathfrak{spo}(\V\,, \B)$. One can easily check that $\Ad(\H)(\mathfrak{h}) \subseteq \mathfrak{h}$, i.e. $(\H\,, \mathfrak{h})$ is a Lie sub-supergroup of $\textbf{GL}(\V)$. We denote by $\textbf{SpO}(\V)$ (or $\textbf{SpO}(\V\,, \B)$) the corresponding Lie supergroup: it is the orthosymplectic Lie supergroup. Similarly, if the form $\B$ on $\V$ is $1$-supersymmetric, we denote by $\textbf{OSp}(\V)$ (or $\textbf{OSp}(\V\,, \B)$) the corresponding Lie supergroup.
\item Let $\B_{1}$ be an odd, non-degenerate, $1$-supersymmetric form on $\V$. Let $\L$ (resp. $\mathfrak{l}$) be the subspace of $\G$ (resp. $\mathfrak{g}$) given by
\begin{equation*}
\L := \left\{g \in \G\,, \B_{1}(g(v_{1})\,, g(v_{2})) = \B_{1}(v_{1}\,, v_{2})\,, v_{1}\,, v_{2} \in \V\right\}\,, 
\end{equation*}
\begin{equation*}
\mathfrak{l} = \left\{\X \in \mathfrak{g}\,, \B_{1}(\X(v_{1})\,, v_{2}) + (-1)^{\left|\X\right| \cdot \left|v_{1}\right|}\B_{1}(v_{1}\,, \X(v_{2})) = 0\right\}\,.
\end{equation*}
We denote by $\textbf{P}(\V)$ (or $\textbf{P}(\V\,, \B_{1})$) the corresponding Lie supergroup: it is the periplectic Lie supergroup. Note that if the form $\B_{1}$ is $-1$-supersymmetric, the corresponding supergroup is isomorphic to ${\textbf{P}}(\V)$.
\end{itemize}

\label{ExampleAugust2}

\end{exem}

\begin{rema}
We keep the notations of Example \ref{ExampleAugust2}. 
\begin{enumerate}
\item Assume that $\dim_{\mathbb{C}}(\V) = (m|n)$. Fix a basis $\left\{v_{1}\,, \ldots\,, v_{m}\right\}$ (resp. $\left\{v_{\bar{1}}\,, \ldots\,, v_{\bar{n}}\right\}$) of $\V_{\bar{0}}$ (resp. $\V_{\bar{1}}$). Then the elements of $\G$ and $\mathfrak{g}$ can be seen as elements of $\Mat(m|n)$. In this case, $\G$ is identified with $\GL(m) \times \GL(n)$ and $\mathfrak{g}$ with $\mathfrak{gl}(m|n)$ and we use the notation $\textbf{GL}(m|n)$. 

In the case $n = m$, if we assume that the basis is such that $\B_{1}(v_{i}\,, v_{\bar{j}}) = \delta_{i\,, j}$ 
then the group $\L$ is identified with 
\begin{equation*}
\left\{\begin{pmatrix} g & 0 \\ 0 & (g^{-1})^{t} \end{pmatrix}\,, g \in \GL(n)\right\}
\end{equation*}
and $\mathfrak{l}$ with $\mathfrak{p}(n)$. We denote by $\textbf{P}(n)$ the corresponding Lie supergroup. 
\item For the orthosymplectic Lie supergroup, assume that $m = 2e$ and $n = 2f$. We choose a basis of $\V$ such that $\Mat(\B) = \begin{pmatrix} \J_{e} & 0 \\ 0 & \Id_{2f} \end{pmatrix}$. In this case, $\H = \Sp(2e) \times \O(2f)$ and $\mathfrak{h} = \mathfrak{spo}(2e|2f)$ and the corresponding supergroup is denoted by $\textbf{SpO}(2e|2f)$. 
\item Assume that $n = m$. One can easily see that the pair $(\Delta(\GL(n) \times \GL(n))\,, \mathfrak{q}(n))$ is a Lie supergroup that we will denote by $\textbf{Q}(n)$: it is the queer Lie supergroup. 
\end{enumerate}

\end{rema}

\begin{defn}

Let $\V$ be a vector space over $\mathbb{K}$, $\B$ be an even, non-degenerate, $(-1)$-supersymmetric form on $\V$ and $\textbf{SpO}(\V\,, \B)$ be the corresponding orthosymplectic Lie supergroup. A pair of sub-supergroups $(\mathscr{G}\,, \mathscr{G}') = \big( (\G\,, \mathfrak{g})\,, (\G'\,, \mathfrak{g}')\big)$ of $\textbf{SpO}(\V\,, \B)$ is a dual pair if $\mathcal{C}_{\textbf{SpO}(\V\,, \B)}(\mathscr{G}) = \mathscr{G}'$ and $\mathcal{C}_{\textbf{SpO}(\V\,, \B)}(\mathscr{G}') = \mathscr{G}$.

\end{defn}
From the classification of dual pairs given in Theorems \ref{TheoremeIntroduction2} and \ref{TheoremeIntroduction1} we obtain the following result.
\begin{prop}

Every irreducible reductive dual pair in a complex orthosymplectic Lie supergroup is isomorphic to one of the following pairs:
\begin{enumerate}
\item $\left(\normalfont{\textbf{GL}}(m|n)\,, \normalfont{\textbf{GL}}(k|l)\right) \subseteq \normalfont{\textbf{SpO}}(2(mk+nl)|2(ml+nk))$\,,
\item $\left(\normalfont{{\textbf{Q}}}(m)\,, \normalfont{\textbf{Q}}(n)\right) \subseteq \normalfont{\textbf{SpO}}(2ab|2ab)$\,,
\item $\left(\normalfont{{\textbf{P}}}(m)\,, \normalfont{\textbf{P}}(n)\right) \subseteq \normalfont{\textbf{SpO}}(2ab|2ab)$\,,
\item $\left(\normalfont{\textbf{SpO}}(2m|n)\,, \normalfont{\textbf{OSp}}(k|2l)\right) \subseteq \normalfont{\textbf{SpO}}(2mk+2nl|4ml+nk)$\,.
\end{enumerate}
where $m\,, n\,, k\,, l \in \mathbb{Z}$ .

\label{DualPairsInOrthosymplecticSupergroup}

\end{prop}

\section{Spinor-Oscillator representation}

\label{SpinorOscillatorRepresentation}

Let $\E = \E_{\bar{0}} \oplus \E_{\bar{1}}$ be a  complex $\mathbb Z_2$-graded vector space  (thus we are assuming  $\mathbb K=\mathbb{C}$) that is  endowed with an even, non-degenerate, $(-1)$-supersymmetric form $\B$. In particular, $\B = \B_{0} \oplus^{\perp} \B_{1}$, where $\B_{0} = \B_{|_{\E_{\bar{0}} \times \E_{\bar{0}}}}$ and $\B_{1} = \B_{|_{\E_{\bar{1}} \times \E_{\bar{1}}}}$. In this section, we will assume that $\E_{\bar 0}$ and $\E_{\bar 1}$ are both even dimensional.
\\

We denote by $\T(\E)$ the tensor algebra
\begin{equation*}
\T(\E) = \bigoplus\limits_{k = 0}^{\infty} \E^{\otimes k}\,.
\end{equation*}
We have a natural $\mathbb{Z} \times \mathbb{Z}_{2}$-graded structure on $\T(\E)$. 
Let us set 
\begin{equation*}
\textbf{WC}(\E\,, \B) = \T(\E) / \langle x \otimes y - (-1)^{\left|x\right|\cdot\left|y\right|} y \otimes x - \B(x\,, y)\,, x\,, y \in \E\rangle\,.
\end{equation*}
We denote by $\boldsymbol\iota: \E \to \textbf{WC}(\E\,, \B)$ the canonical injection. One can see that
\begin{equation*}
\textbf{WC}(\E\,, \B) = \mathscr{W}(\E_{\bar{0}}\,, \B_{0}) \otimes \Cliff(\E_{\bar{1}}\,, \B_{1})\,.
\end{equation*}
where $\mathscr{W}(\E_{\bar{0}}\,, \B_{0})$ (resp. $\Cliff(\E_{\bar{1}}\,, \B_{1})$) is the Weyl algebra (resp. Clifford algebra) corresponding to $(\E_{\bar{0}}\,, \B_{0})$ (resp. $(\E_{\bar{1}}\,, \B_{1})$).  The Weyl-Clifford algebra inherits a $\mathbb{Z}$-grading coming from the one on $\T(\E)$. We denote by $\textbf{WC}(\E\,, \B)_{k}\,, k \in \mathbb{Z}^{+}$, the elements of $\textbf{WC}(\E\,, \B)$ of degree not greater than $k$. One can easily see that for every $k\,, l \in \mathbb{Z}^{+}$, we have
\begin{equation*}
\textbf{WC}(\E\,, \B)_{k} \cdot \textbf{WC}(\E\,, \B)_{l} \subseteq \textbf{WC}(\E\,, \B)_{k+l}\qquad
\text{and}\qquad\left[\textbf{WC}(\E\,, \B)_{k}\,, \textbf{WC}(\E\,, \B)_{l}\right] \subseteq \textbf{WC}(\E\,, \B)_{k+l-2}\,,
\end{equation*}
where $\left[\cdot\,, \cdot\right]$ is the canonical supercommutator on $\textbf{WC}(\E\,, \B)$. Moreover, it follows that $\textbf{WC}(\E\,, \B)_{1}$ is a nilpotent Lie superalgebra and $\left[\textbf{WC}(\E\,, \B)_{2}\,, \textbf{WC}(\E\,, \B)_{2}\right] \subseteq \textbf{WC}(\E\,, \B)_{2}$.\\

Let $\Omega_{\bar{0}}$ and $\Omega_{\bar{1}}$ be the subspaces of $\textbf{WC}(\E)_{2}$ given by
\begin{equation*}
\Omega_{\bar{0}} = \left\{\boldsymbol\iota(e_{0}) \boldsymbol\iota(f_{0}) + \boldsymbol\iota(f_{0})\boldsymbol\iota(e_{0})\,, e_{0}\,, f_{0} \in \E_{\bar{0}}\right\} \oplus \left\{\boldsymbol\iota(e_{1}) \boldsymbol\iota(f_{1}) - \boldsymbol\iota(f_{1})\boldsymbol\iota(e_{1})\,, e_{1}\,, f_{1} \in \E_{\bar{1}}\right\}\,, 
\end{equation*}
\begin{equation*}
\Omega_{\bar{1}} = \left\{\boldsymbol\iota(e_{0})\boldsymbol\iota(f_{1})\,, e_{0} \in \E_{\bar{0}}\,, f_{1} \in \E_{\bar{1}}\right\}\,,
\end{equation*}
and let 
\begin{equation}
\Omega = \Omega_{\bar{0}} \oplus \Omega_{\bar{1}}\,.
\label{DefinitionOmega}
\end{equation}

\begin{lemme}

We have $\left[\Omega\,, \Omega\right] \subseteq \Omega$. Moreover, $\Omega \cong \mathfrak{spo}(\E\,, \B)$.

\label{Degree2Elements}

\end{lemme}

\begin{proof}

The proof of this lemma is similar to the proof \cite[Proposition~5.5]{CHENGWANG}.
\end{proof}

Let $\F\subseteq\E$ be a $\mathbb{Z}_{2}$-graded subspace which is maximal $\B$-isotropic. In particular, $\F_{\bar{0}} = \F \cap \E_{\bar{0}}$ (resp. $\F_{\bar{1}} = \F \cap \E_{\bar{1}}$) is maximal $\B_{0}$-isotropic (resp. maximal $\B_{1}$-isotropic).
Let $\S = \textbf{WC}(\E\,, \B) / \textbf{WC}(\E\,, \B) \cdot \F$. We have a natural action  of $\textbf{WC}(\E\,, \B)$ on $\S$ and thus by inclusion, $\Omega \curvearrowright \S$. Let $\beta: \mathfrak{spo}(\E\,, \B) \to \Omega$ denote the  isomorphism of Lemma \ref{Degree2Elements}. In particular, we get an  action $\pi$ of $\mathfrak{spo}(\E\,, \B)$ on $\S$ by setting
\begin{equation}
\pi(\X) s := \beta(\X)\cdot s\,, \qquad (\X \in \mathfrak{spo}(\E\,, \B)\,, s \in \S)\,.
\label{eq:actionofBeta}
\end{equation}
 One can see that, up to equivalence, the representation obtained does not depend on the choice of a super-Lagrangian $\F$.

\begin{rema}

The natural action of $\Omega_{\bar{0}}$ on $\textbf{WC}(\E\,, \B)$ \big(coming from $\left[\Omega_{\bar{0}}\,, \textbf{WC}(\E\,, \B)\right] \subseteq \textbf{WC}(\E\,, \B)$\big) exponantiates to an action of $\Sp(\E_{\bar{0}}\,, \B_{0}) \times \O(\E_{\bar{1}}\,, \B_{1})$ on $\textbf{WC}(\E\,, \B)$ by automorphisms. The proof is similar to the one given in \cite[Proposition~5.6]{CHENGWANG}.

\label{RemarkActionSpO}

\end{rema}

In particular, we get an action of the orthosymplectic Lie supergroup $\textbf{SpO}(\E\,, \B)$ on the Weyl-Clifford algebra $\textbf{WC}(\E\,, \B)$.

\begin{rema}
\begin{enumerate}
\item One can view the above construction from another perspective, that is by realizing the action view  multiplication and derivation operators (see \cite[Section~2]{HOWE89} or \cite[Chapter~5]{CHENGWANG}). Let $\V = \V_{\bar{0}} \oplus \V_{\bar{1}}$ be a complex vector space. We denote by $\S(\V)$ the corresponding supersymmetric algebra, i.e.
\begin{equation*}
\S(\V) = \left(\bigoplus\limits_{k = 0}^{\infty}\V^{\otimes k}\right) / \langle x \otimes y - (-1)^{\left|x\right| \cdot \left|y\right|} y \otimes x\,, x\,, y \in \V\rangle\,.
\end{equation*}
One can easily see that
\begin{equation*}
\S(\V) = \S(\V_{\bar{0}}) \otimes \Lambda(\V_{\bar{1}})\,.
\end{equation*}
For every $u \in \V_{\alpha}, \alpha \in \mathbb{Z}_{2}$, we denote by $\M_{u}$ the operator of $\End(\S(\V))$ given by 
\begin{equation*}
\M_{u}(v) = uv \qquad (v \in \S(\V))\,.
\end{equation*}
Note that $\left|\M_{u}\right| = \left|u\right|$.
Similarly, for every $u^{*}_{0} \in \V^{*}_{\bar{0}}$ and $u^{*}_{1} \in \V^{*}_{\bar{1}}$, we define the derivation operators $\D_{u^{*}_{0}}$ and $\D_{u^{*}_{1}}$ of $\End(\S(\V))$ by
\begin{equation*}
\D_{u^{*}_{0}}(a_{1} \ldots a_{n} b_{1} \ldots b_{m}) = \sum\limits_{k = 1}^{n} u^{*}_{0}(a_{k}) a_{1} \ldots \hat{a}_{k} \ldots a_{n} b_{1} \ldots b_{m}\,,
\end{equation*}
\begin{equation*}
\D_{u^{*}_{1}}(a_{1} \ldots a_{n} b_{1} \ldots b_{m}) = \sum\limits_{k = 1}^{m} (-1)^{k-1} u^{*}_{1}(b_{k}) a_{1} \ldots a_{n} b_{1} \ldots \hat{b}_{k} \ldots b_{m}\,.
\end{equation*}
with $a_{1}\,, \ldots\,, a_{n} \in \V_{\bar{0}}\,, b_{1}\,, \ldots b_{m} \in \V_{\bar{1}}$\,.
For every $u\,, v \in \V$ and $u^{*}\,, v^{*} \in \V^{*}$, we have
\begin{equation*}
\M_{u} \circ \M_{v} = (-1)^{|u|\cdot|v|} \M_{v} \circ \M_{u} \,, \qquad \D_{u^{*}} \circ \D_{v^{*}} = (-1)^{|u^{*}|\cdot|v^{*}|} \D_{v^{*}} \circ \D_{u^{*}}\,.
\end{equation*}
Moreover, $\D_{u^{*}} \circ \M_{v} - (-1)^{\left|u^{*}\right| \cdot \left|v\right|} \M_{v} \circ \D_{u^{*}} = u^{*}(v) \cdot 1$.
Let $\E = \V \oplus \V^{*}$ (in particular, $\E$ is $\mathbb{Z}_{2}$-graded and we have $\E_{\alpha} = \V_{\alpha} \oplus \V^{*}_{\alpha}\,, \alpha \in \mathbb{Z}_{2}$). Let $\boldsymbol e: \E \to \End(\S(\V))$ be the map given by
\begin{equation*}
\boldsymbol e(e) = \M_{e}\,, \quad \boldsymbol e(e^{*}) = \D_{e^{*}}\,, \quad (e \in \V\,, e^{*} \in \V^{*})\,.
\end{equation*}
We denote by $\textbf{WC}(\E)$ the subalgebra of $\End(\S(\V))$ generated by the elements $\boldsymbol e(\tilde{e})\,, \tilde{e} \in \V$: this is the Weyl-Clifford algebra. As before, for every $k \in \mathbb{Z}^{+}$, we denote by $\textbf{WC}(\E)_{k}$ the elements of $\textbf{WC}(\E)$ of degree not greater than $k$.

Let $\Omega$ be the subspace of $\textbf{WC}(\E)_{2}$ defined in Equation \eqref{DefinitionOmega}. Using the isomorphism between $\Omega$ and $\mathfrak{spo}(\E)$, we get an action of $\mathfrak{spo}(\E)$ on $\S(\V)$ obtained by restriction of the natural action $\textbf{WC}(\E) \subseteq \End(\S(\V)) \curvearrowright \S(\V)$.
\item Note that  in the construction of (1)
the form $\B$ is hidden
in the background. In this construction, the space $\E$ is of the form $\E = \V \oplus \V^{*}$. Indeed $\B$ is the even form given  by
\begin{equation*}
\B(e_{0}+e^{*}_{0}\,, f_{0}+f^{*}_{0}) = e^{*}_{0}(f_{0}) - f^{*}_{0}(e_{0})\,,
\end{equation*}
\begin{equation*}
\B(e_{1}+e^{*}_{1}\,, f_{1}+f^{*}_{1}) = e^{*}_{1}(f_{1}) + f^{*}_{1}(e_{1})\,,
\end{equation*}
\begin{equation*}
\B(e_{0}+e^{*}_{0}\,, e_{1}+e^{*}_{1}) = 0\,,
\end{equation*}
for every $e_{0}\,, f_{0} \in \V_{\bar{0}}\,, e^{*}_{0}\,, f^{*}_{0} \in \V^{*}_{\bar{0}}\,, e_{1}\,, f_{1} \in \V_{\bar{1}}$ and $e^{*}_{1}\,, f^{*}_{1} \in \V^{*}_{\bar{1}}$.
On can easily see that this form is $(-1)$-supersymmetric and that for every $e\,, f \in \E$, we have:
\begin{equation*}
\left[\boldsymbol e(e)\,, \boldsymbol e(f)\right] = \B(e\,, f) \cdot 1.
\end{equation*}
In particular, the two constructions of the Spinor-Oscillator representation are linked by identifying $\F=\V$.  \item The action of $\mathfrak{spo}(\E)_{\bar{0}} = \mathfrak{sp}(\E_{\bar{0}}) \oplus \mathfrak{so}(\E_{\bar{1}})$ on $\S(\V) = \S(\V_{\bar{0}}) \otimes \Lambda(\V_{\bar{1}})$ is quite special. The action of $\mathfrak{so}(\E_{\bar{1}})$ on $\Lambda(\V_{\bar{1}})$ can be exponantiated to the group $\Pin(\E_{\bar{1}})$: this is a standard construction of the spinorial representation of $\Pin(\E_{\bar{1}})$. The action of $\mathfrak{sp}(\E_{\bar{0}})$ on $\S(\V_{\bar{0}})$ does not exponantiate to a group action. However, the action of the Lie algebra $\mathfrak{gl}(\V_{\bar{0}}) \subseteq \mathfrak{sp}(\E_{\bar{0}})$ on $\S(\V_{\bar{0}})$ can be exponantiated to $\GL(\V_{\bar{0}})$. Moreover, the space $\S(\V_{\bar{0}})$ can be embedded in a larger space, with a natural realization as a space of smooth form-valued functions, and the corresponding action of $\mathfrak{sp}(\E_{\bar{0}})$ on the larger space exponantiates to a group action of $\Sp(\E_{\bar{0}})$. This is the algebraic $(\mathfrak{sp}(\E_{\bar{0}})\,, \GL(\V_{\bar{0}}))$-module of the Weil representation (also called Fock model of the Weil representation).

\end{enumerate}

\label{RemaSpinWeil}
\end{rema}

\medskip

We now switch to the case $\mathbb K=\mathbb R$ and construct the Spinor-Oscillator representation of the \emph{real} orthosymplectic Lie superalgebra. We begin with the same complex orthosymplectic vector space
 $(\E\,, \B)$ considered above, 
and we choose a  complex conjugate-linear involution $\tau$  on $\E$. In particular, we get the decomposition $\E = \E^{+} \oplus \E^{-}$ over $\mathbb R$ ($\pm 1$-eigenspaces of $\tau$). We can choose $\tau$ such that  $\B_{|_{\E^{-} \times \E^{-}}}$ is non-degenerate and real-valued.
The involution $\tau$ on $\E$ extends to $\textbf{WC}(\E\,, \B)$ by
\begin{equation*}
\tau(\X\Y) = (-1)^{\left|\X\right|\cdot\left|\Y\right|} \tau(\Y)\tau(\X)\,, \qquad \left(\X\,, \Y \in \textbf{WC}(\E^{\mathbb{C}}\,, \B^{\mathbb{C}})\right)\,.
\end{equation*}
In particular, $\tau$ is a superinvolution on $\textbf{WC}(\E\,, \B)$.
Let $\Omega$ be the subspace of $\textbf{WC}(\E\,, \B)_{2}$ defined in Equation \eqref{DefinitionOmega} and let $\F$ be a maximal $\B$-isotropic subspace of $\E$. We denote by $\S = \S(\F)$ the oscillator representation of $\Omega$ (and $\mathfrak{spo}(\E\,, \B)$). One can easily see that $\tau(\Omega) = \Omega$. Let $\Omega(\mathbb{R}) := \left\{\omega \in \Omega\,, \tau(\omega) = -\omega\right\}$. For every $\X\,, \Y \in \Omega(\mathbb{R})$, we have
\begin{eqnarray*}
\tau(\left[\X\,, \Y\right]) & = & \tau(\X\Y) - (-1)^{\left|\X\right|\cdot\left|\Y\right|}\tau(\Y\X) = (-1)^{\left|\X\right|\cdot\left|\Y\right|}\tau(\Y)\tau(\X) - \tau(\X)\tau(\Y) \\
& = & (-1)^{\left|\X\right|\cdot\left|\Y\right|}\Y\X - \X\Y = -\left[\X\,, \Y\right]\,,
\end{eqnarray*}
i.e. $\Omega(\R)$ is a Lie superalgebra. Moreover, for every $e\,, f \in \E$, we have:
\begin{equation*}
\B(\tau(e)\,, \tau(f)) = \tau(e)\tau(f) - (-1)^{\left|e\right|\cdot\left|f\right|} \tau(f)\tau(e) = \tau((-1)^{\left|e\right|\cdot\left|f\right|}fe - ef) = -\tau(\B(e\,, f))\,.
\end{equation*}

The relation $\left[\Omega(\mathbb{R})\,, \E^{-}\right] \subseteq \E^{-}$ implies the following lemma:
\begin{lemme}

The Lie superalgebra $\Omega(\R)$ is isomorphic to $\mathfrak{spo}(\E^{-}\,, \B_{|_{\E^{-}}})$.
\label{lem:omeR}
\end{lemme}
Since $\Omega(\mathbb R)\subseteq \Omega$, we obtain  an action of $\Omega(\mathbb{R})$ on $\S$ by restriction. 
Thus by Lemma \ref{lem:omeR}
we obtain an action of $\mathfrak{spo}(\E^{-}\,, \B_{|_{\E^{-}}})$ on $\S$: this is the Spinor-Oscillator representation of the \emph{real} orthosymplectic Lie superalgebra $\mathfrak{spo}(\E^{-}\,, \B_{|_{\E^{-}}})$.


\section{Double Commutant Theorem}

\label{SectionDoubleCommutant}

In this section we prove Theorems \ref{thm-thm1.8-hadi} and \ref{DoubleCommutantTheoremR}.
 We begin with the case $\mathbb{K} = \mathbb{C}$.
Let $\E$ be a complex $\mathbb Z_2$-graded vector space, with $\dim_{\mathbb{C}}(\E_{\bar{0}})$ and $\dim_{\mathbb{C}}(\E_{\bar{1}})$ even, and let $\B$ be a $(-1)$-supersymmetric form on $\E$ as in Section \ref{SpinorOscillatorRepresentation}. We also use  $(\pi, \S)$ to denote the Spinor-Oscillator representation of $\mathfrak{spo}(\E\,, \B)$, defined in Equation \eqref{eq:actionofBeta}.
As explained in Remark \ref{RemaSpinWeil},  one can take $\S$ to be $\S(\F)$, where $\F$ is a maximal $\B$-isotropic subspace of $\E$.


\begin{rema}

The Lie supergroup $\textbf{SpO}(\E\,, \B)$ acts naturally on $\E$ and the action can be extended to $\S(\E)$; in particular, we have a natural action of $\mathscr{G}$ on $\S(\E)$.

\end{rema}

\begin{proof}
[Proof of Theorem \ref{thm-thm1.8-hadi}]
For every $k$, we define the symbol map \[
\sigma: \textbf{WC}(\E\,, \B)_{k} \to \bigoplus\limits_{i = 0}^{k}\S^{i}(\E)
\]as in \cite[Equation~(5.3)]{CHENGWANG}. In particular, the map $\sigma$ defines an isomorphism of $\mathscr{G}$-modules between $\gr\left(\textbf{WC}(\E\,, \B)\right)$ and $\S(\E)$, where $\gr\left(\textbf{WC}(\E\,, \B)\right)$ is the graded algebra corresponding to the filtered algebra $\left\{\textbf{WC}(\E\,, \B)_{k}\right\}_{k \in \mathbb{Z}^{+}}$.
From the results of \cite{SERGEEV} and Proposition \ref{DualPairsInOrthosymplecticSupergroup}, it follows that $\S(\E)^{\mathscr{G}}$ is generated by degree two elements. In particular, by using the map $\sigma$, we get
\begin{equation*}
\textbf{WC}(\E\,, \B)^{\mathscr{G}} = \langle \mathfrak{spo}(\E\,, \B)^{\mathscr{G}}\rangle = \langle \mathfrak{g}'\rangle\,,
\end{equation*}
and the theorem follows.
\end{proof}

\begin{rema}

\begin{enumerate}
\item For the complex irreducible reductive dual pairs $(\mathscr{G}\,, \mathscr{G}') = (\textbf{GL}(m|n)\,, \textbf{GL}(k|l))$\,, $({\textbf{Q}}(n)\,, \textbf{Q}(m))$ or $({\textbf{P}}(n)\,, \textbf{P}(m))$ of $\textbf{SpO}(\E\,, \B)$, it follows from \cite{SERGEEV} that $\textbf{WC}(\E\,, \B)^{\mathscr{G}} = \textbf{WC}(\E\,, \B)^{\mathfrak{g}}$. However, for the dual pair $(\textbf{SpO}(\V)\,, \textbf{OSp}(\W)) \subseteq \textbf{SpO}(\E = \W \otimes_{\mathbb{C}} \V)$, when the dimension of $\V_{\bar{1}}$ is even we have a strict inclusion of $\textbf{WC}(\E\,, \B)^{\textbf{SpO}(\V)}$  in $\textbf{WC}(\E\,, \B)^{\mathfrak{spo}(\V)}$. Indeed the Pfaffian constructed in \cite[Lemma~5.4]{SERGEEV} is $\mathfrak{spo}(\V)$-invariant but not $\textbf{SpO}(\V)$-invariant.
\item  Theorem \ref{thm-thm1.8-hadi} has been obtained recently in \cite{DAVIDSONKUJAWAMUTH} by Davidson-Kujawa-Mouth for the dual pair $({\textbf{P}}(n)\,, \textbf{P}(m))$ (see \cite[Theorem~7.8.2]{DAVIDSONKUJAWAMUTH}). It can also be obtained for the pairs $(\textbf{GL}(m|n)\,, \textbf{GL}(k|l))$ and $({\textbf{Q}}(n)\,, \textbf{Q}(m))$ by using \cite[Corollary~3.1]{CHENGWANG1} and \cite[Section~5.2.4]{CHENGWANG}.
\item 
Recall the action $\pi$ of $\mathfrak{g}$ on $\S(\F)$. We obtain a natural action of $\mathfrak{g}$ on $\End(\S(\F))$ given by 
\begin{equation*}
\X \cdot \T = \pi(\X) \circ \T - (-1)^{\left|\X\right| \cdot \left|\T\right|} \T \circ \pi(\X)\,, \qquad (\X \in \mathfrak{g}\,, \T \in \End(\S(\F)))\,.
\end{equation*}
One can easily see that the space $\textbf{WC}(\E\,, \B)$ is invariant under the action of $\mathfrak{g}$ and that the elements of $\textbf{WC}(\E\,, \B)^{\mathfrak{g}}$ are exactly the elements of $\textbf{WC}(\E\,, \B)$ that supercommute with the action of $\mathfrak{g}$ on $\S(\F)$.
\end{enumerate}

\end{rema}

\begin{rema}

In this section (and in Section \ref{SpinorOscillatorRepresentation}), we assumed that $\dim_{\mathbb{C}}(\E_{\bar{1}})$ is even. The construction is a slightly different when $\E_{\bar{1}}$ is odd (see \cite{NISHIYAMA1}).
For the pair $(\mathfrak{spo}(2m|2n+1)\,, \mathfrak{osp}(2k+1|2l))$ as a dual  pair in \[\mathfrak{spo}(2(m(2k+1)+l(2n+1))\,|\,4(ml +nk)+2(n+k)+1),
\]the proof of Theorem \ref{thm-thm1.8-hadi} needs some minor modifications but is still valid. In Slupinski (see \cite[Lemma~4.2.1]{SLUPINSKI}) Theorem \ref{thm-thm1.8-hadi} is proved in the special case of  a dual pair consisting of two (odd dimensional) orthogonal groups.

\end{rema}

Next we consider the case $\mathbb{K} = \mathbb{R}$.
In this case  $\E$ is a real vector space and $\B$ is a $(-1)$-supersymmetric form on $\E$. We denote by $(\pi\,, \S)$ the Spinor-Oscillator representation of $\mathfrak{spo}(\E\,, \B)$. We denote by $(\E^{\mathbb{C}}\,, \B^{\mathbb{C}})$ the complexification of $(\E\,, \B)$. 





\begin{rema}
Let 
us denote 
the real and complex orthosymplectic Lie supergroups corresponding to $(\E\,, \B)$ and $(\E^{\mathbb{C}}\,, \B^{\mathbb{C}})$ by
\[
\textbf{SpO}(\E\,, \B) = (\Sp(\E_{\bar{0}}\,, \B_{0}) \times \O(\E_{\bar{1}}\,, \B_{1})\,, \mathfrak{spo}(\E\,, \B))
\]and \[
\textbf{SpO}(\E^{\mathbb{C}}\,, \B^{\mathbb{C}}) = (\Sp(\E^{\mathbb{C}}_{\bar{0}}\,, \B^{\mathbb{C}}_{0}) \times \O(\E^{\mathbb{C}}_{\bar{1}}\,, \B^{\mathbb{C}}_{1})\,, \mathfrak{spo}(\E^{\mathbb{C}}\,, \B^{\mathbb{C}}))\,,
\] respectively.
There is a natural embedding of Harish-Chandra pairs $\textbf{SpO}(\E\,,\B)\to
\textbf{SpO}(\E^{\mathbb{C}}\,, \B^{\mathbb{C}})
$.
Let $(\G\,, \mathfrak{g})$ be a sub-supergroup of $\textbf{SpO}(\E\,, \B)$, where $\G$ is a real  algebraic subgroup of 
$
\Sp(\E_{\bar{0}}\,, \B_{0}) \times \O(\E_{\bar{1}}\,, \B_{1})$. We define the sub-supergroup $(\G_\mathbb C\,,\mathfrak{g}_\mathbb C)$ of
$\textbf{SpO}(\E^{\mathbb{C}}\,, \B^{\mathbb{C}})$
as follows: $\mathfrak g_\mathbb C$ is the $\mathbb C$-linear span of the image of $\mathfrak g$ in $\mathfrak{spo}(\E\,,\B)$, and 
$\G_\mathbb C$ is the Zariski closure of the image of $\G$ in 
$
\Sp(\E^{\mathbb{C}}_{\bar{0}}\,, \B^{\mathbb{C}}_{0}) \times \O(\E^{\mathbb{C}}_{\bar{1}}\,, \B^{\mathbb{C}}_{1})$.
By Remark \ref{RemarkActionSpO}, it follows that $\G_{\mathbb{C}}$ acts naturally on $\textbf{WC}(\E^{\mathbb{C}}\,, \B^{\mathbb{C}})$. Similarly, the complexification $\mathfrak{g}_{\mathbb{C}}$ of $\mathfrak{g}$ is a Lie sub-superalgebra of $\mathfrak{spo}(\E^{\mathbb{C}}\,, \B^{\mathbb{C}})$ and thus $\mathfrak{g}_{\mathbb{C}}$ acts on $\textbf{WC}(\E^{\mathbb{C}}\,, \B^{\mathbb{C}})$. The pair $(\G_{\mathbb{C}}\,, \mathfrak{g}_{\mathbb{C}})$  can be viewed as the complexification of $(\G\,, \mathfrak{g})$.
With the preceding  notation, one can easily verify that
\begin{equation*}
\textbf{WC}(\E^{\mathbb{C}}\,, \B^{\mathbb{C}})^{\big(\G\,, \mathfrak{g}\big)} = \textbf{WC}(\E^{\mathbb{C}}\,, \B^{\mathbb{C}})^{\big(\G_{\mathbb{C}}\,, \mathfrak{g}_{\mathbb{C}}\big)}\,,
\end{equation*}
which is the key point in the proof of the double commutant theorem in the real case.

\label{RemarkComplexification}

\end{rema}

\begin{proof}[Proof of Theorem \ref{DoubleCommutantTheoremR}]
First we consider dual pairs of Type II.
Let \[
\big((\G\,, \mathfrak{g})\,, (\G'\,, \mathfrak{g}')\big) = \big((\GL(\V_{\bar{0}}) \times \GL(\V_{\bar{1}})\,,\mathfrak{gl}(\V))\,, (\GL(\W_{\bar{0}}) \times \GL(\W_{\bar{1}})\,,\mathfrak{gl}(\W))\big)
\]be such an irreducible reductive dual pair in $\textbf{SpO}(\E\,, \B)$, where $\V$ and $\W$ are $\mathbb Z_2$-graded vector spaces over $\mathbb{D}$ with $\mathbb{D} = \mathbb{R}, \mathbb{C}$ or $\mathbb{H}$, and $\E = (\W \otimes_{\mathbb{D}} \V)_{\mathbb{R}} \oplus (\W^{*} \otimes_{\mathbb{D}} \V^{*})_{\mathbb{R}}$. In the proofs that follow,
we use the results of Sergeev stating that in each case invariants are generated by quadratic elements \cite{SERGEEV}.

\begin{itemize}
\item \underline{$\mathbb{D} = \mathbb{R}$: } The complexification of $\W \otimes_{\mathbb{R}} \V $ is $\W^{\mathbb{C}} \otimes_{\mathbb{C}} \V^{\mathbb{C}}$.  We have
\begin{eqnarray*}
\S(\E^{\mathbb{C}})^{\big(\G_{\mathbb{C}}\,, \mathfrak{g}_{\mathbb{C}}\big)} & = & \S(\W^{\mathbb{C}} \otimes_{\mathbb{C}} \V^{\mathbb{C}} \oplus {\W^{\mathbb{C}}}^{*} \otimes_{\mathbb{C}} {\V^{\mathbb{C}}}^{*})^{\big(\GL(\V^{\mathbb{C}}_{\bar{0}}) \times \GL(\V^{\mathbb{C}}_{\bar{1}})\,, \mathfrak{gl}(\V_{\mathbb{C}})\big)} \\
& = & \langle \S^{2}(\W^{\mathbb{C}} \otimes_{\mathbb{C}} \V^{\mathbb{C}} \oplus {\W^{\mathbb{C}}}^{*} \otimes_{\mathbb{C}} {\V^{\mathbb{C}}}^{*})^{\big(\GL(\V^{\mathbb{C}}_{\bar{0}}) \times \GL(\V^{\mathbb{C}}_{\bar{1}})\,, \mathfrak{gl}(\V_{\mathbb{C}})\big)}  \rangle \\
& = & \langle \mathfrak{spo}(\W^{\mathbb{C}} \otimes_{\mathbb{C}} \V^{\mathbb{C}} \oplus {\W^{\mathbb{C}}}^{*} \otimes_{\mathbb{C}} {\V^{\mathbb{C}}}^{*})^{\big(\GL(\V^{\mathbb{C}}_{\bar{0}}) \times \GL(\V^{\mathbb{C}}_{\bar{1}})\,, \mathfrak{gl}(\V^{\mathbb{C}})\big)} \rangle = \langle \mathfrak{gl}(\W^{\mathbb{C}}) \rangle = \langle \mathfrak{g}'_{\mathbb{C}} \rangle \\
& = & \langle \mathfrak{g}' \rangle\,,
\end{eqnarray*}
and by using the map $\sigma: \gr\left(\textbf{WC}(\E^{\mathbb{C}}\,, \B^{\mathbb{C}})\right) \to \S(\E^{\mathbb{C}})$ defined in proof of Theorem \ref{thm-thm1.8-hadi}, we get that $\textbf{WC}(\E^{\mathbb{C}}\,, \B^{\mathbb{C}})^{\big(\G\,, \mathfrak{g}\big)} = \textbf{WC}(\E^{\mathbb{C}}\,, \B^{\mathbb{C}})^{(\G_{\mathbb{C}}\,, \mathfrak{g}_{\mathbb{C}})}$ is generated by $\mathfrak{g}'$.
\item \underline{$\mathbb{D} = \mathbb{C}$: } We have $\E = (\W \otimes_{\mathbb{C}} \V)_{\mathbb{R}} \oplus (\W^{*} \otimes_{\mathbb{C}} \V^{*})_{\mathbb{R}}$, where $\W$ and $\V$ are complex vector spaces, and then $\E_{\mathbb{C}} = \W \otimes_{\mathbb{C}} \V \oplus \W \otimes_{\mathbb{C}} \V \oplus \W^{*} \otimes_{\mathbb{C}} \V^{*} \oplus \W^{*} \otimes_{\mathbb{C}} \V^{*}$. In particular,
\begin{eqnarray*}
\S(\E^{\mathbb{C}})^{\big(\G_{\mathbb{C}}\,, \mathfrak{g}_{\mathbb{C}}\big)} & = & \S(\W \otimes_{\mathbb{C}} \V \oplus \W \otimes_{\mathbb{C}} \V \oplus \W^{*} \otimes_{\mathbb{C}} \V^{*} \oplus \W^{*} \otimes_{\mathbb{C}} \V^{*})^{\big(\G \times \G\,, \mathfrak{g} \oplus \mathfrak{g}\big)} \\
& = & \S(\W \otimes_{\mathbb{C}} \V \oplus \W^{*} \otimes_{\mathbb{C}} \V^{*})^{\big(\G\,, \mathfrak{g}\big)} \otimes \S(\W \otimes_{\mathbb{C}} \V \oplus \W^{*} \otimes_{\mathbb{C}} \V^{*})^{\big(\G\,, \mathfrak{g}\big)} \\
& = & \langle \S^{2}(\W \otimes_{\mathbb{C}} \V \oplus \W^{*} \otimes_{\mathbb{C}} \V^{*})^{\big(\G\,, \mathfrak{g}\big)} \rangle \otimes \langle \S^{2}(\W \otimes_{\mathbb{C}} \V \oplus \W^{*} \otimes_{\mathbb{C}} \V^{*})^{\big(\G\,, \mathfrak{g}\big)} \rangle \\
& = & \langle \mathfrak{g}' \rangle \otimes \langle \mathfrak{g}' \rangle = \langle \mathfrak{g}' \oplus \mathfrak{g}' \rangle = \langle \mathfrak{g}'_{\mathbb{C}} \rangle \\
& = & \langle \mathfrak{g}'\rangle
\end{eqnarray*}
and again it follows that $\textbf{WC}(\E^{\mathbb{C}}\,, \B^{\mathbb{C}})^{\big(\G_{\mathbb{C}}\,, \mathfrak{g}_{\mathbb{C}}\big)}$ is generated by $\mathfrak{g}'$.
\item \underline{$\mathbb{D} = \mathbb{H}$: } We have $\E = (\W \otimes_{\mathbb{H}} \V)_{\mathbb{R}} \oplus (\W \otimes_{\mathbb{H}} \V)^{*}_{\mathbb{R}}$ (where $\W$ and $\V$ are vector spaces over $\mathbb{H}$) and then $\E_{\mathbb{C}} = \W^{1} \otimes _{\mathbb{C}} \V^{1} \oplus (\W^{1} \otimes _{\mathbb{C}} \V^{1})^{*}$ (here $\W^{1}$ is the complex vector space corresponding to $\W$, i.e. if $\W = \mathbb{H}^{n|m}$, $\W^{\mathbb{C}} = \mathbb{C}^{2n|2m}$). In particular,
\begin{eqnarray*}
\S(\E^{\mathbb{C}})^{\big(\G_{\mathbb{C}}\,, \mathfrak{g}_{\mathbb{C}}\big)} & = & \S(\W^{1} \otimes_{\mathbb{C}} \U^{1} \oplus {\W^{1}}^{*} \otimes_{\mathbb{C}} {\U^{1}}^{*})^{\big(\GL(\V^{1}_{\bar{0}}) \times \GL(\V^{1}_{\bar{1}})\,, \mathfrak{gl}(\V^{1})\big)} \\ 
& = & \langle \S^{2}(\W^{1} \otimes_{\mathbb{C}} \U^{1} \oplus {\W^{1}}^{*} \otimes_{\mathbb{C}} {\U^{1}}^{*})^{\big(\GL(\V^{1}_{\bar{0}}) \times \GL(\V^{1}_{\bar{1}})\,, \mathfrak{gl}(\V^{1})\big)} \rangle \\
& = & \langle \mathfrak{spo}(\W^{1} \otimes_{\mathbb{C}} \U^{1} \oplus {\W^{1}}^{*} \otimes_{\mathbb{C}} {\U^{1}}^{*})^{\big(\GL(\V^{1}_{\bar{0}}) \times \GL(\V^{1}_{\bar{1}})\,, \mathfrak{gl}(\V^{1})\big)} \rangle \\
& = & \langle \mathfrak{gl}(\W^{1}) \rangle = \langle \mathfrak{g}'_{\mathbb{C}} \rangle \\
& = & \langle \mathfrak{g}' \rangle 
\end{eqnarray*}
and then $\textbf{WC}(\E^{\mathbb{C}}\,, \B^{\mathbb{C}})^{(\G_{\mathbb{C}}\,, \mathfrak{g}_{\mathbb{C}})}$ is generated by $\mathfrak{g}'$.
\end{itemize}
We now consider the irreducible dual pairs $\left((\G, \mathfrak{g})\,, (\G', \mathfrak{g}')\right) = \left((\GL(\V)^{\triangle}\,, \mathfrak{q}(\V))\,, (\GL(\W)^{\triangle}\,, {\mathfrak{q}}(\W))\right)$ in $\textbf{SpO}(\E\,, \B)$, where $\V$ and $\W$ are vector spaces over $\mathbb{K}$ with $\mathbb{D} = \mathbb{R}, \mathbb{C}$ or $\mathbb{H}$, $\E = (\W \otimes_{\mathbb{D}} \V)_{\mathbb{R}}$ and where $\GL(\V)^{\triangle}$ is the "diagonal subgroup" of $\GL(\V_{\bar{0}}) \times \GL(\V_{\bar{1}})$.
Again, we distinguish three different cases: 
\begin{itemize}
\item \underline{$\mathbb{D} = \mathbb{R}$: } The complexification of $\W \otimes_{\mathbb{R}} \V $ is $\W^{\mathbb{C}} \otimes_{\mathbb{C}} \V^{\mathbb{C}}$. We have
\begin{eqnarray*}
\S(\E^{\mathbb{C}})^{\big(\G_{\mathbb{C}}\,, \mathfrak{g}_{\mathbb{C}}\big)} & = & \S(\W^{\mathbb{C}} \otimes_{\mathbb{C}} \V^{\mathbb{C}} )^{\big(\GL(\V^{\mathbb{C}})^{\triangle}\,, \mathfrak{q}(\V_{\mathbb{C}})\big)} \\
& = & \langle \S^{2}(\W^{\mathbb{C}} \otimes_{\mathbb{C}} \V^{\mathbb{C}})^{\big(\GL(\V^{\mathbb{C}})^{\triangle}\,, \mathfrak{q}(\V_{\mathbb{C}})\big)}  \rangle \\
& = & \langle \mathfrak{spo}(\W^{\mathbb{C}} \otimes_{\mathbb{C}} \V^{\mathbb{C}})^{\big(\GL(\V^{\mathbb{C}})^{\triangle}\,, \mathfrak{q}(\V^{\mathbb{C}})\big)} \rangle = \langle {\mathfrak{q}}(\W^{\mathbb{C}}) \rangle = \langle \mathfrak{g}'_{\mathbb{C}} \rangle \\
& = & \langle \mathfrak{g}' \rangle\,,
\end{eqnarray*}
and then $\textbf{WC}(\E^{\mathbb{C}}\,, \B^{\mathbb{C}})^{\big(\G\,, \mathfrak{g}\big)} = \textbf{WC}(\E^{\mathbb{C}}\,, \B^{\mathbb{C}})^{(\G_{\mathbb{C}}\,, \mathfrak{g}_{\mathbb{C}})}$ is generated by $\mathfrak{g}'$.
\item \underline{$\mathbb{D} = \mathbb{C}$: } We have $\E = \W \otimes_{\mathbb{R}} \V$, where $\W$ and $\V$ are complex vector spaces, and then $\E_{\mathbb{C}} = \W \otimes_{\mathbb{C}} \V \oplus \W \otimes_{\mathbb{C}} \V$. In particular,
\begin{eqnarray*}
\S(\E^{\mathbb{C}})^{\big(\G_{\mathbb{C}}\,, \mathfrak{g}_{\mathbb{C}}\big)} & = & \S(\W \otimes_{\mathbb{C}} \V \oplus \W \otimes_{\mathbb{C}} \V)^{\big(\G \times \G\,, \mathfrak{g} \oplus \mathfrak{g}\big)} \\
& = & \S(\W \otimes_{\mathbb{C}} \V)^{\big(\G\,, \mathfrak{g}\big)} \otimes \S(\W \otimes_{\mathbb{C}} \V)^{\big(\G\,, \mathfrak{g}\big)} \\
& = & \langle \S^{2}(\W \otimes_{\mathbb{C}} \V)^{\big(\G\,, \mathfrak{g}\big)} \rangle \otimes \langle \S^{2}(\W \otimes_{\mathbb{C}} \V)^{\big(\G\,, \mathfrak{g}\big)} \rangle \\
& = & \langle \mathfrak{g}' \rangle \otimes \langle \mathfrak{g}' \rangle = \langle \mathfrak{g}' \oplus \mathfrak{g}' \rangle = \langle \mathfrak{g}'_{\mathbb{C}} \rangle \\
& = & \langle \mathfrak{g}'\rangle
\end{eqnarray*}
and again it follows that $\textbf{WC}(\E^{\mathbb{C}}\,, \B^{\mathbb{C}})^{\big(\G_{\mathbb{C}}\,, \mathfrak{g}_{\mathbb{C}}\big)}$ is generated by $\mathfrak{g}'$.
\item \underline{$\mathbb{D} = \mathbb{H}$: } We have $\E = (\W \otimes_{\mathbb{H}} \V)_{\mathbb{R}}$ (where $\W$ and $\V$ are vector spaces over $\mathbb{H}$) and then $\E_{\mathbb{C}} = \W^{1} \otimes _{\mathbb{C}} \V^{1}$ (where $\V^{1}$ and $\W^{1}$ has been defined previously). In particular,
\begin{eqnarray*}
\S(\E^{\mathbb{C}})^{\big(\G_{\mathbb{C}}\,, \mathfrak{g}_{\mathbb{C}}\big)} & = & \S(\W^{1} \otimes_{\mathbb{C}} \U^{1})^{\big(\GL(\V^{1})^{\triangle}\,, \mathfrak{q}(\V^{1})\big)} \\ 
& = & \langle \S^{2}(\W^{1} \otimes_{\mathbb{C}} \U^{1})^{\big(\GL(\V^{1})^{\triangle}\,, \mathfrak{q}(\V^{1})\big)} \rangle \\
& = & \langle \mathfrak{spo}(\W^{1} \otimes_{\mathbb{C}} \U^{1})^{\big(\GL(\V^{1})^{\triangle}\,, \mathfrak{q}(\V^{1})\big)} \rangle \\
& = & \langle \mathfrak{gl}(\W^{1}) \rangle = \langle \mathfrak{g}'_{\mathbb{C}} \rangle \\
& = & \langle \mathfrak{g}' \rangle 
\end{eqnarray*}
and then $\textbf{WC}(\E^{\mathbb{C}}\,, \B^{\mathbb{C}})^{(\G_{\mathbb{C}}\,, \mathfrak{g}_{\mathbb{C}})}$ is generated by $\mathfrak{g}'$.
\end{itemize}

For Type I dual pairs the proofs are similar. We are going to prove it in one case only. Let $(\mathscr{G}\,, \mathscr{G}') = (\textbf{U}(\V)\,, \textbf{U}(\W))$ be an irreducible dual pair in $\textbf{SpO}(\E\,, \B)$, where $\V$ and $\V'$ are $\mathbb{Z}_{2}$-graded complex vector spaces, $\E = (\W \otimes_{\mathbb{C}} \V)_{\mathbb{R}}$ and $\textbf{U}(\V) = \left(\U(\V_{\bar{0}}) \times \U(\V_{\bar{1}})\,, \mathfrak{u}(\V)\right)$. One can see that $\E^{\mathbb{C}} = \W \otimes_{\mathbb{C}} \V \oplus \W \otimes_{\mathbb{C}} \V$ and that, in the notation of Remark \ref{RemarkComplexification} we have $\textbf{U}(\V)_{\mathbb{C}} = \textbf{GL}(\V)$. In particular, 
\begin{equation*}
\S(\E^{\mathbb{C}})^{\textbf{U}(\V)} = \S(\E^{\mathbb{C}})^{\textbf{GL}(\V)} = \S(\W \otimes_{\mathbb{C}} \V \oplus \W \otimes_{\mathbb{C}} \V)^{\textbf{GL}(\V)} = \langle \mathfrak{gl}(\W)\rangle = \langle\mathfrak{u}(\W)\rangle\,,
\end{equation*}
and the Theorem follows for the pair $(\textbf{U}(\V)\,, \textbf{U}(\W))$. 
\end{proof}

\section{Howe duality for the pair (\textbf{SpO}$(2n|1)\,, \mathfrak{osp}(2k|2l))$}

\label{Section9} 

Let $(\mathscr{G}\,, \mathscr{G}')$ be an irreducible reductive dual pair in $\textbf{SpO}(\E\,, \B)$, with $\E = \V \oplus \V^{*}$  a complex $\mathbb Z_2$-graded vector space, where $\B$ is the standard even, non-degenerate, $(-1)$-supersymmetric form on $\V \oplus \V^{*}$ (in particular, $\V$ is $\B$-isotropic). We assume that  $\mathscr{G} \subseteq \textbf{GL}(\V)$. Let $(\pi\,, \S(\V))$ be the Spinor-Oscillator representation of $\mathfrak{spo}(\E\,, \B)$ as in Remark \ref{RemaSpinWeil}. 
In this section, we assume the following conditions:
\begin{itemize}
   \item[(a)]
Every finite dimensional $\mathscr{G}$-module is completely reducible.\item[(b)] $\dim_{\mathbb{C}} \End_{\mathscr{G}}(\U_{\lambda}) = 1$ for every irreducible finite dimensional $\mathscr{G}$-module $(\lambda\,, \U_{\lambda})$.
\end{itemize}

We first look at the action of $\mathscr{G}$ on $\S(\V)$. Using that $\mathscr{G} \subseteq \textbf{GL}(\V)$, it follows that $\mathscr{G} \curvearrowright \S^{k}(\V)$ for every $k \in \mathbb{Z}_{+}$. In particular, the action of $\mathscr{G}$ on $\S(\V)$ is completely reducible. We obtain the decomposition
\begin{equation*}
\S(\V) = \bigoplus\limits_{(\lambda\,, \U_{\lambda}) \in \mathscr{G}_{\irr}} \U(\lambda)\,,
\end{equation*}
where $\mathscr{G}_{\irr}$ is the set of finite dimensional irreducible $\mathscr{G}$-modules and $\U(\lambda)$ is the $\lambda$-isotypic component corresponding to $(\lambda\,, \U_{\lambda}) \in \mathscr{G}_{\irr}$, i.e.
\begin{equation*}
\U(\lambda) = \left\{\T(\U_{\lambda})\,, \T \in \Hom_{\mathscr{G}}(\U_{\lambda}\,, \S(\V))\right\}\,.
\end{equation*}
For every $(\lambda\,, \U_{\lambda}) \in \mathscr{G}_{\irr}$, let $\Omega(\lambda)$ be the space $\Omega(\lambda) := \Hom_{\mathscr{G}}(\U_{\lambda}\,, \S(\V))$. For every $\X \in \textbf{WC}(\E\,, \B)\,, g \in \G\,, \Y \in \mathfrak{g}$ and $\omega_{\lambda} \in \Omega(\lambda)$, we have:
\begin{equation*}
(\pi(\X) \circ \omega_{\lambda}) \circ \lambda(g) = \pi(\X) \circ \pi(g) \circ \omega_{\lambda} = \pi(g) \circ (\pi(\X) \circ \omega_{\lambda})\,,
\end{equation*}
\begin{eqnarray*}
(\pi(\X) \circ \omega_{\lambda}) \circ \lambda(\Y)  & = & (-1)^{\left|\Y\right| \cdot \left|\omega_{\lambda}\right|} \pi(\X) \circ \pi(\Y) \circ \omega_{\lambda} \\
& = & (-1)^{\left|\X\right| \cdot \left|\Y\right|} (-1)^{\left|\Y\right| \cdot \left|\omega_{\lambda}\right|} \pi(\Y) \circ \pi(\X) \circ \omega_{\lambda} = (-1)^{\left|\Y\right| \cdot (\left|\omega_{\lambda}\right| + \left|\omega_{\lambda}\right|)} \pi(\Y) \circ \pi(\X) \circ \omega_{\lambda} \\
& = & (-1)^{\left|\Y\right| \cdot \left|\pi(\X) \circ \omega_{\lambda}\right|} \pi(\Y) \circ (\pi(\X) \circ \omega_{\lambda})\,,
\end{eqnarray*}
i.e. $\Omega(\lambda)$ has a natural structure of $\textbf{WC}(\E\,, \B)^{\mathscr{G}}$-module given by
\begin{equation*}
\X \cdot \omega_{\lambda} := \pi(\X) \circ \omega_{\lambda}\,, \qquad (\X \in \textbf{WC}(\E\,, \B)^{\mathscr{G}}\,, \omega_{\lambda} \in \Omega(\lambda))\,.
\end{equation*}

We recall the following lemma (see \cite[Proposition~5.7]{CHENGWANG}).

\begin{lemme}

For every finite dimensional vector space $\X \subseteq \S(\V)$, we have
\begin{equation*}
\normalfont{\textbf{WC}}(\E\,, \B)_{|_{\X}} = \Hom(\X\,, \S(\V))\,.
\end{equation*}
\label{LemmaRestriction}
\end{lemme}
Assume now that $\X \in \mathscr{G}_{\irr}$. We have a natural action of $\mathscr{G}$ on $\Hom(\X\,, \S(\V))$ given by
\begin{equation*}
g \cdot \T(x) = g(\T(g^{-1}x))\,, \quad \Y \cdot \T(x) = \Y(\T(x)) - (-1)^{\left|\Y\right| \cdot \left|\T\right|} \T(\Y(x))\,, \quad ((g\,, \Y) \in \mathscr{G}\,, x \in \X\,, \T \in \Hom(\X\,, \S(\V)))\,.
\end{equation*}

\begin{lemme}

The actions of $\mathscr{G}$ on $\normalfont{\textbf{WC}}(\E\,, \B)$ and $\Hom(\X\,, \S(\V))$ are completely reducible.

\label{LemmaDoubleSemisimple}

\end{lemme}

\begin{proof}

We denote by $\gr(\textbf{WC}(\E\,, \B))$ the graded algebra correpsonding to the filtration $\left\{\textbf{WC}(\E\,, \B)_{k}\right\}_{k}$. Let $\sigma: \gr(\textbf{WC}(\E\,, \B)) \to \S(\E)$ be the Weyl symbol defined in \cite[Proposition~5.4]{CHENGWANG}. The map $\sigma$ is an isomorphism of $\mathscr{G}$-modules. The action of $\mathscr{G}$ on $\S(\E)$ is such that $\mathscr{G} \curvearrowright \S^{k}(\E)$ for every $k \in \mathbb{Z}_{+}$, so $\mathscr{G} \curvearrowright \textbf{WC}(\E\,, \B)$ semisimply.
Similarly, using that \[\Hom(\X\,, \S(\V)) = \X^{*} \otimes \S(\V) = \bigoplus\limits_{k \in \mathbb{Z}_{+}} \X^{*} \otimes \S^{k}(\V),
\]it follows that the action of $\mathscr{G}$ on $\Hom(\X\,, \S(\V))$ is semisimple.
\end{proof}

We denote by $\textbf{WC}(\E\,, \B)^{\mathscr{G}}$ and $\Hom_{\mathscr{G}}(\X\,, \S(\V))$ the space of $\mathscr{G}$-invariants. In other words, $\textbf{WC}(\E\,, \B)^{\mathscr{G}}$ (respectively, $\Hom_{\mathscr{G}}(\X\,, \S(\V))$) is the isotypic component in $\textbf{WC}(\E\,, \B)$ (respectively, $\Hom(\X\,, \S(\V))$) corresponding to the trivial action of $\mathscr{G}$.
Using Lemma \ref{LemmaDoubleSemisimple}, we denote by $\P_{1}$ and $\P_{2}$ the projections $\P_{1}: \textbf{WC}(\E\,, \B) \to \textbf{WC}(\E\,, \B)^{\mathscr{G}}$ and $\P_{2}: \Hom(\X\,, \S(\V)) \to \Hom_{\mathscr{G}}(\X\,, \S(\V))$. We denote by $\P_{\X}: \textbf{WC}(\E\,, \B) \to \Hom(\X\,, \S(\V))$ the map obtained from Lemma \ref{LemmaRestriction}. One can easily see that
\begin{equation*}
\P_{\X} \circ \P_{1} = \P_{2} \circ \P_{\X}\,.
\end{equation*}

\begin{lemme}

For every finite dimensional $\X\subseteq\S(\V)$, we have $\normalfont{\textbf{WC}}(\E\,, \B)^{\mathscr{G}}_{|_{\X}} = \Hom_{\mathscr{G}}(\X\,, \S(\V))$.

\label{ProjectionXG}

\end{lemme}

\begin{proof}

Let $\omega_{x} \in \Hom_{\mathscr{G}}(\X\,, \S(\V))$. In particular, $\P_{2}(\omega_{x}) = \omega_{x}$. Using that $\omega_{x} \in \Hom(\X\,, \S(\V))$, it follows from Lemma \ref{LemmaRestriction} that there exists $\Z \in \textbf{WC}(\E\,, \B)$ such that $\P_{\X}(\Z) = \omega_{x}$. Then
\begin{equation*}
\omega_{x} = \P_{2}(\omega_{x}) = \P_{2}(\P_{\X}(\Z)) = \P_{\X}(\P_{1}(\Z))\,,
\end{equation*}
and using that $\P_{1}(\Z) \in \textbf{WC}(\E\,, \B)^{\mathscr{G}}$, the Lemma follows.
\end{proof}

We obtain the following proposition.

\begin{prop}

For every $(\lambda\,, \U_{\lambda}) \in \mathscr{G}_{\irr}$, the $\normalfont{\textbf{WC}}(\E\,, \B)^{\mathscr{G}}$-module $\Omega(\lambda)$ is irreducible. Moreover, if $\lambda \nsim \mu$, the $\normalfont{\textbf{WC}}(\E\,, \B)^{\mathscr{G}}$-modules $\Omega(\lambda)$ and $\Omega(\mu)$ are not equivalent.

\end{prop}

\begin{proof}

We first start by proving the irreducibility of the $\mathscr{G}$-modules $\Omega(\lambda)$. Let $0\neq \A\,, \B \in \Omega(\lambda)$ and let $\A(\lambda) := \A(\U_{\lambda})\,, \B(\lambda) = \B(\U_{\lambda})$. The spaces $\A(\lambda)$ and $\B(\lambda)$ are irreducible $\mathscr{G}$-modules isomorphic to $\U_{\lambda}$ or $\Pi(\U_{\lambda})$. Thus there exists an element $\omega_{\lambda} \in \Hom_{\mathscr{G}}(\A(\lambda)\,, \B(\lambda)) \subseteq \Hom_{\mathscr{G}}(\A(\lambda)\,, \S(\V))$ such that $\omega_{\lambda} \circ \A = \B$. According to Lemma \ref{ProjectionXG}, there exists an element $\Z \in \textbf{WC}(\E\,, \B)^{\mathscr{G}}$ such that $\P_{\A(\lambda)}(\Z) = \omega_{\lambda}$. Using that
\begin{equation*}
\B = \omega_{\lambda} \circ \A = \Z \cdot \A\,,
\end{equation*}
we get that $\Z \cdot \A = \B$, i.e. $\Omega(\lambda)$ is an $\textbf{WC}(\E\,, \B)^{\mathscr{G}}$-irreducible module.

Now, let $\lambda\,, \mu \in \mathscr{G}_{\irr}$ be such that $\lambda \nsim \mu$. Let $\Gamma \in \Hom_{\textbf{WC}(\E\,, \B)^{\mathscr{G}}}(\Omega(\lambda)\,, \Omega(\mu))$. We will prove that $\Gamma=0$. Let $\T \in \Omega(\lambda)$. Then $\T(\U_\lambda) \cap \Gamma(\T)(\U_\mu) = \left\{0\right\}$ because the two subspaces are non-isomorphic irreducible $\mathscr{G}$-modules. Let $\P_{\Gamma}: \T(\U_\lambda) \oplus \Gamma(\T)(\U_\mu) \to \Gamma(\T)(\U_\mu)$ be the projection onto the second component. One can easily see that $\P_{\Gamma}$ is $\mathscr{G}$-equivariant. In particular, $\P_{\Gamma} \in  \Hom_{\mathscr{G}}(\T(\U_\lambda) \oplus \Gamma(\T)(\U_\mu)\,, \S(\V))$, and it follows from Lemma \ref{ProjectionXG} that there exists $\omega \in \normalfont{\textbf{WC}}(\E\,, \B)^{\mathscr{G}}
$
such that $\omega_{|_{\T(\U_\lambda) \oplus \Gamma(\T)(\U_\mu)}} = \P_{\Gamma}$. Using that $\P_{\Gamma}\circ \T = 0$, it follows that $\omega\cdot \T = 0$ and then
\begin{equation*}
0 =  \Gamma(\omega\cdot \T) = \omega\cdot \Gamma(\T) = \P_{\Gamma}\circ \Gamma(\T) = \Gamma(\T)\,,
\end{equation*}
i.e. $\Gamma(\T) = 0$. In particular, $\Gamma = 0$ and the proposition follows.
\end{proof}

\begin{coro}

As a $\mathscr{G} \times \normalfont{\textbf{WC}}(\E\,, \B)^{\mathscr{G}}$-modules, we have:
\begin{equation*}
\S(\V) = \bigoplus\limits_{(\lambda\,, \U_{\lambda}) \in \mathscr{G}_{\irr}} \U_{\lambda} \otimes \Omega(\lambda)\,,
\end{equation*}
where $\U_{\lambda}$ (resp. $\Omega(\lambda)$) is an irreducible $\mathscr{G}$-module (resp. irreducible $\normalfont{\textbf{WC}}(\E\,, \B)^{\mathscr{G}}$-module).

\end{coro}

In Section \ref{SectionDoubleCommutant}, we proved that $\textbf{WC}(\E\,, \B)^{\mathscr{G}} = \langle \mathfrak{g}' \rangle$ (see Theorem~\ref{thm-thm1.8-hadi}). In particular, for every $\lambda \in \mathscr{G}_{\irr}$, the space $\Omega(\lambda)$ is a $\mathfrak{g}'$-module.

\begin{coro}

For every $\lambda \in \mathscr{G}_{\irr}$, the $\mathfrak{g}'$-module $\Omega(\lambda)$ is irreducible.

\end{coro}

\begin{coro}

As a $\mathscr{G} \times \mathfrak{g}'$-module, we have
\begin{equation*}
\S(\V) = \bigoplus\limits_{(\lambda\,, \U_{\lambda}) \in \mathscr{G}_{\irr}} \U_{\lambda} \otimes \Omega(\lambda)\,,
\end{equation*}
where $\Omega(\lambda)$ is an irreducible $\mathfrak{g}'$-module.

\label{CorollarySpo2n1}

\end{coro}

\begin{theo}

Let $(\mathscr{G}\,, \mathscr{G}') = (\normalfont{\textbf{SpO}}(2n|1)\,, \normalfont{\textbf{OSp}}(2k|2l)) \subseteq \normalfont{\textbf{SpO}}(\mathbb{C}^{2k|2l} \otimes \mathbb{C}^{2n|1})$. Then
\begin{equation*}
\S(\mathbb{C}^{k|l} \otimes \mathbb{C}^{2n|1}) = \bigoplus\limits_{\lambda} \lambda \otimes \theta(\lambda)
\end{equation*}
where $\lambda$ is a finite-dimensional irreducible $\mathscr{G}$-module and $\theta(\lambda)$ is an irreducible $\mathfrak{g}'$-module.
\label{thm:SPO2n1}
\end{theo}

\begin{proof}

Let $\U = \mathbb{C}^{2n|1}$, $\W = \mathbb{C}^{2k|2l}$ and $\E = \W \otimes \U = \mathbb{C}^{2k|2l} \otimes_{\mathbb{C}} \mathbb{C}^{2n|1}$. The space $\mathbb{C}^{k|l}$ is a maximal isotropic subspace of $\mathbb{C}^{2k|2l}$ and then $\V = \mathbb{C}^{k|l} \otimes \mathbb{C}^{2n|1}$ is a maximal isotropic subspace of $\E$.

As explained in \cite[Corollary~2.33]{CHENGWANG}, $\textbf{SpO}(2n|1)$ satisfies the conditions (a) and (b) in the beginning of this section. In particular, the theorem follows from Corollary \ref{CorollarySpo2n1}.
\end{proof}

\appendix

\section{Semisimple superalgebras and Super Wedderburn Theorem}

In this appendix, unless stated otherwise we assume that $\mathbb K$ is an arbitrary field. 
All of the superalgebras are  finite dimensional and over $\mathbb{K}$ and all of the ideals are $\mathbb Z_2$-graded.
Some of the proofs of this  appendix rely on the work of  Wall \cite{WALL}. For a more elaborate treatment of central simple graded algebras see \cite[Chap. IV]{LAM}. 
The following definition is due to C.T.C. Wall~\cite{WALL}.
\begin{defn}

A $\mathbb{K}$-superalgebra $\mathscr{A}$ is said to be simple if it does not contain any nontrivial ideals. We say that $\mathscr{A}$ is central if $\mathscr{Z}(\mathscr{A}) \cap \mathscr{A}_{\bar{0}} = \mathbb{K}$, where $\mathscr{Z}(\mathscr{A})$ denotes  the  center of $\mathscr{A}$ considered as an ungraded algebra. That is, $\mathscr Z(\mathscr A):=\{a\in\mathscr R\,:\,ab=ba\ \text{ for all }b\in\mathscr A\}$. 
\label{dfn-frstApx}
\end{defn}

\begin{theo}[Super Wedderburn Theorem]
The following statements are equivalent:
\begin{enumerate}
\item Every $\mathscr{A}$-module is semisimple\,.
\item The left regular module $\mathscr{A}$ is a direct sum of minimal left ideals\,.
\item $\mathscr{A}$ is a direct product of simple superalgebras\,.
\end{enumerate}
\label{SuperWedderburnTheorem}
\end{theo}

A proof of this theorem can be found in \cite[Proposition~2.4]{JOZEFIAK}.

\begin{defn}

A superalgebra $\mathscr{A}$ is called semisimple if it satisfies one of the three equivalent conditions of Theorem \ref{SuperWedderburnTheorem}.

\end{defn}

Let $\mathscr{A}$ and $\mathscr{B}$ be two  simple $\mathbb{K}$-superalgebras.
Our goal is to prove that the tensor product $\mathscr{A} \otimes_{\mathbb{K}} \mathscr{B}$ is semisimple. 
Set $\Z_{\mathscr{A}} := \mathscr{Z}(\mathscr{A})\cap\mathscr A_{\bar{0}}$ and $\Z_{\mathscr{B}} := \mathscr{Z}(\mathscr{B})\cap\mathscr B_{\bar{0}}$. 
\begin{lemme}

$\Z_{\mathscr{A}}$ is a field.

\label{LemmaCenterField}

\end{lemme}

\begin{proof}

One can easily see that $\Z_{\mathscr{A}} \subseteq \mathscr{Z}(\mathscr{A})\cap  \mathscr{Z}(\mathscr{A}_{\bar{0}})$. As explained in \cite[Lemma~3]{WALL}, either $\mathscr{A}$ is simple (as an ungraded algebra) or $\mathscr{A}_{\bar{0}}$ is simple. In any case, the center of a finite-dimensional simple algebra is a field (see \cite[Lemma~2.32]{KNAPP}). Moreover, for every non-zero element $a \in \Z_{\mathscr{A}}$, the element $a^{-1}$ is in $\Z_{\mathscr{A}}$, and the lemma follows.
\end{proof}

Note that we have the following equality:
\begin{equation}
\mathscr{A} \otimes_{\mathbb{K}} \mathscr{B} = \left(\mathscr{A} \otimes_{\Z_{\mathscr{A}}} \Z_{\mathscr{A}}\right) \otimes_{\mathbb{K}} \mathscr{B} = \mathscr{A} \otimes_{\Z_{\mathscr{A}}} \left(\Z_{\mathscr{A}} \otimes_{\mathbb{K}} \mathscr{B}\right)\,,
\label{DecompositionOfA}
\end{equation}
and $\mathscr{A}$ is $\Z_{\mathscr{A}}$-central.

\begin{rema}

Let $\mathscr{A} = \mathscr{A}_{\bar{0}} \oplus \mathscr{A}_{\bar{1}}$ be a superalgebra which is semisimple as an ungraded algebra. Then it is also a semisimple superalgebra.
To show this, we prove that $\mathscr{A}$ is  a direct product of simple graded ideals. Since $\mathscr{A}$ is ungraded semisimple, we can express it as a  product of simple (ungraded) ideals, i.e.,
\begin{equation*}
\mathscr{A}=\mathscr{A}^1\times \cdots \times \mathscr{A}^r\,.
\end{equation*}
If $r=1$ then $\mathscr{A}$ is simple ungraded, hence simple graded, and we are done. If $r>1$, 
choose a simple graded ideal $\mathscr{I} = \mathscr{I}_{\bar{0}} \oplus \mathscr{I}_{\bar{1}}$ of $\mathscr{A}$. Since $\mathscr{I}$ is also an ungraded ideal, it should be  a product of some of the $\mathscr{A}^i$. For example, $\mathscr{I}=\mathscr{A}^1\times \cdots\times \mathscr{A}^s$ for some $s\leq r$. If $s=r$ then $\mathscr{A}$ is a simple superalgebra and we are done. If $s<r$, then 
note that
$\mathscr{I}':=\mathscr{A}^{s+1}\times \cdots\times \mathscr{A}^r$ is the 
annihilator of $\mathscr{I}$ in $\mathscr{A}$, i.e.,
\begin{equation*}
\mathscr{I}'=\left\{a\in \mathscr{A}\,:\,a\mathscr  I=0\right\}.
\end{equation*}
It is now straightforward to check that $\mathscr{I}'$ is a graded ideal of $\mathscr{A}$. From $\mathscr{A}=\mathscr{I}\oplus \mathscr{I}'$ and induction on $r$ it follows that $\mathscr{A}$ is a semisimple superalgebra. 
\label{rema-A.5}
\end{rema}
The first part of the following lemma is due to Wall \cite{WALL}. 
\begin{lemme}
\begin{enumerate}
\item Let $\mathscr{C} = \mathscr{C}_{\bar{0}} \oplus \mathscr{C}_{\bar{1}}$ be a $\mathbb{K}$-superalgebra such that $\mathscr{C}_{\bar{0}}$ is a simple $\mathbb{K}$-algebra and $\mathscr{C}_{\bar{1}} = \mathscr{C}_{\bar{0}} u$ for some $u \in \mathscr{Z}(\mathscr{C})_{\bar{1}}$ satisfying $u^{2} = 1$. Then $\mathscr{C}$ is a simple $\mathbb{K}$-superalgebra.
\item Let $\mathscr{C}$ be a simple $\mathbb{K}$-superalgebra and let $\mathbb L$ be a field that contains $\mathbb K$. Then 
$\mathbb{L} \otimes_{\mathbb{K}} \mathscr{C}$ is a semisimple $\mathbb{L}$-superalgebra. Moreover, if $\mathscr{C}$ is $\mathbb{K}$-central simple, $\mathbb{L} \otimes_{\mathbb{K}} \mathscr{C}$ is $\mathbb{L}$-central simple.
\end{enumerate}
\label{LemmaExtensionField}

\end{lemme}

\begin{proof}

\begin{enumerate}
\item Let $\mathscr{I} = \mathscr{I}_{\bar{0}} \oplus \mathscr{I}_{\bar{1}}$ be a graded ideal of $\mathscr{C}$. By assumption $\mathscr{C}_{\bar{0}}$ is simple, hence either $\mathscr{I}_{\bar{0}} = \{0\}$ or $\mathscr{I}_{\bar{0}} = \mathscr{C}_{\bar{0}}$. In the former case, for any $c \in \mathscr{I}_{\bar{1}}$, we have $cu \in \mathscr{I}_{\bar{0}} = \{0\}$, hence $cu = 0$, i.e. $c = cu^{2} = 0$, and $\mathscr{I} = \{0\}$. In the latter case, we have $\mathscr{C}_{\bar{1}} = \mathscr{C}_{\bar{0}} u = \mathscr{I}_{\bar{0}} u \subseteq \mathscr{I}_{\bar{1}}$, i.e. $\mathscr{I} = \mathscr{C}$.
\item As explained in \cite[Lemma~3]{WALL}, either $\mathscr{C}$ is simple (ungraded) or $\mathscr{C}_{\bar{1}} = \mathscr{C}_{\bar{0}}u$ for $u \in \mathscr{Z}(\mathscr{C})_{\bar{1}}$ with $u^{2} = 1$. If $\mathscr{C}$ is simple (ungraded) then from classical theory of associative algebras we know that $\mathbb{L}\otimes_{\mathbb{K}} \mathscr{C}$ is semisimple (as an ungraded algebra), hence
by Lemma \ref{rema-A.5}
 it is also semisimple as a superalgebra. 
 
If $\mathscr{C}_{\bar{0}}$ is simple and $\mathscr{C}_{\bar{1}}=\mathscr{C}_{\bar{0}}u$ for $u$ central satisfying $u^2=1$, then it follows from \cite[Proposition~2.33]{KNAPP} that $\mathbb{L} \otimes_{\mathbb{K}} \mathscr{C}_{\bar{0}}$ is semisimple, hence 
\begin{equation*}
\mathbb L\otimes_{\mathbb K}\mathscr{C}_{\bar{0}} \cong \mathscr{C}^{(1)}\times\cdots\times \mathscr{C}^{(r)},
\end{equation*}
where the $\mathscr{C}^{(i)}$ are simple algebras. Clearly $\mathscr{C}^{(i)}\oplus \mathscr{C}^{(i)}u$ is an ideal of $\mathbb L\otimes_\mathbb K\mathscr{C}$. 
Now let $e_i\in \mathscr{C}^{(i)}$ denote the identity element of the simple algebra $\mathscr{C}^{(i)}$, and set $z_i:=e_iu$. Then $z_i\in \mathscr{C}^{(i)}u$, and indeed it is a central element of $\mathscr{C}^{(i)}\oplus \mathscr{C}^{(i)}u$. Furthermore, $z_i^2=(e_iu)^2=e_i^2u^2=e_i$. 
Now set 
$\tilde{\mathscr C}^{(i)}:=\mathscr{C}^{(i)}\oplus \mathscr{C}^{(i)}u$. Then 
$\tilde{\mathscr C}^{(i)}:=
\mathscr{C}^{(i)}\oplus \mathscr{C}^{(i)}z_i$
and thus from the first part of the Lemma it follows that $\tilde{\mathscr C}^{(i)}$ is a simple ideal of $\mathscr{C}$. 
Clearly $\mathbb L\otimes_\mathbb K\mathscr C\cong\tilde{\mathscr C}^{(1)}\times \cdots\times\tilde{\mathscr C}^{(r)}$, hence from
Theorem \ref{SuperWedderburnTheorem}
it follows that
$\mathbb{L} \otimes_{\mathbb{K}} \mathscr{C}$ is semisimple. 

Now, we assume that $\mathscr{C}$ is a central simple $\mathbb{K}$-superalgebra. It follows that $\mathscr{Z}(\mathbb{L} \otimes_{\mathbb{K}} \mathscr{C})_{\bar{0}} = \mathbb{L}$. As explained previously, $\mathbb{L} \otimes_{\mathbb{K}} \mathscr{C}$ is semisimple. According to Theorem \ref{SuperWedderburnTheorem}, it follows that $\mathbb{L} \otimes_{\mathbb{K}} \mathscr{C} = \mathscr{D}_{1} \times \ldots \times \mathscr{D}_{y}$, where the $\mathscr{D}_{i}$ are simple $\mathbb{K}$-superalgebras. Note that 
\begin{equation*}
\mathscr{Z}(\mathscr{D}_{1})_{\bar{0}} \times \ldots \otimes \mathscr{Z}(\mathscr{D}_{y})_{\bar{0}} \subseteq \mathscr{Z}(\mathbb{L} \times_{\mathbb{K}} \mathscr{C})_{\bar{0}} = \mathbb{L}\,,
\end{equation*}
and it follows that $y = 1$, i.e. $\mathbb{L} \otimes_{\mathbb{K}} \mathscr{C}$ is simple.
\qedhere\end{enumerate}
\end{proof}

\begin{prop}
\begin{enumerate}
\item The tensor product of two finite-dimensional semisimple superalgebras is a semisimple superalgebra.
\item A quotient of a semisimple superalgebra is semisimple.
\end{enumerate}
\label{PropositionSemisimpleSuperalgebras}
\end{prop}

\begin{proof}

(1) It is enough to prove that the tensor product of two simple $\mathbb{K}$-superalgebras is semisimple. Using Equation \eqref{DecompositionOfA}, we obtain
\begin{equation*}
\mathscr{A} \otimes_{\mathbb{K}} \mathscr{B} = \mathscr{A} \otimes_{\Z_{\mathscr{A}}} \left(\Z_{\mathscr{A}} \otimes_{\mathbb{K}} \mathscr{B}\right),
\end{equation*}
and it follows from Lemma \ref{LemmaExtensionField} that $\Z_{\mathscr{A}} \otimes_{\mathbb{K}} \mathscr{B}$ is semisimple. In particular, it is enough to prove that for every central simple $\mathbb{K}$-superalgebra $\mathscr{A}$ and every simple $\mathbb{K}$-superalgebra $\mathscr{B}$, $\mathscr{A} \otimes_{\mathbb{K}} \mathscr{B}$ is simple. 
Let $\mathscr{A}$ and $\mathscr{B}$ be such superalgebras. Then
\begin{equation*}
\mathscr{A} \otimes_{\mathbb{K}} \mathscr{B} = \mathscr{A} \otimes_{\mathbb{K}} \left(\Z_{\mathscr{B}} \otimes_{\Z_{\mathscr{B}}} \mathscr{B}\right) = \left(\mathscr{A} \otimes_{\mathbb{K}} \Z_{\mathscr{B}}\right) \otimes_{\Z_{\mathscr{B}}} \mathscr{B}\,,
\end{equation*}
where $\Z_{\mathscr{B}} = \mathscr{Z}(\mathscr{B})_{\bar{0}}$. From Lemma \ref{LemmaCenterField}, we know that $\Z_{\mathscr{B}}$ is a field and from Lemma \ref{LemmaExtensionField}, it follows that both $\mathscr{A}\otimes_{\Z_\mathscr B}\Z_{\mathscr B}$ and $\mathscr{B}$ are central simple $\Z_{\mathscr{B}}$-superalgebras. In particular, from \cite[Theorem~2]{WALL} it follows that $\left(\mathscr{A} \otimes_{\mathbb{K}} \Z_{\mathscr{B}}\right) \otimes_{\Z_{\mathscr{B}}} \mathscr{B}$ is simple, and thus Proposition \ref{PropositionSemisimpleSuperalgebras} follows.

(2) Follows immediately from Theorem~\ref{SuperWedderburnTheorem}.
\end{proof}

\begin{rema}

In \cite{WALL}, Wall gave an explicit classification of simple superalgebras over any field. In particular, from the results of \cite{WALL} it follows that every semisimple
superalgebra $\mathscr{A}$ over $\mathbb{C}$ is of the form
\begin{equation}
\mathscr{A} \cong \bigoplus_{i=1}^{n} \Mat(a_{i}| b_{i}, \mathbb{C}) \oplus \bigoplus_{j=1}^{m} \mathrm{Q}(c_{i}, \mathbb{C})\,,
\label{DecompositionAComplex}
\end{equation}
with $n, m, a_{i}, b_{i}, c_{i} \in \mathbb{Z}^{+}$. Here  $\Mat(a_i|b_i,\mathbb C)$ and  $\mathrm Q(c_i,\mathbb C)$ denote the  \emph{associative} superalgebras that, when equipped with their  commutator brackets, yield the Lie superalgebras $\mathfrak{gl}(a_i|b_i,\mathbb C)$ and $\mathfrak{q}(c_i,\mathbb C)$.  Moreover, for every $a, b, c, d \in \mathbb{Z}^{+}$,
we have \begin{enumerate}
\item $\mathrm{Q}(a, \mathbb{C}) \otimes_{\mathbb{C}} \mathrm{Q}(b, \mathbb{C}) \cong \Mat(ab|ab, \mathbb{C})$\,,
\item $\mathrm{Q}(a, \mathbb{C}) \otimes_{\mathbb{C}} \Mat(b|c, \mathbb{C}) \cong \mathrm{Q}(a(b+c), \mathbb{C})$\,,
\item $\Mat(a|b, \mathbb{C}) \otimes_{\mathbb{C}} \Mat(c|d, \mathbb{C}) \cong \Mat(ac+bd|ad+bc, \mathbb{C})$\,.
\end{enumerate}

A detailed proof of the isomorphism \eqref{DecompositionAComplex} can be found in \cite[Theorem~3.1]{CHENGWANG}. Similarly, 
every semismimple superalgebra $\mathscr{A}$ over $\mathbb{R}$ is of the form
\begin{equation*}
\mathscr{A} \cong \bigoplus_{i=1}^{n} \Mat(a_{i}| b_{i}, \mathbb{S}_{i}) \oplus \bigoplus_{j=1}^{m} \mathrm{Q}(c_{j}, \mathbb{X}_{j}) \,,
\end{equation*}
where $\mathbb{S}_{i}, \mathbb{X}_{j} = \mathbb{R}, \mathbb{C}$ or $\mathbb{H}$ and with $n, m, k, a_{i}, b_{i}, c_{i},  \in \mathbb{Z}^{+}$. 
\label{PropositionRealCase}

\end{rema}

We finish this section with the following Lemma. We include a proof because we did not find a reference. 

\begin{lemme}
\label{LemmaAppendixB1}
Let $\V$ be a finite dimensional $\mathbb Z_2$-graded vector space and let $\mathscr{A} \subseteq \End(\V)$ be a subsuperalgebra. Assume that $\V$ is a semisimple $\mathscr{A}$-module. Then $\mathscr{A}$ is a semisimple superalgebra.
\end{lemme}

\begin{proof}
Set $d:=\dim(\V)$ and $\mathscr{B}:=\End(\V)$. Consider $\mathscr{B}$ as a left module over itself, with respect to left multiplication. Then
$\mathscr{B} \cong \V^{\oplus d}$ as $\mathscr{B}$-modules.
In fact the summands on the right hand side can be taken to be left ideals $\mathscr{B}e_{i}$ where $e_{1},\ldots,e_{d}$ is a system of idempotents in $\mathscr{B}$ satisfying $e_{i}e_{j}=\delta_{i\,,j}e_{i}$ and $\sum_{i=1}^d e_{i}=1$. For example if we realize $\End(\V)$ as $\End(\mathbb{K}^{m|n})$ then we can set $e_{i}:=\E_{i,i}$ for $1\leq i\leq m+n$.

Since $\mathscr{A} \subseteq \mathscr{B}$ is a subalgebra, it follows that $\mathscr{B} \cong \V^{\oplus d}$ as $\mathscr{A}$-modules. Since $\V$ is semisimple, it follows that $\mathscr{B}$ is a semisimple $\mathscr{A}$-module (with respect to left multiplication).  

Now consider $\mathscr{A}$ as an $\mathscr{A}$-module under left multiplication. Then $\mathscr{A}$ is an $\mathscr{A}$-submodule of $\mathscr{B}$, and hence it is a semisimple module (since $\mathscr{B}$ is also semisimple). In particular, it follows from Theorem \ref{SuperWedderburnTheorem} that $\mathscr{A}$ is a semisimple superalgebra. 
\end{proof}

\section{Superhermitian forms on $\U$ and superinvolutions on $\End_{\mathbb{D}}(\U)$}

\label{AppendixSuperInvolution}

Let $\mathbb{D}$ be a division superalgebra over a field $\mathbb{K}$ and $\iota$ be a ($\mathbb K$-linear) superinvolution on $\mathbb{D}$. 
\begin{rema}

As explained in Section \ref{SectionDivisionSuperalgebra}, if $\mathbb K=\mathbb R$ then  $\mathbb{D}$ can be  either $\mathbb{R}, \mathbb{C}, \mathbb{H}$ or $\mathbb{C} \oplus \mathbb{C} \cdot \varepsilon$ with $\varepsilon^{2} = 1$ and $i\varepsilon = \varepsilon i$.
Also, when $\mathbb K=\mathbb C$ we must have $\mathbb D=\mathbb D_{\bar 0}=\mathbb C$.
\end{rema}
Let $\U = \mathbb{D}^{n|m}$ be the right-$\mathbb{D}$-module given by
\begin{equation*} 
\U := \begin{cases} \mathbb{D}^{n|m} & \text{ if } \mathbb{D}_{\bar{1}} = \{0\}, \\ \mathbb{D}^{k} & \text{ if } \mathbb{D}_{\bar{1}} \neq \{0\}. \end{cases}
\end{equation*}

Our goal in this appendix is to  explain the correspondence between superhermitian forms on $\U$ and superinvolutions on $\End_{\mathbb{D}}(\U)$.

We first need to introduce some notation. We denote by $\U^{*}:=\Hom_{\mathbb{D}}(\U\,, \mathbb{D})$ the dual of $\U$, i.e. the set of additive maps $f: \U \to \mathbb{D}$ satisfying $f(u\cdot \D)=f(u)\cdot \D$ for $u \in \U$, $\D \in \mathbb{D}$. Note that $\U^{*}$ is a right $\mathbb{D}$-module: for $f \in \U^{*}$ and $\D \in \mathbb{D}$ we set
\begin{equation*}
(f\cdot \D)(u) := (-1)^{\left|\D\right| \cdot \left|f\right|} \iota(\D)f(u)\,, \qquad (u \in \U, \D \in \mathbb{D})\,.
\end{equation*}

We denote by $\langle \cdot\,,\cdot\rangle: \U^*\times \U\to \mathbb{D}$ the pairing given by $\langle f, u\rangle:=f(u)$ for $f\in \U^*$ and $u \in \U$. We define $\langle u\,, f\rangle := (-1)^{\left|f\right|\cdot \left|u\right|}\langle f\,, u\rangle$.  

\begin{rema}
\begin{enumerate}
\item To any homogeneous $( \iota,\epsilon)$-superhermitian form $\gamma$, we can associate a $\mathbb{D}$-module homomorphism $\Psi^{\gamma}: \U \to \U^{*}$:
\begin{equation*}
u \mapsto u^{*}\,, \qquad \text{where} \qquad \langle u^{*}\,, v \rangle:=\gamma(u\,, v)\,, v \in \U\,.
\end{equation*}
Note that $\left|\gamma\right| = \left|\Psi^{\gamma}\right|$. The superhermitian form $\gamma$ is non-degenerate if and only if the corresponding map $\Psi^{\gamma}: \U \to \U^{*}$ is a $\mathbb{D}$-module ismomorphism. 
\item There is an even canonical isomorphism of $\mathbb{D}$-modules $\eta: \U \to \U^{**}$, given by $u \mapsto \eta(u) := u^{**}$ where
\begin{equation}
u^{**}(f):=(-1)^{\left|u\right|\cdot\left|f\right|}\iota(f(u))\,, \qquad (f\in \U^{*}\,, u \in \U)
\label{November13}
\end{equation}
such that the right action of $\mathbb{D}$ on $\U^{**}$ is given by
\begin{equation*}
u^{**}\cdot \D(f) := (-1)^{\left|\D\right|\cdot\left|u\right|} \iota(\D) u^{**}(f)\,, \qquad (u \in \U, f \in \U^{*}, \D \in \mathbb{D})\,.
\end{equation*}
One can see that the map given in \eqref{November13} is well-defined. Indeed, for every $u \in \U, f \in \U^{*}$ and $\D \in \mathbb{D}$, we get
\begin{align*}
 u^{**}(f\cdot \D) &= (-1)^{\left|u\right|\cdot \left|f \cdot \D\right|}\iota\left((f\cdot \D)(u)\right) = (-1)^{\left|u\right|\cdot \left|f \cdot \D\right|} (-1)^{\left|\D\right|\cdot \left|f\right|}\iota(\iota(\D)f(u)) \\ 
     & =  (-1)^{\left|u\right|\cdot \left|f \cdot \D\right|} (-1)^{\left|\D\right|\cdot \left|f\right|} (-1)^{\left|\D\right| \cdot \left|f(u)\right|} \iota(f(u))\D\\
     & = (-1)^{\left|u\right|\cdot \left|f \cdot \D\right|} (-1)^{\left|\D\right|\cdot \left|f\right|} (-1)^{\left|\D\right| \cdot \left|f(u)\right|} (-1)^{\left|f\right| \cdot \left|u\right|} u^{**}(f)\D  
      =  u^{**}(f)\D\,.
\end{align*}
Moreover,
\begin{eqnarray*}
(u\cdot \D)^{**}(f) & = & (-1)^{\left|u\cdot \D\right| \cdot \left|f\right|}\iota(f(u\cdot \D)) = (-1)^{\left|u\cdot \D\right| \cdot \left|f\right|}\iota(f(u)\cdot \D) = (-1)^{\left|u\cdot \D\right| \cdot \left|f\right|}(-1)^{\left|f(u)\right| \cdot \left|\D\right|}\iota(\D)\iota(f(u)) \\
& = & (-1)^{\left|u\right| \cdot \left|f\right|}(-1)^{\left|\D\right| \cdot \left|u\right|}\iota(\D) \iota(f(u))\,,
\end{eqnarray*}
and
\begin{equation*}
(u^{**}\cdot \D)(f) = (-1)^{\left|\D\right| \cdot \left|u\right|}\iota(\D)u^{**}(f) = (-1)^{\left|u\right| \cdot \left|f\right|}(-1)^{\left|\D\right| \cdot \left|u\right|}\iota(\D) \iota(f(u))\,,
\end{equation*}
i.e. the map $u \to u^{**}$ is $\mathbb{D}$-linear.

\item Given a homogeneous $\mathbb{D}$-linear map $\T: \V \to \W$ between two right $\mathbb{D}$-modules, we define $\T^{*}: \W^{*} \to \V^{*}$ by
\begin{equation*}
\T^{*}f(v):=(-1)^{\left|\T\right|\cdot \left|f\right|}f(\T(v))\,, \qquad (f \in \W^{*}, v \in \V)\,.
\end{equation*}
In particular, $\left|\T\right| = \left|\T^{*}\right|$. The map $\T^{*}$ is $\mathbb{D}$-linear. Indeed, for every $u \in \U, f \in \U^{*}$ and $\D \in \mathbb{D}$,
\begin{equation*}
\T^{*}(f\cdot \D)(u) = (-1)^{\left|\T\right|\cdot \left|f \cdot \D\right|}(f \cdot \D)(\T(u))=(-1)^{\left|\T\right| \cdot \left|f \cdot \D\right|}(-1)^{\left|\D\right| \cdot \left|f\right|}\iota(\D)f(\T(u))
\end{equation*}
and 
\begin{equation*}
\T^{*}f \cdot \D(u) = (-1)^{\left|\D\right| \cdot \left|T^{*}f\right|} \iota(\D)\T^{*}f(u) = (-1)^{\left|\T\right| \cdot \left|f\right|}(-1)^{\left|\D\right| \cdot \left|\T^{*}f\right|}\iota(\D)f(\T(u))\,.
\end{equation*}
i.e. $\T^{*}(f\cdot \D)(u) = (\T^{*}f \cdot \D)(u)$. Moreover, for $\T: \U \to \U, \S: \U \to \U, f \in \U^{*}$ and $u \in \U$, we get:
\begin{equation*}
(\S\T)^{*}f(v)=(-1)^{(\left|\S\right|+\left|\T\right|)\cdot \left|f\right|}f(\S\T(v)) 
\end{equation*}
and
\begin{equation*}
\T^{*}\S^{*}f(v)=(-1)^{(\left|\S\right|+\left|f\right|)\cdot \left|\T\right|}\S^{*}f(\T(v)) = (-1)^{(\left|\S\right|+\left|f\right|)\cdot \left|\T\right|+ \left|\S\right|\cdot \left|f\right|}f(\S(\T(v)))\,,
\end{equation*}
i.e. $(\S\T)^{*}=(-1)^{\left|\S\right| \cdot \left|\T\right|}\T^{*}\S^{*}$.
\item Let $\T:\U\to \U^{*}$ be a $\mathbb{D}$-linear map. Identifying $\U^{**}$ and $\U$ via $\eta$, we can think of $\T^{*}$ as a map $\U \to \U^{*}$. For $v, w\in \U$, we have
\begin{equation*}
\T^{*}v(w) = \T^{*}v^{**}(w) =(-1)^{\left|\T\right| \cdot \left|v\right|}v^{**}(\T(w)) = (-1)^{\left|\T\right| \cdot \left|v\right| + (\left|\T\right|+\left|w\right|)\cdot \left|v\right|}\tau(\T(w)(v))=(-1)^{\left|v\right| \cdot \left|w\right|}\iota(\langle \T(w)\,, v\rangle)\,.
\end{equation*}
Thus $\T^{*}$ is uniquely defined by the relation
\begin{equation*}
\langle \T^{*}(v)\,,w\rangle = (-1)^{\left|v\right| \cdot \left|w\right|}\iota\left(\langle \T(w)\,, v\rangle\right)\,.
\end{equation*}
Similarly, let $\T: \U^{*} \to \U$ be $\mathbb{D}$-linear. Then we can think of $\T^{*}$ as a map $\U^{*} \to \U$. Let $f, g \in \U^{*}$, and assume that $\T^{*}f = u^{**} = u$ for $u\in \U$. This implies  $\left|u\right| = \left|\T\right| + \left|f\right|$. Now $\left|g\right| \neq \left|\T\right| + \left|f\right|$ implies $\T^{*}f(g) = 0$ and $\left|u\right| \neq \left|g\right|$ implies $u^{**}(g) = 0$. But in view of $\left|u\right| = \left|\T\right| + \left|f\right|$, we have $\left|g\right| = \left|\T\right| + \left|f\right|$ if and only if $\left|u\right| = \left|g\right|$. It follows that when at least one of  $\T^{*}f(g)$ and $u^{**}(g)$ are nonzero, we have 
\begin{equation*}
\T^{*}f(g) = (-1)^{\left|\T\right| \cdot \left|f\right|}f(\T(g)) = (-1)^{\left|\T\right| \cdot \left|f\right| }\langle f\,, \T(g)\rangle =(-1)^{\left|f\right| \cdot \left|g\right|}\langle \T(g)\,, f\rangle\,,
\end{equation*}
and
\begin{equation*}
u^{**}(g) = (-1)^{\left|g\right| \cdot \left|u\right|}\iota(g(u)) = (-1)^{\left|g\right| \cdot \left|u\right|}\iota(\langle g\,, \T^{*}(f)\rangle) = \iota(\langle \T^{*}f\,, g\rangle)\,.
\end{equation*}
From $\T^{*}f(g) = u^{**}(g)$, it follows that 
\begin{equation*}
\langle \T(g)\,, f\rangle = (-1)^{\left|f\right| \cdot \left|g\right|}\iota(\langle \T^{*}f\,, g\rangle)\,,
\end{equation*}
and in fact the latter relation uniquely determines $\T^*: \U^{*} \to \U$.

\item Finally, let $\T: \U \to \U$ be a $\mathbb{D}$-linear map. 
Then
\begin{equation*}
\T^{**}= \eta \circ \T \circ \eta^{-1}\,.
\end{equation*}
This follows from 
\begin{align*}
(\T^{**}v^{**})(f) & = (-1)^{\left|\T\right| \cdot \left|v\right|}v^{**}(\T^{*}(f))=(-1)^{\left|\T\right| \cdot \left|v\right| + \left|v\right| \cdot (\left|\T\right|+\left|f\right|)}\iota(\T^{*}(f)(v))\\
& =
(-1)^{\left|\T\right| \cdot \left|v\right| + \left|v\right|\cdot (\left|\T\right| + \left|f\right|) + \left|\T\right| \cdot \left|f\right|}\iota(f(\T(v))) = (-1)^{\left|\T\right| \cdot \left|f\right| + \left|v\right| \cdot \left|f\right|}\iota(f(\T(v)))\,,
\end{align*}
and 
\begin{align*}
(\T(v))^{**}(f) = (-1)^{(\left|\T\right| + \left|v\right|)\cdot f}\iota(f(\T(v)))
=(-1)^{\left|\T\right| \cdot \left|f\right| + \left|v\right| \cdot \left|f\right|}\iota(f(\T(v)))\,.
\end{align*}
\end{enumerate}

\label{remark:B1}
\end{rema}

Let $\gamma$ be a non-degenerate, even or odd, $(\iota, \varepsilon)$-superhermitian on $\U$. Since $\gamma$ is non-degenerate, for every $\T \in \End_{\mathbb{D}}(\U)$ there exists a unique map $\T^{\natural} \in \End_{\mathbb{D}}(\U)$ such that
\begin{equation}
\gamma(\T^{\natural}(w)\,, v) = (-1)^{\left|\T\right| \cdot \left|w\right|}\gamma(w\,, \T(v))\,.
\label{FromGammaUToEndDU}
\end{equation}

\begin{lemme}

The endomorphism $\T^{\natural}$ is $\mathbb{D}$-linear. Moreover, the map $\T \to \T^{\natural}$ is $\mathbb{D}$-linear and satisfies
\begin{equation*}
\T^{\natural \natural} = \T\,, \qquad (\S\T)^{\natural} = (-1)^{\left|\S\right|\cdot \left|\T\right|} \T^{\natural}\S^{\natural}\,, \qquad (\S\,, \T \in \End_{\mathbb{D}}(\U))\,.
\end{equation*}

\end{lemme}

\begin{proof}
Let $\T, \S \in \End_{\mathbb{D}}(\U)$, $u, v \in \U$ and $\D \in \mathbb{D}$. First,
\begin{eqnarray*}
\gamma(\T^{\natural}(u\cdot \D)\,, v) & = & (-1)^{\left|\T\right|\cdot \left|u\cdot \D\right|} \gamma(u\cdot \D\,, \T(v)) = (-1)^{\left|\T\right|\cdot \left|u\cdot \D\right|} (-1)^{\left|\D\right|\cdot \left|u\right| + \left|\D\right| \cdot \left|\gamma\right|} \iota(\D)\gamma(u\,, \T(v)) \\
    & = & (-1)^{\left|\T\right|\cdot \left|u\right| + \left|\T\right|\cdot \left|\D\right| + \left|\D\right|\cdot \left|u\right| + \left|\D\right|\cdot \left|\gamma\right| + \left|\T\right|\cdot \left|u\right|} \iota(\D)\gamma(\T^{\natural}(u)\,, v) \\
    & = & (-1)^{\left|\T\right|\cdot \left|u\right| + \left|\T\right|\cdot \left|\D\right| + \left|\D\right|\cdot \left|u\right| + \left|\D\right|\cdot \left|\gamma\right| + \left|\T\right|\cdot \left|u\right| + \left|\D\right| \cdot \left|\gamma\right| + \left|\D\right| \cdot \left|u\right| + \left|\T\right| \cdot \left|\D\right|} \gamma(\T^{\natural}(u)\cdot \D\,, v) \\
    & = & \gamma(\T^{\natural}(u) \cdot \D\,, v)
    \end{eqnarray*}
i.e. $\T^{\natural}$ is $\mathbb{D}$-linear. Moreover
\begin{equation*}
\gamma(\T^{\natural} \S^{\natural}(u)\,, v) = (-1)^{\left|\T\right| \cdot (\left|\S\right|+\left|u\right|)}\gamma(\S^{\natural}u\,, \T(v)) = (-1)^{\left|\T\right| \cdot (\left|\S\right| + \left|u\right|) + \left|\S\right| \cdot \left|u\right|} \gamma(u\,, \S(\T(v)))\,,
\end{equation*}
and
\begin{equation*}
\gamma((\S\T)^{\natural}(u)\,, v) = (-1)^{(\left|\S\right| + \left|\T\right|)\cdot \left|u\right|}\gamma(u\,, \S(\T(v)))\,,
\end{equation*}
from which it follows that 
\begin{equation*}
(\S\T)^{\natural} =(-1)^{\left|\S\right| \cdot \left|\T\right|}\T^{\natural} \S^{\natural}\,.
\end{equation*}
Finally, $\T^{\natural\natural}=\T$ because 
\begin{eqnarray*}
\gamma(\T^{\natural\natural}u\,, v)&
=&(-1)^{\left|\T\right| \cdot \left|u\right|}\gamma(u\,, \T^\natural(v))=
(-1)^{\left|\T\right| \cdot \left|u\right| + (\left|\T\right| + \left|v\right|)\cdot \left|u\right|}\varepsilon\iota\big(\gamma(\T^{\natural}(v)\,, u)\big)\\
&=&
(-1)^{\left|\T\right| \cdot \left|u\right| + (\left|\T\right| + \left|v\right|) \cdot \left|u\right| + \left|\T\right| \cdot \left|v\right|}
\varepsilon\iota\big(\gamma(v\,,\T(w))\big)\\
&=&
(-1)^{\left|\T\right| \cdot \left|u\right| + (\left|\T\right| + \left|v\right|)\cdot \left|u\right| + \left|\T\right| \cdot \left|v\right| + \left|v\right| \cdot (\left|\T\right| + \left|u\right|)}
\varepsilon\iota(\varepsilon)\gamma(\T(u)\,, v) =  \gamma(\T(u)\,, v)\,.\ \ \ \ \ 
\qedhere\end{eqnarray*}
\end{proof}

\begin{rema}

The following properties are straightforward:
\begin{itemize}
\item[(1)] $(z\Id_{\U})^\natural=(-1)^{\left|z\right|\cdot \left|\gamma\right|} \iota(z)\Id_{\U}$ for $z \in \mathscr{Z}(\mathbb{D})$\,. Indeed,
\begin{equation*}
\gamma((z\Id_{\U})^{\natural}u\,, v) = (-1)^{\left|z\Id_{\U}\right| \cdot \left|u\right|} \gamma(u\,, v \cdot z) = (-1)^{\left|z\right| \cdot \left|u\right|+ \left|z\right|\cdot \left|\gamma\right| + \left|z\right|\cdot \left|u\right|} \gamma(u\cdot \iota(z)\,, v) = (-1)^{\left|z\right| \cdot \left|\gamma\right|} \gamma(\iota(z)\Id_{\U}(u)\,, v)\,,
\end{equation*}
i.e. $(z\Id_{\U})^{\natural} = (-1)^{\left|z\right| \cdot \left|\gamma\right|} \iota(z)\Id_{\U}$.
\item[(2)] Recall that $\Psi^\gamma : \U \to \U^{*}$ is the $\mathbb{D}$-linear map corresponding to $\gamma$. Then \[
\T^{\natural} = 
(-1)^{|\Psi^\gamma|\cdot |\T|}
(\Psi^\gamma)^{-1}\circ \T^{*} \circ \Psi^\gamma.
\]Indeed for $v \in \U$, we set $v^{*} := \Psi^\gamma(v)$, so that $\langle v^{*}\,, w\rangle = \gamma(v\,,w)$ for $w \in \U$. Now 
\begin{align*}
\langle \T^{*}v^{*}\,, x\rangle = (-1)^{\left|\T\right| \cdot \left|v^*\right|}\langle v^{*}\,, \T(x)\rangle 
&
= (-1)^{\left|\T\right| \cdot (\left|\Psi^\gamma\right|+\left|v\right|)}\gamma(v\,, \T(x))\\
&= 
(-1)^{|\Psi^\gamma|\cdot |\T|}
\gamma(\T^{\natural}(v)\,, x)=
(-1)^{|\Psi^\gamma|\cdot |\T|}
\langle (\T^{\natural}v)^{*}, x\rangle\,,
\end{align*}
so that $\T^{*}v^{*}=
(-1)^{|\Psi^\gamma|\cdot |\T|}(\T^{\natural} v)^{*}$ and thus $(\Psi^\gamma)^{-1}(\T^{*}v^{*})=
(-1)^{|\Psi^\gamma|\cdot |\T|}
\T^{\natural}v$. 
\end{itemize}

\label{RemarkNovember14AH}

\end{rema}

\bigskip

The correspondence ofsuperhermitian forms and superinvolutions relies on the graded version of Skolem-Noether theorem which is proved in  \cite[Proposition~8]{ELDUQUEVILLA}. To make our exposition self-contained, we recall this theorem in the case of central simple superalgebras of the form $\End_{\mathbb D}(\U)$.

\begin{rema}
Let $\theta_1,\theta_2$ be $\mathbb K$-linear superinvolutions of $\End_{\mathbb D}(\U)$ such that 
$\theta_1(\X)=(-1)^{|\P|\cdot |\X|}\P\theta_2(\X)\P^{-1}$ for a homogeneous invertible 
$\P\in\End_{\mathbb D}(\U)$. Also, let
$\gamma_k:\U\times \U\to\mathbb  D$ for $k=1,2$ be $(\iota,\epsilon_k)$-superhermitian forms on $\U$ such that  \eqref{FromGammaUToEndDU}
 holds for $\gamma=\gamma_k$ and $\theta=\theta_k$. 
From the relations 
\[
\theta_k(\T)=(-1)^{|\Psi^{\gamma_k}|\cdot|\T|}(\Psi^{\gamma_1})^{-1}\circ \T^*\circ \Psi^{\gamma_k} 
\]
it follows that $(-1)^{|\X|\cdot |\T|}\X\T=\T\X$ for $\X:=(\Psi^{\gamma_2})^{-1}\circ \Psi^{\gamma_1}\circ \P$ and all $\T\in\End_{\mathbb D}(\U)$, i.e.,  
$\X\in \tilde{\mathscr Z}(\End_{\mathbb D}(\U))$.
Since $\tilde{\mathscr Z}(\End_{\mathbb D}(\U))$ is purely even, it follows that $|\X|=\bar 0$. Thus if $\gamma_1$ and $\gamma_2$ have the same parity then $\P$ is even, and if $\gamma_1$ and $\gamma_2$ have opposite parities then $\P$ is odd. 
\label{Remark.B5}
\end{rema}

\begin{prop}
\label{prp:Skolem}
Let $\mathbb D$, $\mathbb K$, and $\U$ be as above. Assume that $\mathrm{char}\,\mathbb K=0$ and  $\tilde{\mathscr Z}(\mathbb D)=\mathbb K$. Furthermore, let  $\Theta$ be a $\mathbb K$-linear automorphism of $\End_{\mathbb D}(\U)$. Then there exists an invertible homogeneous $\P\in\End_{\mathbb D}(\U)$
such that $\Theta(\X)=(-1)^{|\P|\cdot|\X| }\P\X\P^{-1}$ for all $\X\in\End_{\mathbb D}(\U)$.
\end{prop}

In order to describe the correspondence of  $(\iota, \varepsilon)$-superhermitian forms  on $\U$ and superinvolutions on $\End_{\mathbb{D}}(\U)$  \eqref{FromGammaUToEndDU}we distinguish two different cases: $\mathbb D_{\bar 1}=0$ and $\mathbb D_{\bar 1}\neq 0$.

\begin{prop}

Let $\mathbb{D} = \mathbb{D}_{\bar{0}}$ be a division superalgebra over $\mathbb{K}$ where $\mathbb K=\mathbb R$ or $\mathbb C$. Also, let  $\iota$ be a  (super)involution on $\mathbb{D}$. Let $\T \mapsto \T^{\diamond}$ be a $\mathbb{K}$-linear superinvolution of $\End_{\mathbb{D}}(\U)$ such that $(z\Id_{\U})^{\diamond} = \iota(z)\Id_{\U}$ for $z\in \mathscr{Z}(\mathbb{D})$. Then the following assertions hold:
\begin{itemize}
\item[\rm (i)] There exists a homogeneous 
$(\iota, \epsilon)$-superhermitian form $\gamma^{\diamond}$ on $\U$ satisfying
\begin{equation}
\label{EquationDiagramMov132}
\gamma^{\diamond}(\T^{\diamond}(u)\,, v) = (-1)^{\left|\T\right| \cdot \left|u\right|}\gamma^{\diamond}(u\,, \T(v))\,, \qquad (u\,, v \in \U)\,.
\end{equation}
\item[\rm (ii)]
If $\gamma'$ is a
$(\iota,\epsilon')$-superhermitian form that satisfies \eqref{EquationDiagramMov132}
then $\gamma'=c\gamma^\diamond$ for some $c\in\mathscr Z(\mathbb D)$ such that $\iota(c)=\pm c$.
Moreover, $\epsilon'=\epsilon$ if   $\iota(c)=c$, and $\epsilon'=-\epsilon$ otherwise.

\end{itemize}
\label{PropositionC4}

\end{prop}

\begin{proof}
(i)
Let $\gamma^{\bullet}$ be any even $(\iota, 1)$-hermitian form on $\U$. Let $\T \mapsto \T^{\bullet}$ denote the involution associated to $\gamma^{\bullet}$ so that
\begin{equation*}
\gamma^{\bullet}(\T^{\bullet}(u)\,, v) = (-1)^{\left|\T\right| \cdot \left|u\right|}\gamma^{\bullet}(u\,, \T(v))\,, \qquad (u\,, v \in \U)\,.
\end{equation*}

It is not difficult to verify that the map $\T \mapsto (\T^{\bullet})^{\diamond}$ is an automorphism of $\End_{\mathbb{D}}(\U)$ which is $\mathscr{Z}(\mathbb{D})$-linear. Note that $\mathscr{Z}(\mathbb{D})$ is isomorphic to the centre of $\End_{\mathbb{D}}(\U)$. 
Thus by Proposition 
\ref{prp:Skolem} this automorphism should be inner, i.e.,  there exists $\P \in \End_{\mathbb{D}}(\U)$ such that
\begin{equation}
(\T^{\bullet})^{\diamond} = (-1)^{\left|\P\right| \cdot \left|\T\right|} \P^{-1}\T\P\,,
\label{TBulletDiamond1}
\end{equation}
that is $\T^{\diamond} = (-1)^{\left|\P\right| \cdot \left|\T\right|}\P^{-1} \T^{\bullet} \P$. The relation $\T^{\bullet\diamond} = \P^{-1}\T\P$ implies that
\begin{eqnarray}
\T = \T^{\diamond\diamond} & = & (\T^{\diamond})^{\diamond} = (-1)^{\left|\P\right| \cdot \left|\T\right|}(\P^{-1}\T^{\bullet}\P)^{\diamond} = \P^{-1} (\P^{-1}\T^{\bullet}\P)^{\bullet} \P \nonumber \\
                   & = & (-1)^{\left|\P\right|\cdot \left|\P^{-1}\T^{\bullet}\right|} \P^{-1} \P^{\bullet} (\P^{-1}\T^{\bullet})^{\bullet} \P \nonumber \\
                   & = & (-1)^{\left|\P\right|\cdot \left|\P^{-1}\T^{\bullet}\right|} (-1)^{\left|\P^{-1}\right|\cdot \left|\T^{\bullet}\right|} \P^{-1} \P^{\bullet} \T^{\bullet\bullet} {\P^{-1}}^{\bullet} \P \nonumber \\
                   & = & (-1)^{\left|\P\right|\cdot \left|\P^{-1}\T^{\bullet}\right|} (-1)^{\left|\P^{-1}\right|\cdot \left|\T^{\bullet}\right|} (-1)^{\left|\P\right|}\P^{-1} \P^{\bullet} \T^{\bullet\bullet} {\P^{-1}}^{\bullet} \P \label{EquationDecember81}\\
                   & = & (\P^{-1} \P^{\bullet}) \T (\P^{-1} \P^{\bullet})^{-1}\,, \nonumber
\end{eqnarray}
where the sign $(-1)^{\left|\P\right|}$ in Equation \eqref{EquationDecember81} comes from the fact that ${\P^{-1}}^{\bullet} = (-1)^{\left|\P\right|} {\P^{\bullet}}^{-1}$. In particular, $\P^{-1}\P^{\bullet} \in \mathscr{Z}(\mathbb{D})$ and consequently   $\P^{\bullet} = c\P$ where $c \in \mathscr{Z}(\mathbb{D})$. By using the same method, there exists $c_{1} \in \mathscr{Z}(\mathbb{D})$ such that $\P^{\diamond} = c_{1}\P$ and it follows from Equation \eqref{TBulletDiamond1} that $c = c_{1}$. After normalizing $\P$ by an element of $\mathscr Z(\mathbb D)$ we can  assume that $c \in \{\pm 1\}$.

Let $\gamma^{\diamond}$ be the form given by
\begin{equation*}
\gamma^{\diamond}(u\,, v) = \gamma^{\bullet}(\P(u)\,, v)\,.
\end{equation*}

Note that $\left|\gamma^{\diamond}\right| = \left|\P\right|$. For every homogeneous $\D \in \mathbb{D}$, $u, v \in \U$, we get:
Moreover,
\begin{equation*}
\gamma^{\diamond}(u \cdot \D\,, v) = \gamma^{\bullet}(\P(u \cdot \D)\,, v) = \gamma^{\bullet}(\P(u) \cdot \D\,, v) =  \iota(\D)\gamma^{\bullet}(\P(u)\,, v) = \iota(\D)\gamma^{\diamond}(u\,, v)
\end{equation*}
and
\begin{equation*}
\gamma^{\diamond}(u\,, v \cdot \D) = \gamma^{\bullet}(\P(u)\,, v \cdot \D) = \gamma^{\bullet}(\P(u)\,, v)\D = \gamma^{\diamond}(u\,, v)\D\,,
\end{equation*}
i.e. the form $\gamma^{\diamond}$ is super-sesquilinear. Moreover,
\begin{eqnarray*}
\gamma^{\diamond}(u\,, v) & = & \gamma^{\bullet}(\P(u)\,, v) = (-1)^{\left|\P\right| \cdot \left|u\right|} \gamma^{\bullet}(u\,, \P^{\bullet}v) = (-1)^{\left|\P\right| \cdot \left|u\right|} \gamma^{\bullet}(u\,, c\P(v))\\
& = & c(-1)^{\left|\P\right| \cdot \left|u\right|} (-1)^{\left|\P(v)\right| \cdot \left|u\right|}\gamma^{\bullet}(\P(v)\,, u) = c(-1)^{\left|v\right|\cdot \left|u\right|}\gamma^{\diamond}(v\,, u)
\end{eqnarray*}
so the form $\gamma^{\diamond}$ is $(\iota, c)$-superhermitian. Finally, for every $u, v \in \U$, we get:
\begin{eqnarray*}
\gamma^{\diamond}(\T^\diamond(u)\,, v) & = & \gamma^{\bullet}(\P\T^{\diamond}(u)\,, v) = \gamma^{\bullet}(\P\T^{\diamond}\P^{-1}\P(u)\,,v) \\
         & = & \gamma^{\bullet}(\T^{\bullet}\P(u)\,, v) = (-1)^{\left|\T\right|\cdot \left|\P(u)\right|} \gamma^{\bullet}(\P(u)\,, \T(v)) = (-1)^{\left|\T\right|\cdot \left|\P\right|}(-1)^{\left|\T\right|\cdot \left|u\right|} \gamma^{\diamond}(u\,, \T(v))\,.
\end{eqnarray*}

(ii) Assume that for another hermitian form $\gamma^{\star}(\cdot\,, \cdot)$ the relation \eqref{EquationDiagramMov132} holds. Let $h_{\diamond}$ and $h_{\star}$ be the associated adjoint maps $\U \to \U^{*}$ according to  part (1) of Remark \ref{remark:B1}. 
From Remark \ref{Remark.B5} it follows that $|h_\star|=|h_\diamond|$.
Since 
\begin{equation*}
\T^{\diamond} =(-1)^{|h_\diamond|\cdot |\T|} h^{-1}_{\diamond}\T^{*} h_{\diamond} = 
(-1)^{|h_\star|\cdot |\T|}
h^{-1}_{\star}\T^{*} h_{\star}\,,
\end{equation*}
it follows that $h_{\diamond} = ch_{\star}$ for some $c \in \mathscr{Z}(\mathbb{D})$, hence $\gamma^{\diamond}(\cdot\,, \cdot) = c\gamma^{\star}(\cdot\,,\cdot)$. Finally, note that  the form $c\gamma^{\star}$ is $(\iota,\pm\epsilon)$-superhermitian if and only if $\iota(c)=\pm c$.  
\end{proof}

Let us now consider the case where $\mathbb D_{\bar 1}\neq 0$. 
In this case basically $\mathbb K=\mathbb R$ and there is only one division superalgebra of interest (but there are two possible superinvolutions).
Recall that $\mathscr Z(\mathbb D)=\mathbb D$, hence $\mathscr Z(\End_{\mathbb D}(\U))$ is isomprhic to $\mathbb D$.
Thus the restriction of any superinvolution of
$\End_{\mathbb D}(\U)$  to
the ungraded centre  defines a superinvolution of $\mathbb D$.

\begin{prop}

Let $\mathbb{D} := \mathrm{Cl}_1(\mathbb C)$ and let $\iota$ be a superinvolution on $\mathbb{D}$. 
Let $\theta$ be an $\mathbb{R}$-linear superinvolution of $\End_{\mathbb{D}}(\U)$ such that $\theta(z\Id_{\U}) = \iota(z)\Id_{\U}$ for $z\in \mathscr{Z}(\mathbb{D})=\mathbb D$. Let $\epsilon\in\{\pm 1\}$. Then 
there exist an even $(\iota,\epsilon)$-superhermitian form $\gamma_{1}$ on $\U$, and an odd $(\iota \circ \delta, \epsilon)$-superhermitian form $\gamma_{2}$ on $\U$, satisfying
\begin{equation*}
\gamma_{1}(\theta(\T)u\,, v) = (-1)^{\left|\T\right| \cdot \left|u\right|}\gamma_{1}(u\,, \T(v))\,, \qquad \gamma_{2}(\theta(\T)u\,, v) = (-1)^{\left|\T\right| \cdot \left|u\right|}\gamma_{2}(u\,, \T(v))\,, \qquad (u, v \in \U)\,.
\end{equation*}

\label{PropositionC5}

\end{prop}

\begin{proof}
Let  $\iota_{1}$ and $\iota_{2}$ be defined as in 
\eqref{eq:superinv}. 
 We assume $\iota:=\iota_1$, as the argument for $\iota:=\iota_2$ is similar.
Let $\gamma^{\theta_{1}}$ (resp. $\gamma^{\theta_{2}}$) be the even $(\iota_{1}, 1)$-superhermitian form on $\mathbb{D}^{k}$ (resp. even $(\iota_{2}, 1)$-superhermitian form on $\mathbb{D}^{k}$) as in part (1) of Remark \ref{SuperHermitianForms}. The forms $\gamma^{\theta_{1}}$ and $\gamma^{\theta_{2}}$ define superinvolutions $\theta_{1}$ and  $\theta_{2}$ on $\End_{\mathbb{D}}(\mathbb{D}^{k})$. One can see that
\begin{equation*}
\theta_{1}(\X) = -\iota_{1}(\X)^{t}\,, \qquad \theta_{2}(\X) = -\iota_{2}(\X)^{t}\,, \qquad (\X \in \End_{\mathbb{D}}(\mathbb{D}^{k}))\,,
\end{equation*}
where $\iota_{1}(\X)_{i, j} = \iota_{1}(\X_{i, j})$ and where $\X^{t}$ corresponds to the usual transpose of the $k\times k$ matrix $\X$.

 It follows from Proposition \ref{PropositionDivisionIota} that
\begin{equation*}
\theta(\lambda \X) = \overline{\lambda} \theta(\X)\,, \qquad (\lambda \in \mathbb{C}, \X \in \End_{\mathbb{D}}(\mathbb{D}^{k}))\,.
\end{equation*}
Note that $\theta_{1} = \theta_{2} \circ \delta$, where $\delta(\X) = (-1)^{\left|\X\right|}\X$ for $ \X \in \End_{\mathbb{D}}(\mathbb{D}^{k})$. Moreover, 
\begin{equation*}
\theta_{1}(\lambda \X) = \overline{\lambda} \theta_{1}(\X)\,, \qquad \theta_{2}(\lambda \X) = \overline{\lambda} \theta_{2}(\X)\,, \qquad (\lambda \in \mathbb{C}, \X \in \End_{\mathbb{D}}(\mathbb{D}^{k}))\,.
\end{equation*}
In particular, it follows that $\widetilde{\theta}_{1} = \theta \circ \theta^{-1}_{1}
$ and $
\widetilde{\theta}_{2} = \theta \circ \theta^{-1}_{2}
$
are $\mathbb C$-linear 
automorphisms of 
$\End_{\mathbb{D}}(\mathbb{D}^{k})$. 
But also from $\tilde{\mathscr Z}(\mathbb D)=\mathbb C$ it follows that  $\tilde{\mathscr Z}(\End_\mathbb D(\U)\cong \mathbb C$.
Thus by Proposition \ref{prp:Skolem}
 there exist homogeneous  $\P_{1}, \P_{2} \in \End_{\mathbb{D}}(\mathbb{D}^{k})$ such that
\begin{equation*}
\widetilde{\theta}_{1}(\X) = (-1)^{\left|\P_{1}\right|\cdot \left|\X\right|}\P^{-1}_{1} \X\P_{1}\,, \qquad \widetilde{\theta}_{2}(\X) = (-1)^{\left|\P_{2}\right|\cdot \left|\X\right|}\P^{-1}_{2} \X\P_{2}\,, \qquad (\X \in \End_{\mathbb{D}}(\mathbb{D}^{k}))\,,
\end{equation*}
i.e. $\theta(\X) = (-1)^{\left|\P_{1}\right|\cdot \left|\X\right|} \P^{-1}_{1} \theta_{1}(\X) \P_{1}$ and $\theta(\X) = (-1)^{\left|\P_{2}\right|\cdot \left|\X\right|} \P^{-1}_{2} \theta_{2}(\X) \P_{2}$.

Now set $\X=\X_\circ:=i\varepsilon\mathrm{Id}_\U$ and note that $\X_\circ\in\mathscr Z(\End_{\mathbb D}(\U))$. It follows that 
\[
\varepsilon=\theta(\X_\circ)
=(-1)^{|\P_1|\cdot |\X_\circ|}\P_1^{-1}\theta_1(\X_\circ)\P_1=
(-1)^{|\P_1|\cdot |\X_\circ|}\varepsilon.
\]
Since $|\X_\circ|=\bar 1$, we must have  $|\P_1|=\bar 0$. A similar argument with $\theta_1$ replaced by $\theta_2$ implies that we must have $|\P_2|=\bar 1$.
Moreover, there exist $c_{1}, c_{2} \in \mathbb{C}$ such that $\theta_{k}(\P_{k}) = c_{k}\P_{k}, k = 1, 2$. Indeed, for every $\T \in \End_{\mathbb{D}}(\mathbb{D}^{k})$,
\begin{eqnarray}
\T = \theta^{2}(\T) & = & \theta(\theta(\T)) = (-1)^{\left|\T\right|\cdot\left|\P_{k}\right|} \theta(\P^{-1}_{k}\theta_{k}(\T)\P_{k}) = (-1)^{\left|\T\right|\cdot\left|\P_{k}\right|} (-1)^{\left|\P_{k}\right|\cdot \left|\P^{-1}_{k}\theta_{k}(\T)\P_{k}\right|} \P^{-1}_{k} \theta_{k}(\P^{-1}_{k}\theta_{k}(\T)\P_{k}) \P_{k} \nonumber \\ 
& = & (-1)^{\left|\T\right|\cdot\left|\P_{k}\right|} (-1)^{\left|\P_{k}\right|\cdot \left|\T\right|} (-1)^{\left|\P_{k}\right| \cdot \left|\P^{-1}_{k}\theta_{k}(\T)\right|} \P^{-1}_{k}\theta_{k}(\P_{k}) \theta_{k}(\P^{-1}_{k}\theta_{k}(\T))\P_{k} \nonumber \\
& = & (-1)^{\left|\T\right|\cdot\left|\P_{k}\right|} (-1)^{\left|\P_{k}\right|\cdot \left|\T\right|} (-1)^{\left|\P_{k}\right| \cdot \left|\P^{-1}_{k}\theta_{k}(\T)\right|} (-1)^{\left|\theta_{k}(\T)\right| \cdot \left|\P^{-1}_{k}\right|} \P^{-1}_{k}\theta_{k}(\P_{k}) \theta^{2}_{k}(\T) \theta_{k}(\P^{-1}_{k})\P_{k} \nonumber \\
& = & (-1)^{\left|\T\right|\cdot\left|\P_{k}\right|} (-1)^{\left|\P_{k}\right|\cdot \left|\T\right|} (-1)^{\left|\P_{k}\right| \cdot \left|\P^{-1}_{k}\theta_{k}(\T)\right|} (-1)^{\left|\theta_{k}(\T)\right| \cdot \left|\P^{-1}_{k}\right|} (-1)^{\left|\P_{k}\right|}\P^{-1}_{k}\theta_{k}(\P_{k}) \theta^{2}_{k}(\T) \theta_{k}(\P_{k})^{-1}\P_{k} \label{EquationDecember8} \\
& = & \P^{-1}_{k}\theta_{k}(\P_{k}) \T (\P^{-1}_{k}\theta_{k}(\P_{k}))^{-1}\,. \nonumber
\end{eqnarray}
where the sign $(-1)^{\left|\P_{k}\right|}$ in Equation \eqref{EquationDecember8} comes from the fact that $\theta_{k}(\P^{-1}_{k}) = (-1)^{\left|\P_{k}\right|} \theta_{k}(\P_{k})^{-1}$. In particular, using that $\left|\theta_{k}(\P_{k})\right| = \left|\P^{-1}_{k}\right|$, it follows that $\P^{-1}_{k}\theta_{k}(\P_{k}) \in \mathscr{Z}(\mathbb{D})_{\bar{0}} \cong \mathbb{C}$ and the existence of the constants $c_{k}, k=1, 2$, follows. 
The matrices $\P_{k}$ can be normalized by scalars in $\mathscr Z(\mathbb D)_{\bar 0}=\mathbb C$ so that $c_{k} = 1$.

For $k \in \{1, 2\}$, we define the forms $\gamma_{k}: \mathbb{D}^{k} \times \mathbb{D}^{k} \to \mathbb{D}$ by
\begin{equation*}
\gamma_{k}(u\,, v) := \gamma^{\theta_{k}}(\P_{k}u\,, v)\,, \qquad (u\,, v \in \mathbb{D}^{k})\,.
\end{equation*}
Note that $\left|\gamma_{k}\right| = \left|\P_{k}\right|$ and that the forms $\gamma_{k}$ are $(\iota_{k}, 1)$-superhermitian. Moreover, for every homogeneous $u, v \in \mathbb{D}^{k}$ and $\T \in \End_{\mathbb{D}}(\mathbb{D}^{k})$, we have
\begin{eqnarray*}
\gamma_{k}(\theta(\T)u\,, v) & = & \gamma^{\theta_{k}}(\P_{k}\theta(\T)(u)\,, v) = \gamma^{\theta_{k}}(\P_{k}\theta(\T) \P^{-1}_{k}\P_{k}(u)\,, v) \\
         & = & (-1)^{\left|\P_{k}\right| \cdot \left|\T\right|} \gamma^{\theta_{k}}(\theta_{k}(\T)\P_{k}(u)\,, v) = (-1)^{\left|\P_{k}\right| \cdot \left|\T\right|}(-1)^{\left|\T\right|\cdot \left|\P_{k}(u)\right|} \gamma^{\theta_{k}}(\P_{k}(u)\,, \T(v)) \\
         & = & (-1)^{\left|\T\right|\cdot \left|u\right|} \gamma_{k}(u\,, \T(v))\,,
\end{eqnarray*}
i.e. $\gamma_k$ satisfies  Equation \eqref{EquationDiagramMov132}.
Finally, 
$i\gamma_k$ for $k=1,2$ is $(\iota_k,-1)$-superhermitian
since $\iota_1(i)=\iota_2(i)=-i$,  and also satisfies Equation \eqref{EquationDiagramMov132}.
\end{proof}

 


\section{Explicit realizations of 
$\mathfrak{gl}_\mathbb D(\W)$ and 
$\mathfrak{g}(\W\,, \gamma)$}

\label{AppendixExplicitRealization}

In this appendix we give explicit realizations of the Lie superalgebras $\mathfrak{gl}_\mathbb D(\W)$ and $\mathfrak{g}(\W\,, \gamma)$ for every possible $\mathbb{D}$ and $\gamma$.
 A summary of the constructions of this appendix is provided in Table~\ref{Table:LieSup}, where the first column gives the symbol used for the Lie superalgebra, and the second column indicates the subsection of this appendix where the constuction of the Lie superalgebra can be found.

\begin{table}[ht]

\begin{tabular}{|c|c|}
\hline 
Lie superalgebra & Case\\
\hline $\mathfrak{gl}(n|m\,,\mathbb D)$; 
$\mathbb D=\mathbb R,\mathbb C,\mathbb H$ 
& $\mathrm{I}$\\
\hline 
$\mathfrak q(n\,,\mathbb D)$;
$\mathbb D=\mathbb R,\mathbb C,\mathbb H$
&
$\mathrm{II}$ or $\mathrm{III}$\\
\hline 
$\bar{\mathfrak{q}}(n)$ & $\mathrm{IV}$ or $\mathrm{V}$\\
\hline
$\mathfrak p(n,\mathbb D)$; $\mathbb D=\mathbb R,\mathbb C$ & 
$\mathrm{I'(a)}$\\
\hline

$\bar{\mathfrak{p}}(n)$ & 
$\mathrm{I'(b)}$\\
\hline
$\mathfrak{p}^*(n)$ & 
$\mathrm{I'(c)}$\\
\hline 
$\mathfrak{osp}(p\,,q|m)$; $m$ even &
$\mathrm{II'(a)}$\\
\hline
$\mathfrak{spo}(n|p\,,q)$; $n$ even &
$\mathrm{II'(a)}$\\

\hline 
$\mathfrak{osp}(n|m)$; $m$ even &
$\mathrm{II'(b)}$\\
\hline
$\mathfrak{spo}(n|m)$; $n$ even &
$\mathrm{II'(b)}$\\
\hline

$\mathfrak{u}(p\,, q| r\,, s)$
&
$\mathrm{II'(c)}$
\\

\hline
$\mathfrak{spo}^*(p\,, q|m)$
&
$\mathrm{II'(d)}$\\
\hline
$\mathfrak{osp}^*(n|p\,, q)$
&
$\mathrm{II'(d)}$\\

\hline 
$\mathfrak q(p,q)$ & 
$\mathrm{III'}$\\
\hline

\end{tabular}

\vspace{3mm}

\caption{\label{Table:LieSup} Real and complex Lie superalgebras that occur in dual pairs}
\end{table}

Let $\mathbb{D}$ be a division superalgebra over $\mathbb{K} = \mathbb{R}$ or $\mathbb{C}$. Let $\W$ denote the right $\mathbb D$-module defined by
\begin{equation*}
\W := \begin{cases} \mathbb{D}^{n|m} & \text{ if } \mathbb{D}_{\bar{1}} = \{0\}, \\ \mathbb{D}^{n} & \text{ if } \mathbb{D}_{\bar{1}} \neq \{0\}. \end{cases}
\end{equation*}
As in Notation \ref{NotationsBasisU}, we denote by $\mathfrak{gl}_{\mathbb{D}}(\W)$ the $\mathbb K$-Lie superalgebra formed by $\mathbb{D}$-linear endomorphisms of $\W$. Recall that for  every superinvolution $\iota$ on $\mathbb{D}$ (see Equation \eqref{IotaProperties}) and every  $(\iota\,, \pm 1)$-superhermitian form $\gamma$ on $\W$ (see Definition \ref{DefSuperhermitian}), we denote by $\mathfrak{g}(\W\,, \gamma)$ the $\mathbb K$-subalgebra of $\mathfrak{gl}_{\mathbb{D}}(\W)$ defined by
\begin{equation*}
\mathfrak{g}(\W\,, \gamma) := \left\{\X \in \mathfrak{gl}_{\mathbb{D}}(\W)\,, \gamma(\X(u)\,, v) + (-1)^{\left|\X\right|\cdot\left|u\right|} \gamma(u\,, \X(v)) = 0\,, u\,, v \in \W\right\}.
\end{equation*}


\begin{nota}

 Let $\mathbb{D}$ be a division superalgebra over $\mathbb{R}$. 
Following the notation of \cite{SERGANOVA}, 
 for every 
 $\X\in \mathfrak{gl}(k|k\,, \mathbb{D}_{\bar{0}})$ of the form
 \[\X = \begin{pmatrix} \A & \B \\ \C & \D \end{pmatrix},\] we denote by $\sigma(\X), \Pi(\X)$ and $-\st(\X)$ the elements of $\mathfrak{gl}(k|k\,, \mathbb{D}_{\bar{0}})$ given by
\begin{equation*}
\sigma(\X) = \begin{pmatrix} \overline{\A} & \overline{\B} \\ \overline{\C} & \overline{\D} \end{pmatrix}\,, \qquad \Pi(\X) = \begin{pmatrix} \D & \C \\ \B & \A \end{pmatrix}\,, \qquad -\st(\X) = \begin{pmatrix} -\A^{t} & \C^{t} \\ -\B^{t} & -\D^{t} \end{pmatrix}\,,
\end{equation*}
where $\overline{\A}$ is the standard conjugation if $\mathbb{D}_{\bar{0}} = \mathbb{C}$ or $\mathbb{H}$.
We denote by $\xi_{1}$ the map
\begin{eqnarray*}
\xi_{1}: \mathbb{C} & \mapsto & \Mat(2\,, \mathbb{R}) \quad,\quad
a + ib  \mapsto  \begin{pmatrix} a & b \\ -b & a \end{pmatrix}.
\end{eqnarray*}
This map is a monomorphism and extends to a map 
\begin{equation}
\xi_{t}: \Mat(t\,, \mathbb{C}) \ni \A + i\B \to \begin{pmatrix} \A & \B \\ -\B & \A \end{pmatrix} \in \Mat(2t\,, \mathbb{R})\,, \qquad (\A\,, \B \in \Mat(t\,, \mathbb{R}))\,.
\label{MAPXI2}
\end{equation}
Similarly, let $\mathbb{H}$ be the field of quaternions, i.e. $\mathbb{H} = \mathbb{R} + i\mathbb{R} + j\mathbb{R} + ij \mathbb{R}$ with $ij = -ji$. Let $z = a + ib + jc + ijd = (a + ib) + j(c - id) \in \mathbb{H}$. We denote by $\xi'_{1}$ the map
\begin{eqnarray*}
\xi_{1}': \mathbb{H} & \mapsto & \Mat(2\,, \mathbb{C}) \\
(a + ib) + j(c - id) & \mapsto & \begin{pmatrix} a + ib & -c + id \\ c + id & a - ib \end{pmatrix}\,.
\end{eqnarray*}
This map is a monomorphism and extends to a map
\begin{align}
\label{MAPXI4}
\xi_{t}': \Mat(t\,, \mathbb{H})
&\to\Mat(2t,\mathbb C),\\
(\A + i\B) + j(\C - i\D) &\mapsto \begin{pmatrix} \A + i\B & -\C + i\D \\ \C + i\D & \A - i\B \end{pmatrix} \,, \qquad (\A\,, \B\,, \C\,, \D \in \Mat(t\,, \mathbb{R}))\,.
\notag
\end{align}

\label{NotationsSerganova}

\end{nota}

\noindent\textbf{The Lie superalgebras $\mathfrak{gl}_\mathbb D(\W)$.}
Our next task is to give concrete realizations of Lie superalgebras of the form $\mathfrak{gl}_{\mathbb{D}}(\W)$ for every possible $\mathbb{D}$. Note that examples include the queer Lie superalgebras and an exotic real form of $\mathfrak{gl}(n|n,\mathbb C)$. We will consider 5 cases below (Cases I--V).

\textbf{Case I:} $\mathbb D=\mathbb{D}_{\bar{0}}$. By using a $\mathbb D$-basis $\mathscr{B}$ of $\W$ as in Notation \ref{NotationsBasisU}, we identify $\mathfrak{gl}_{\mathbb{D}}(\W) \cong \mathfrak{gl}(n|m\,, \mathbb{D})$. By using the embeddings \eqref{MAPXI2} and \eqref{MAPXI4}, $\mathfrak{gl}_\mathbb D(\W)$ can be identified with a subalgebra of $\mathfrak{gl}(nd| md\,, \mathbb{R})$, with $d := \dim_{\mathbb{R}}(\mathbb{D})$.

\textbf{Case II:} $\mathbb{D} = \mathbb{D}_{\bar{0}} \oplus \mathbb{D}_{\bar{0}} \cdot \varepsilon$ where $\mathbb D_{\bar 0}\in\{\mathbb R,\mathbb C,\mathbb H\}$ , with $\varepsilon^{2} = 1$ and $\varepsilon$ commuting with $\mathbb{D}_{\bar{0}}$.
 Using a $\mathbb D$-basis $\mathscr{B}_{1}$ of $\W = \mathbb{D}^{n}$ as in Notation \ref{NotationsBasisU}, every $\X \in \mathfrak{gl}_{\mathbb{D}}(\W)$ can be written as $\X = \Y + \Z\cdot\varepsilon$, where $\Y, \Z \in \Mat(n\,, \mathbb{D}_{\bar{0}})$, and where $\Z \cdot \varepsilon$ is a matrix given by
\begin{equation*}
(\Z\cdot\varepsilon)_{i, j} = z_{i, j}\varepsilon\,, \qquad (1 \leq i, j \leq k)\,.
\end{equation*}

Let $\Psi_{1}: \mathbb{D} \to \Mat(1|1, \mathbb{D}_{\bar{0}})$ be the map given by
\begin{equation*}
\Psi_{1}(a + b\varepsilon) = \begin{pmatrix} a & b \\ b & a \end{pmatrix}\,, \qquad (a, b \in \mathbb{D}_{\bar{0}})\,.
\end{equation*}
The map $\Psi_{1}$ can be extended to an embedding of $\mathfrak{gl}_{\mathbb{D}}(\mathbb{D}^{n})$ in $\mathfrak{gl}(n|n,\mathbb D_{\bar 0})$, as follows:
\begin{equation}
\Psi_{n}: \mathfrak{gl}_{\mathbb{D}}(\mathbb{D}^{n}) \ni \A + \B \cdot \varepsilon \to \begin{pmatrix} \A & \B \\ \B & \A \end{pmatrix} \in \mathfrak{gl}(n|n\,, \mathbb{D}_{\bar{0}})\,, \qquad (\A\,, \B \in \Mat(n\,, \mathbb{D}_{\bar{0}})\,.
\label{EmbeddingDAbelian}
\end{equation}

Let $\Omega_{2}$ be the matrix in $\Mat(n|n\,, \mathbb{D}_{\bar{0}})_{\bar{1}}$ given by
\begin{equation*}
\Omega_{2} = \begin{pmatrix} 0 & \Id_{n} \\ -\Id_{n} & 0 \end{pmatrix}\,.
\end{equation*}
A straightforward calculation verifies that 
\begin{equation*}
\Im(\Psi_{n}) = \left\{\X \in \Mat(n|n\,, \mathbb{D}_{\bar{0}})\,, \left[\X, \Omega_{2}\right] = 0\right\},
\end{equation*}
or by using Notation \ref{NotationsSerganova}, we get
\begin{equation}
\Im(\Psi_{k}) = \left\{\X \in \Mat(k|k\,, \mathbb{D}_{\bar{0}})\,, \Pi(\X) = \X\right\}\,.
\label{EquationOmega2}
\end{equation}
We denote the Lie superalgebra defined in
\eqref{EquationOmega2} by $\mathfrak{q}(k\,, \mathbb{D}_{\bar{0}})$. 

\begin{rema}
For $\mathbb D=\mathbb D_{\bar 0}=\mathbb R$ or $\mathbb C$, the notation $\mathfrak q(k,\mathbb D)$ is standard. 
For $k\geq 2$, the algebra 
$\mathfrak q(k,\mathbb H)$ 
has a unique simple subquotient that is  isomorphic  to the Lie superalgebra $\mathfrak{psq}^*(k)$ in Serganova's notation
\cite{SERGANOVA}. 
\end{rema}

$\textbf{Cade III:}$ $\mathbb{D} = \mathbb{D}_{\bar{0}} \oplus \mathbb{D}_{\bar{0}} \cdot \varepsilon$ where $\mathbb D_{\bar 0}\in\{\mathbb R,\mathbb C,\mathbb H\}$ , with $\varepsilon^{2} = -1$ and $\varepsilon$ commuting with $\mathbb{D}_{\bar{0}}$. Again, every  $\X \in \mathfrak{gl}_{\mathbb{D}}(\W)$ can be written as $\X = \Y + \Z\cdot\varepsilon$, where $\Y, \Z \in \Mat(k\,, \mathbb{D}_{\bar{0}})$. Let $\Psi'_{1}: \mathbb{D} \to \Mat(1|1, \mathbb{D}_{\bar{0}})$ be the map given by
\begin{equation*}
\Psi'_{1}(a + b\varepsilon) = \begin{pmatrix} a & b \\ -b & a \end{pmatrix}\,, \qquad (a, b \in \mathbb{D}_{\bar{0}})\,.
\end{equation*}
The map $\Psi'_{1}$ can be extended to 
an embedding of 
$\mathfrak{gl}_{\mathbb{D}}(\mathbb{D}^{n})$ into
$\mathfrak{gl}(n|n,\mathbb D_{\bar 0})$ as follows:
\begin{equation}
\Psi'_{n}: \mathfrak{gl}_{\mathbb{D}}(\mathbb{D}^{n}) \ni \A + \B \cdot \varepsilon \to \begin{pmatrix} \A & \B \\ -\B & \A \end{pmatrix} \in \mathfrak{gl}(n|n\,, \mathbb{D}_{\bar{0}})\,, \qquad (\A\,, \B \in \Mat(n\,, \mathbb{D}_{\bar{0}})\,.
\label{EmbeddingDAbelian2}
\end{equation}

Let $\widetilde{\Omega}_{2}$ be the matrix in $\Mat(n|n\,, \mathbb{D}_{\bar{0}})_{\bar{1}}$ given by
\begin{equation*}
\widetilde{\Omega}_{2} = \begin{pmatrix} 0 & \Id_{n} \\ \Id_{n} & 0 \end{pmatrix}\,.
\end{equation*}
A straightforward calculation verifies that \begin{equation*}
\Im(\Psi'_{n}) = \left\{\X \in \Mat(n|n\,, \mathbb{D}_{\bar{0}})\,, \left[\X, \widetilde{\Omega}_{2}\right] = 0\right\}\,,
\end{equation*}
or by using Notation \ref{NotationsSerganova}, we have
\begin{equation}
\Im(\Psi'_{n}) = \left\{\X \in \Mat(n|n\,, \mathbb{D}_{\bar{0}})\,, \st \circ \Pi\circ\st^{-1}(\X) = \X\right\}\,.
\label{EquationOmega22}
\end{equation}

The Lie superalgebra $\Im(\Psi'_n)$ is also isomorphic to ${\mathfrak{q}}(n\,, \mathbb{D}_{\bar{0}})$ via the map 
\[
\Im(\Psi_n)\to \Im(\Psi'_n)\ ,\ 
\X\mapsto -\st(\X).
\] 


\textbf{Case IV:} $\mathbb{D} = \mathbb{C} \oplus \mathbb{C} \cdot \varepsilon$, with $\varepsilon^{2} = 1$ and $\varepsilon i = -i\varepsilon$: In this case we have \[
(a+ib) \varepsilon = \varepsilon \overline{(a+ib)}.
\] Let $\check\Psi_{1}: \mathbb{D} \to \Mat(1|1\,, \mathbb{C})$ be the map given by
\begin{equation*}
\check\Psi_{1}(a + b\varepsilon) = \begin{pmatrix} a & b \\ \overline{b} & \overline{a} \end{pmatrix}\,, \qquad (a, b \in \mathbb{C})\,.
\end{equation*}
The map $\check\Psi_{1}$ can be extended to an embedding of $\mathfrak{gl}_{\mathbb{D}}(\mathbb{D}^{n})$ into $\mathfrak{gl}(n|n,\mathbb D_{\bar 0})$ as follows:
\begin{equation*}
\check\Psi_{n}: \mathfrak{gl}_{\mathbb{D}}(\mathbb{D}^{n}) \ni \A + \B \cdot \varepsilon \to \begin{pmatrix} \A & \B \\ \overline{\B} & \overline{\A} \end{pmatrix} \in \Mat(n|n\,, \mathbb{C})\,, \qquad (\A\,, \B \in \Mat(n\,, \mathbb{D}_{\bar{0}}))\,.
\end{equation*}

 We denote  the image of $\check\Psi_{n}$ by
$\bar{\mathfrak{q}}(n)$. 
Thus

\begin{equation*}
\bar{\mathfrak{q}}(n) = \left\{\X \in \mathfrak{gl}(n|n\,, \mathbb{C})\,, \begin{pmatrix} \A & \B \\ \overline{\B} & \overline{\A} \end{pmatrix}\,, \A\,, \B \in \Mat(k\,, \mathbb{C})\right\}\,.
\end{equation*}
One can easily check that 
\begin{equation}
\bar{\mathfrak{q}}(n) = \left\{\X \in \Mat(n|n\,, \mathbb{C})\,, \sigma \circ \Pi(\X) = \X\right\}\,.
\end{equation}

\begin{rema}
Note that $\bar{\mathfrak q}(n)$ is a real form of $\mathfrak{gl}(n|n\,,\mathbb C)$. 
For $k\geq 2$ the algebra $\bar{\mathfrak q}(n)$ has a unique simple subquotient that is  isomorphic to the Lie superalgebra $^\circ\mathfrak{pq}(n)$ in Serganova's notation \cite{SERGANOVA}.
\end{rema}

\textbf{Case V:} $\mathbb{D} = \mathbb{C} \oplus \mathbb{C} \cdot \varepsilon$, with $\varepsilon^{2} = -1$ and $\varepsilon i = -i\varepsilon$.
 Let $\check \Psi'_{1}: \mathbb{D} \to \Mat(1|1\,, \mathbb{C})$ be the map given by
\begin{equation*}
\check \Psi'_{1}(a + b\varepsilon) = \begin{pmatrix} a & b \\ -\overline{b} & \overline{a} \end{pmatrix}\,, \qquad (a, b \in \mathbb{C})\,.
\end{equation*}
The map $\check \Psi'_{1}$ can be extended to an embedding of $\mathfrak{gl}_{\mathbb{D}}(\mathbb{D}^{n})$ into
$\mathfrak{gl}(n|n,\mathbb D_{\bar 0})$ as follows:
\begin{equation*}
\check \Psi'_{n}: \mathfrak{gl}_{\mathbb{D}}(\mathbb{D}^{n}) \ni \A + \B \cdot \varepsilon \to \begin{pmatrix} \A & \B \\ -\overline{\B} & \overline{\A} \end{pmatrix} \in \Mat(n|n\,, \mathbb{C})\,, \qquad (\A\,, \B \in \Mat(n\,, \mathbb{D}_{\bar{0}}))\,.
\end{equation*}
The Lie superalgebras 
$\bar{\mathfrak q}(n,\mathbb C)$ and 
$\Im(\check\Psi_n')$ are
isomorphic   via the map $\X\mapsto -\st(\X)$. 

\label{RemarkNovember3}

\bigskip
\noindent\textbf{The Lie superalgebras $\mathfrak g(\W,\gamma)$.}
In the rest of this appendix we assume that   $\iota: \mathbb{D} \to \mathbb{D}$ is a super-involution on $\mathbb{D}$ and $\gamma$ is a non-degenerate, even or odd, $(\iota, 1)$-superhermitian form on $\W$. 
Analogous realizations can be given when $\gamma$ is $(\iota,-1)$-superhermitian (see Remarks~\ref{remark:iota,-1},~\ref{rem:spo(n,pmq)} and~\ref{rem:C6}).
We consider~3~general cases (Cases $\mathrm{I}'$--$\mathrm{III}'$).

\bigskip

\textbf{Case $\mathbf{I'}$:} $\gamma$ odd and $\mathbb{D} = \mathbb{D}_{\bar{0}}$. From non-degeneracy of $\gamma$ it follows that $n = m$.
There are 4 subcases $\mathrm{I'(a)}$--$\mathrm{I'(d)}$.
In all of these subcases, 
the basis $\mathscr{B}$ of $\W$ as in Notation \ref{NotationsBasisU} can be chosen such that
\begin{equation*}
\gamma(e_{i}\,, e_{j}) = 0\,, \qquad \gamma(f_{i}\,, f_{j}) = 0\,, \qquad \gamma(e_{i}\,, f_{j}) = \delta_{i\,, j}\,.
\end{equation*}
In particular, the matrix of $\gamma$ in the basis $\mathscr{B}$ is
\begin{equation*}
\Mat_{\mathscr{B}}(\gamma) = \begin{pmatrix} 0 & \Id_{n} \\ \Id_{n} & 0 \end{pmatrix}\,.
\end{equation*}

\textbf{Subcase $\mathbf{I'(a)}$: 
}
$\mathbb{D} = \mathbb{R}$ or $\mathbb C$ and $\iota$ is the trivial involution. 
 Then 
\begin{equation}
\mathfrak{g}(\W\,, \gamma) \cong \left\{\X \in \mathfrak{gl}(n|n\,, \mathbb{D})\,, \X = \begin{pmatrix} \X_{1} & \X_{2} \\ \X_{3} & -\X^{t}_{1}\end{pmatrix}, \X_{2} = \X^{t}_{2}\,, \X_{3} = -\X^{t}_{3}\right\},
\end{equation}
and by using Notation \ref{NotationsSerganova} we have 
\begin{equation}
\mathfrak{g}(\W\,, \gamma) \cong \left\{\X \in \Mat(n|n\,, \mathbb{D})\,, -\st \circ \Pi(\X) = \X\right\}
.
\label{USPPQR}
\end{equation}
It is standard practice to  denote the Lie superalgebra~\eqref{USPPQR}
by  $\mathfrak{p}(n,\mathbb D)$.

\textbf{Subcase $\mathbf{I'(b)}$:} $\mathbb{D} = \mathbb{C}$ and $\iota(z) = \bar{z}$ for $z \in \mathbb{C}$.



Then
\begin{equation}
\mathfrak{g}(\W\,, \gamma) \cong \left\{\X \in \Mat(n|n\,, \mathbb{C})\,, \X = \begin{pmatrix} \X_{1} & \X_{2} \\ \X_{3} & -\X^{*}_{1}\end{pmatrix}\,, \X_{2} = \X^{*}_{2}\,, \X_{3} = - \X^{*}_{3}\,, \X_{1} \in \Mat(n\,, \mathbb{C})\right\}\,.
\end{equation}
Using Notation \ref{NotationsSerganova} we define 
\begin{equation}
\bar{\mathfrak{p}}(n):=\mathfrak{g}(\W\,, \gamma) \cong \left\{\X \in \Mat(n|n\,, \mathbb{C})\,, -\st \circ \Pi \circ \sigma(\X) = \X\right\}\,.
\label{eq:ubaralg}
\end{equation}
Note that $\bar{\mathfrak{p}}(n)$ is a real form of $\mathfrak{gl}(2n|2n\,,\mathbb C)$. For $n\geq 1$ this algebra has a unique simple subquotient that is isomorphic to the Lie superalgebra $\mathfrak{us\pi}(n)$
in Serganova's notation \cite{SERGANOVA}.

\textbf{Subcase $\mathbf{I'(c)}$:}
$\mathbb{D} = \mathbb{H}$ and $\iota(z)=\bar z$ for $z\in\mathbb H$, is the  quaternionic  conjugation.  Then
\begin{equation*}
\mathfrak{g}(\W\,, \gamma) \cong \left\{\X \in \Mat(n|n\,, \mathbb{H})\,, \X = \begin{pmatrix} \X_{1} & \X_{2} \\ \X_{3} & -\X^{*}_{1}\end{pmatrix}, \X_{2} = \X^{*}_{2}\,, \X_{3} = - \X^{*}_{3}\,, \X_{1} \in \Mat(n\,, \mathbb{H})\right\},
\end{equation*}
Using Notation \ref{NotationsSerganova} we define 
\begin{equation}
\mathfrak p^*(n):=
\mathfrak{g}(\W\,, \gamma) \cong \left\{\X \in \Mat(n|n\,, \mathbb{H})\,, -\st \circ \Pi \circ \sigma(\X) = \X\right\}\,.
\label{eq:p*H}
\end{equation}
Note that $\mathfrak p^*(n)$ is a real form of $\mathfrak p(2n,\mathbb C)$.
\begin{rema}
 All of the constructions in Subcases $\mathrm{I'(a)}$--$\mathrm{(d)}$ have analogues for a $(\iota,-1)$-superhermitian form $\gamma:\W\times \W\to\mathbb D$. We can choose a $\mathbb D$-basis such that 
\begin{equation*}
\Mat_{\mathscr{B}}(\gamma) = \begin{pmatrix} 0 & \Id_{n} \\ -\Id_{n} & 0 \end{pmatrix}\,.
\end{equation*}
Then the Lie superalgebras that we obtain in Subcases 
$\mathrm{I'(a)}$--$\mathrm{(c)}$  are as follows:
\begin{equation}
\left\{\X \in \Mat(n|n\,, \mathbb{R})\,, \X = \begin{pmatrix} \X_{1} & \X_{2} \\ \X_{3} & -\X^{t}_{1}\end{pmatrix}\,, \X_{2} = -\X^{t}_{2}\,, \X_{3} = \X^{t}_{3}\,, \X_{1} \in \Mat(n\,, \mathbb{R})\right\}
\cong
{\mathfrak{p}}(n\,, \mathbb{R})\,,
\label{Eq:A1}
\end{equation}

\begin{equation}
 \left\{\X \in \Mat(n|n\,, \mathbb{C})\,, \X = \begin{pmatrix} \X_{1} & \X_{2} \\ \X_{3} & -\X^{t}_{1}\end{pmatrix}\,, \X_{2} = -\X^{t}_{2}\,, \X_{3} = \X^{t}_{3}\,, \X_{1} \in \Mat(n\,, \mathbb{C}) \right\}
\cong\mathfrak{p}(n,\mathbb C)\,, 
\label{Eq:A2}
\end{equation}

\begin{equation}
\left\{\X \in \Mat(n|n\,, \mathbb{C})\,, \X = \begin{pmatrix} \X_{1} & \X_{2} \\ \X_{3} & -\X^{*}_{1}\end{pmatrix}\,, \X_{2} = -\X^{*}_{2}\,, \X_{3} = \X^{*}_{3}\,, \X_{1} \in \Mat(n\,, \mathbb{C})\right\}
\cong
\bar{\mathfrak p}(n)
\,,
\label{Eq:A3}
\end{equation}

\begin{equation}
 \left\{\X \in \Mat(n|n\,, \mathbb{H})\,, \X = \begin{pmatrix} \X_{1} & \X_{2} \\ \X_{3} & -\X^{*}_{1}\end{pmatrix}\,, \X_{2} = -\X^{*}_{2}\,, \X_{3} = \X^{*}_{3}\,, \X_{1} \in \Mat(n\,, \mathbb{H})\right\}\cong\mathfrak p^*(n)\,.
\label{Eq:A4}
\end{equation}
The isomorphism between each of the 
Lie superalgebras defined on the left hand side of   \eqref{Eq:A1}--\eqref{Eq:A3}
and its counterpart in Subcases $\mathrm{I'(a)}$--$\mathrm{(b)}$ is the map $\X\mapsto-\mathrm{st}(\X)$. For the Lie superalgebra $\eqref{Eq:A4}$ and  Subcase $\mathrm{I'(c)}$, 
the map is $\X\mapsto -\mathrm{st}(\bar\X)$ where $\bar X$ denotes the quaternionic conjugation of $\X$.

\label{remark:iota,-1}
\end{rema}

\textbf{Case $\mathbf{II'}$:}
$\gamma$ even and $\mathbb{D} = \mathbb{D}_{\bar{0}}$. 
Recall that $\gamma$ is assumed to be $(\iota,1)$-superhermitian. There are 4 subcases $\mathrm{II'(a)}$--$\mathrm{II'(d)}$.

\textbf{Subcase $\mathbf{II'(a)}$:}
$\mathbb{D} = \mathbb{R}$ (hence  $\iota$ is the trivial involution).  The restriction of $\gamma$ to $\W_{\bar{0}}$ (resp. $\W_{\bar{1}}$) is non-degenerate and symmetric  (resp. skew-symmetric). In particular, $m = 2e$ is even. The matrix $\mathscr{B}$ of $\W$ as in Notation \ref{NotationsBasisU} can be chosen such that
\begin{equation*}
\Mat_{\mathscr{B}}(\gamma) = \begin{pmatrix} \Id_{p, q} & 0 \\ 0 & \J_{m} \end{pmatrix}\,, 
\qquad
\J_{m} = \begin{pmatrix} 0 & \Id_{e} \\ -\Id_{e} & 0 \end{pmatrix}\,,
\qquad (n = p+q)\,.
\end{equation*}
One can easily see that $\X = \begin{pmatrix} \X_{1} & \X_{2} \\ \X_{3} & \X_{4} \end{pmatrix} \in \mathfrak{g}(\W\,, \gamma)$ if and only if
\begin{equation}
\X^{t}_{1}\Id_{p, q} + \Id_{p, q}\X_{1} = 0\,, \qquad \J_{m}\X_{4} + \X^{t}_{4}\J_{m} = 0\,, \qquad \X^{t}_{3}\J_{m} + \Id_{p, q}\X_{2} = 0\,.
\label{eq:X1X3X4}
\end{equation}
In particular, 
$\mathfrak{g}(\W\,, \gamma)$ is 
the Lie superalgebra that consists of matrices $\X\in\Mat(n|m,\mathbb R)$ of the form  
\[
\X = \begin{pmatrix} \X_{1} & \X_{3} \\ -\Id_{p, q}\X^{t}_{3}\J_{m} & \X_{4} \end{pmatrix},
\]
where $\X_1$ and $X_4$ satisfy the relations \eqref{eq:X1X3X4} and $\X_3\in\Mat(n\times m,\mathbb R)$. In particular, 
$\X_{1} \in \mathfrak{o}(p\,, q\,, \mathbb{R})$ and $ \X_{4} \in \mathfrak{sp}(2e\,, \mathbb{R})$
. This Lie superalgebra is usually denoted by $\mathfrak{osp}(p\,,\,q|m)$.

\begin{rema}
An analogous construction with a $(\iota,-1)$-superhermitian form yields the Lie superalgebra $\mathfrak{spo}(n|p\,,q)$ where $p+q=m$. 

\label{rem:spo(n,pmq)}
\end{rema}

\textbf{Subcase $\mathbf{II'}$(b):}
$\mathbb{D} = \mathbb{C}$ and $\iota(z)=z$ for $z\in \mathbb C$. The restriction of $\gamma$ to $\W_{\bar{0}}$ (resp. $\W_{\bar{1}}$) is non-degenerate and symmetric (resp. skew symmetric). In particular, $m = 2e$ is even. The basis $\mathscr{B}$ of $\W$ as in Notation~\ref{NotationsBasisU} can be chosen such that
\begin{equation}
\Mat_{\mathscr{B}}(\gamma) = \begin{pmatrix} \Id_{n} & 0 \\ 0 & \J_{m} \end{pmatrix}\,, \qquad \J_{m} = \begin{pmatrix} 0 & \Id_{e} \\ -\Id_{e} & 0 \end{pmatrix}\,.
\label{eq:X1X3X3C}
\end{equation}
The Lie superalgebra
$\mathfrak{g}(W\,,\gamma)$ consists of 
matrices 
$\X\in\Mat(n|m\,,\mathbb C)$ of the form 
\[\X = \begin{pmatrix} \X_{1} & \X_{2} \\ \X_{3} & \X_{4} \end{pmatrix} 
\]
that satisfy the relations
\begin{equation*}
\X^{t}_{1} + \X_{1} = 0\,, \qquad \J_{m}\X_{4} + \X^{t}_{4}\J_{m} = 0\,, \qquad \X^{t}_{3}\J_{m} + \X_{2} = 0\,.
\end{equation*}
This is the Lie superalgebra $\mathfrak{osp}(n|m\,,\mathbb C)$. An analogous construction with a $(\iota,-1)$-superhermitian form yields the Lie superalgebra $\mathfrak{spo}(n|m\,,\mathbb C)$. 

In particular,

\textbf{Subcase $\mathbf{II'}$(c):}
$\mathbb{D} = \mathbb{C}$ and $\iota(z)=\bar z$ for $z\in\mathbb C$. The restriction of $\gamma$ to $\W_{\bar{0}}$ (resp. $\W_{\bar{1}}$) is non-degenerate and $(\iota, 1)$-hermitian (resp. $(\iota, -1)$-hermitian). We can choose a basis $\mathscr{B}$ such that 
\begin{equation*}
\Mat_{\mathscr{B}}(\gamma) = \begin{pmatrix} \Id_{p, q} & 0 \\ 0 & i\Id_{r,s} \end{pmatrix}\,, \qquad (n = p+q, m = r+ s)\,.
\end{equation*}
Elements of  $\mathfrak{g}(\W\,, \gamma)$
are matrices $\X\in\Mat(n|m,\mathbb C)$ of the form 
\[\X = \begin{pmatrix} \X_{1} & \X_{2} \\ \X_{3} & \X_{4} \end{pmatrix},
\]
that satisfy the relations
\begin{equation*}
\X^{*}_{1}\Id_{p, q} + \Id_{p, q}\X_{1} = 0\,, \qquad \X^{*}_{4}\Id_{r, s} + \Id_{r, s}\X^{*}_{4} = 0\,, \qquad i\X^{*}_{3}\Id_{r, s} + \Id_{p, q}\X_{2} = 0\,.
\end{equation*}
This Lie superalgebra is usually denoted by  $\mathfrak{u}(p\,, q| r\,, s)$.
Note that an analogous construction with $\gamma$ assumed to be $(\iota,-1)$-superhermitian yields the same Lie superalgebra.

\textbf{Subcase $\mathbf{II'(d)}$:}
 $\mathbb{D} = \mathbb{H}$ and $\iota(z)=\bar z$ for $z\in\mathbb H$.  The restriction of $\gamma$ to $\W_{\bar{0}}$ (resp. $\W_{\bar{1}}$) is non-degenerate and $(\iota, 1)$-hermitian  (resp. $(\iota, -1)$-hermitian). We can choose a homogeneous $\mathbb D$-basis $\mathscr{B}$ of $\W$ such that \begin{equation*}
\Mat_{\mathscr{B}}(\gamma) = \begin{pmatrix} \Id_{p, q} & 0 \\ 0 & i\Id_{m} \end{pmatrix}\,, \qquad (n = p+q)\,.
\end{equation*}
One can easily see that $\X = \begin{pmatrix} \X_{1} & \X_{2} \\ \X_{3} & \X_{4} \end{pmatrix} \in \mathfrak{g}(\W\,, \gamma)$ if and only if
\begin{equation*}
\X^{*}_{1}\Id_{p, q} + \Id_{p, q}\X_{1} = 0\,, \qquad \X_{4}-i\X_4^*i = 0\,, \qquad 
\X_2^*\Id_{p,q}-i\X_3^{}=0\,.
\end{equation*}
In particular, 
\begin{equation*}
\mathfrak{g}(\W\,, \gamma) \cong 
\left\{
\X = \begin{pmatrix} \X_{1} & -\Id_{p, q}\X^{*}_{3}i \\ \X_{3} & \X_{4} \end{pmatrix}
\in\Mat(n|m\,, \mathbb{H})
\,, \X_{1} \in \mathfrak{sp}^{*}(p\,, q)\,, \X_{4} \in \mathfrak{so}^{*}(m)\,, \X_{3} \in \Mat(n \times m\,, \mathbb{H})\right\}\,.
\end{equation*}
The latter Lie superalgebra is usually denoted by $\mathfrak{spo}^*(p\,,q|m)$. 
\begin{rema}
If $\gamma$ is chosen to be $(\iota,-1)$-superhermitian, then a similar constriction yields the Lie superalgebras $\mathfrak{osp}^*(n|p\,,q)$ where $p+q=m$. 
\label{rem:C6}
\end{rema}

\textbf{Case $\mathbf{III'}$:}
$\gamma$ even and $\mathbb{D}_{1} \neq \{0\}$, hence $\W=\mathbb D^n$.
According to Lemma~\ref{LemmaTechnicalDecember} we can assume that $\mathbb D=\mathrm{Cl}_1(\mathbb C)$ and $\iota$ is one of the two superinvolutions $\iota_1$ and $\iota_2$ defined in~\eqref{eq:superinv}. 
In what follows we assume that $\iota=\iota_1$, as the case $\iota=\iota_2$ is analogous. 

\begin{rema}

\begin{enumerate}
\item Let $\mathbb{D} := \mathrm{Cl}_1(\mathbb C)$, that is $\mathbb D:=\mathbb{C} \oplus \mathbb{C} \cdot \varepsilon$ with $\varepsilon^{2} = 1$ and $i\varepsilon = \varepsilon i$. Let $\iota$ be one of the two  superinvolutions $\iota_1$ and $\iota_2$  on $\mathbb{D}$ defined in  \eqref{eq:superinv}. One can easily see that for any integer $0\leq p\le n$ the form $\gamma: \mathbb{D}^{n} \times \mathbb{D}^{n} \to \mathbb{D}$ given by
\begin{equation*}
\gamma(u\,, v) = \sum\limits_{i = 1}^{p} \iota(u_{i})v_{i}
-
\sum\limits_{i = p+1}^{n} \iota(u_{i})v_{i}
\,, \qquad (u = (u_{1}\,, \ldots\,, u_{n})\,, v = (v_{1}\,, \ldots\,, v_{n}) \in \mathbb{D}^{n})\,,
\end{equation*}
is an even $(\iota, 1)$-superhermitian form. Furthermore, we can choose a homogeneous $\mathbb D$-basis $\mathscr B_1$ of $\W$ as in Notation~\ref{NotationsBasisU} such that 
 $\Mat_{\mathscr B_1}(\gamma)=\Id_{p,n-p}$.

\item Let $\mathbb D$ and $\iota$ be as in part (1) above. Let $\D \in \mathbb{D}_{\bar{1}}$ be such that $\iota(\D) = c\D$, with $c \in \{\pm 1\}$, and let $\gamma$ be an even, $(\iota, \epsilon)$-superhermitian form. Let $\tilde{\gamma}: \mathbb{D}^{n} \times \mathbb{D}^{n} \to \mathbb{D}$ be the form given by
\begin{equation*}
\tilde{\gamma}(u\,, v) := (-1)^{\left|u\right|}\gamma(u\,, v \D)\,, \qquad (u\,, v \in \mathbb{D}^{n})\,.
\end{equation*}
For every $u, v \in \mathbb{D}^{n}$ and $\D_{1}, \D_{2} \in \mathbb{D}$, we have
\begin{align*}
\tilde{\gamma}(u \D_{1}, v \D_{2}) 
&
= (-1)^{\left|u\D_{1}\right|}\gamma(u \D_{1}\,, v\D_{2}\D) = (-1)^{\left|u\D_{1}\right|} \gamma(u\D_{1}\,, v\D\D_{2}) \\
& =  (-1)^{\left|u\right|+\left|\D_{1}\right|}(-1)^{\left|\D_{1}\right| \cdot \left|u\right|} \iota(\D_{1}) \gamma(u\,, v\D) \D_{2} = (-1)^{\left|u\right|+\left|\D_{1}\right|}(-1)^{\left|\D_{1}\right| \cdot \left|u\right|} (-1)^{\left|u\right|} \iota(\D_{1}) \tilde{\gamma}(u\,, v) \D_{2} \\
& =  (-1)^{\left|\D_{1}\right|\cdot \left|u\right| + \left|\D_{1}\right|} \iota(\D_{1}) \tilde{\gamma}(u\,, v) \D_{2} = (-1)^{\left|\D_{1}\right|\cdot\left|u\right| + \left|\D_{1}\right|\cdot\left|\tilde{\gamma}\right|} \iota(\D_{1})\tilde{\gamma}(u\,, v)\D_{2}\,,
\end{align*}
and
\begin{eqnarray*}
\tilde{\gamma}(u\,, v) & = & (-1)^{\left|u\right|} \gamma(u\,, v\D) = \epsilon(-1)^{\left|u\right|}(-1)^{\left|u\right| \cdot \left|v\D\right|} \iota(\gamma(v\D\,, u)) = \epsilon(-1)^{\left|u\right|}(-1)^{\left|u\right| \cdot \left|v\D\right|}(-1)^{\left|\D\right|\cdot\left|v\right|} \iota(\iota(\D)\gamma(v\,, u)) \\
& = & c\epsilon(-1)^{\left|u\right|}(-1)^{\left|u\right| \cdot \left|v\D\right|}(-1)^{\left|\D\right|\cdot\left|v\right|} \iota(\gamma(v\,, u\D)) = c\epsilon(-1)^{\left|u\right| \cdot \left|v\D\right|}(-1)^{\left|\D\right|\cdot \left|v\right|} (-1)^{\left|v\right|+\left|u\right|} \iota(\tilde{\gamma}(v\,, u)) \\
& = & c\epsilon(-1)^{\left|u\right|\cdot\left|v\right|} \iota(\tilde{\gamma}(v\,, u))
 \end{eqnarray*}  
i.e. $\tilde{\gamma}$ is an odd $(\iota, c\epsilon)$-superhermitian form on $\mathbb{D}^{n}$.     
\end{enumerate}

\label{SuperHermitianForms}

\end{rema}

 Let $\iota_\circ$ denote the complex conjgation on $\mathbb C=\mathbb D_{\bar 0}$, so that ${\iota_1}_{|_{\mathbb D_{\bar 0}}}=\iota_\circ$. Then  $\gamma_{|_{W_{\bar 0}\times \W_{\bar 0}}}$ is  $(1, \iota_\circ)$-hermitian, hence we can choose a $\mathbb{D}_{\bar{0}}$-basis $\mathscr{B}_{1}$ of $\W_{\bar{0}}$ such that $\Mat_{\mathscr{B}_{1}}(\gamma_{|_{\W_{\bar 0}\times \W_{\bar 0}}}) = \Id_{p, q}$, with $p+q=n$.

Elements of  $\mathfrak{gl}_\mathbb D(\W)$ are $n\times n$ matrices with entries in $\mathbb D$. Any such matrix can be expressed as $\X_1+\varepsilon \X_2$ where $\X_1,\X_2\in \Mat(n,\mathbb C)$. 
Let $u, v \in \W_{\bar 0}=\mathbb{D}^{n}_{\bar{0}}$ and $\X \in \mathfrak{g}(\W\,, \gamma)_{\bar{0}} \subseteq 
\mathfrak{gl}_\mathbb D(\W)_{\bar 0}=
\Mat(n\,, \mathbb{C})$. Then from the relation 
$\gamma(\X(u)\,, v) + \gamma(u\,, \X(v)) = 0 $ we obtain 
\begin{equation*}
 \iota_\circ(v)^{t}\Id_{p\,, q}\X u + \iota_\circ(v)^{t} \iota_\circ(\X)^{t} \Id_{p\,, q} u = 0\,,
\end{equation*}
i.e. $\Id_{p, q}\X + \X^*\Id_{p, q} = 0$. In particular, $\mathfrak{g}(\W, \gamma)_{\bar{0}} \cong \mathfrak{u}(p\,, q)$.

Similarly, let $\Y \in \mathfrak{g}(\W\,, \gamma)_{\bar{1}}$. Then there exists $\Z \in \Mat(n\,, \mathbb{C})$ such that $\Y = \Z \cdot \varepsilon$. 
From the relation
\[
\gamma((\Z\cdot \varepsilon) u,v\cdot \varepsilon)
+\gamma(u,(\Z\cdot \varepsilon)(v\cdot \varepsilon))=0
\]
for $u,v\in\W_{\bar 0}\cong \mathbb C^n$ we obtain 
$\Z^*=-i\Id_{p,q}\,\Z\, \Id_{p,q}$. 
Consequently, 
\begin{equation}
\mathfrak{g}(\W\,, \gamma) \cong \left\{\X + \Y\varepsilon
,\ \X^*=-\Id_{p,q}\,\X\,\Id_{p,q},\,
\Y^*=-i\Id_{p,q}\,\Y\, \Id_{p,q},
\X,\Y\in\Mat(n,\mathbb C)
\right\}\,.
\label{GUGammaAppendix}
\end{equation}
We denote the latter Lie superalgebra, which is a real form of $\mathfrak q(n,\mathbb C)$,  by $\mathfrak{q}(p,q)$. For $p+q\geq 2$ the Lie superlagebra $\mathfrak {q}(p,q)$ has a unique subquotient that is isomorphic to 
the simple Lie superalgebra $\mathfrak{uspq}(p,q)$ in  
Serganova's notation~\cite{SERGANOVA}.

\begin{rema}
In the case where $\mathbb K=\mathbb C$,  constructions of the Lie superalgebras $\mathfrak{gl}_\mathbb D(\W)$ and  $\mathfrak{g}(\W,\gamma)$
are analogous (and indeed much simpler).     In this case, the two algebras of the form $\mathfrak{gl}_\mathbb D(\W)$ are 
$\mathfrak{gl}(n|m\,,\mathbb C)$ and $\mathfrak{q}(n\,,\mathbb C)$. 
Furthermore, by Proposition~\ref{PropositionDivisionIota} the only division superalgebra with a $\mathbb C$-linear superinvolution is $\mathbb C$ (and the superinvolution is trivial).
The resulting Lie superalgebras $\mathfrak{g}(\W,\gamma)$ are 
$\mathfrak{osp}(n|m)$ (or $\mathfrak{spo}(n|m)$) and $\mathfrak{p}(n,\mathbb C)$.
\end{rema}

\end{document}